\documentclass[12pt]{amsart} %{article}
\title[Pure extension of the theta divisor]{Pure extension of the theta divisor over the moduli space of abelian varieties}

\usepackage{amsmath, amssymb, mathrsfs, amsthm}% amsfonts, url,, tikz, wrapfig, newclude, }
\usepackage{mathtools} % for \coloneqq (and presumably other stuff!)
\usepackage{hyperref}
\usepackage{float}
\usepackage[all]{xy}
\usepackage{tikz}
\usetikzlibrary{arrows,matrix,positioning}
\tikzstyle{dmatrix}=[matrix of math nodes,row sep=2.5em, column sep=2.5em,
text height=1.5ex, text depth=0.25ex]
\usepackage{graphicx}

\usepackage{enumitem}

\usepackage[left=3.5cm,top=3.5cm,right=3.5cm]{geometry}

\usepackage{tikz-cd} 

\usepackage[english]{babel}
\usepackage[%
    sorting=anyt,
    maxnames=10,
    sortcites=true,doi=false,
    firstinits=true,hyperref,backend=bibtex]{biblatex}
    \renewbibmacro{in:}{%
      \ifentrytype{article}{}{\printtext{\bibstring{in}\intitlepunct}}}
\usepackage[babel]{csquotes}
\usepackage{guit}
\DeclareFieldFormat
[article,inbook,incollection,inproceedings,patent,thesis,unpublished]
  {title}{{#1\isdot}}
\DeclareFieldFormat{pages}{#1}
\usepackage{cancel}
\AtEveryBibitem{%
\ifentrytype{book}{
    \clearfield{url}%
    \clearfield{pages}
    \clearfield{isbn}
    \clearfield{urldate}%
    \clearfield{review}%
    \clearfield{series}%%
}{}
\ifentrytype{article}{
    \clearfield{url}%
    \clearfield{issn}
    \clearfield{doi}
    \clearfield{isbn}
    \clearfield{urldate}%
    \clearfield{review}%
}{}
\ifentrytype{collection}{
    \clearfield{url}%
    \clearfield{urldate}%
    \clearfield{review}%
    \clearfield{series}%%
}{}
\ifentrytype{incollection}{
    \clearfield{url}%
    \clearfield{urldate}%
    \clearfield{review}%
    \clearfield{series}%%
}{}
}

\newcommand{\on}[1]{\operatorname{#1}}
\newcommand{\bb}[1]{{\mathbb{#1}}}

\newcommand{\ca}[1]{{\mathcal{#1}}}
\newcommand{\cat}[1]{\textbf{#1}}
\newcommand{\op}{\mathsf{op}}

\newcommand{\Span}[1]{\left<#1\right>}

\newcommand{\lra}{\longrightarrow}
\newcommand{\hra}{\hookrightarrow}

\newcommand{\iso}{\stackrel{\sim}{\lra}}

\theoremstyle{definition}
\newtheorem{definition}{Definition}[section]

\theoremstyle{plain}% default
\newtheorem{proposition}[definition]{Proposition}
\newtheorem{lemma}[definition]{Lemma}
\newtheorem{theorem}[definition]{Theorem}
\newtheorem{corollary}[definition]{Corollary}

\newtheorem*{theorem*}{Theorem}

\theoremstyle{remark}

\newtheorem{remark}[definition]{Remark}

\newtheorem{example}[definition]{Example}

\renewcommand{\bar}{\overline}

\thanks{The authors are very grateful for the hospitality of the Mathematisches Forschungsinstitut Oberwolfach, where part of this work was conducted. A.\,M.\,Botero was funded by the Deutsche Forschungsgemeinschaft (DFG, German Research Foundation) – Project-ID 491392403 – TRR 358. J.~I.~Burgos was partially supported by grants
  PID2022-142024NB-I00 and CEX2023-001347-S funded by
  MICIU/AEI/10.13039/501100011033. D.~Holmes was funded by the European Union (ERC, EAGL, 101169685). Views and opinions expressed are however those of the author(s) only and do not necessarily reflect those of the European Union or the European Research Council. Neither the European Union nor the granting authority can be held responsible for them. }
\author[Botero]{Ana Mar\'ia Botero}
\address{Bielefeld University \\ Universit\"atsstra{\ss}e 25  \\ 133615 Bielefeld  \\ Germany }
\email{\href{mailto:abotero@math.uni-bielefeld.de}{abotero@math.uni-bielefeld.de}}

\author[Burgos Gil]{Jos\'e Ignacio Burgos Gil}
\address{ICMAT \\ Nicolás Cabrera 13-15 \\ 28049 Madrid  \\ Spain }
\email{\href{mailto:burgos@icmat.es}{burgos@icmat.es}}

\author[Holmes]{David Holmes}
\address{Leiden University \\ Einsteinweg 55  \\ 2333 CC Leiden  \\ The Netherlands }
\email{\href{mailto:holmesdst@math.leidenuniv.nl}{holmesdst@math.leidenuniv.nl}}

\author[de Jong]{Robin de Jong}
\address{Leiden University \\ Einsteinweg 55  \\ 2333 CC Leiden  \\ The Netherlands }
\email{\href{mailto:rdejong@math.leidenuniv.nl}{rdejong@math.leidenuniv.nl}}

\newcounter{nootje}
\def\Z{\mathbb{Z}}
\def\R{\mathbb{R}}

\def\Q{\mathbb{Q}}

\newcommand{\trop}{\mathsf{trop}}
\newcommand{\ghost}{\overline{\mathsf{M}}}
\newcommand{\M}{\mathsf{M}}
\newcommand{\gp}{\mathsf{gp}}

\newcommand{\Vor}{\mathsf{Vor}}

\renewcommand{\val}{\mathsf{val}}

\newcommand{\LOG}{{\sf{log}}}

\newcommand*\isom{\xrightarrow{\sim}}
\newcommand{\divisor}{\operatorname{div}}
\newcommand{\ord}{\operatorname{ord}}

\newcommand{\Hom}{\operatorname{Hom}}

\newcommand{\Spec}{\operatorname{Spec}}

\newcommand{\E}{\operatorname{E}}

\newcommand{\Sk}{\operatorname{Sk}}

\newcommand{\logcbDiv}{\operatorname{log-C-b-Div}}
\newcommand{\logwbDiv}{\operatorname{log-W-b-Div}}

\newcommand{\sPL}{\operatorname{sPL}}

\newcommand{\lb}[1]{\mathcal{#1}} % line bundle
\newcommand{\llb}[1]{\mathfrak{#1}} % log line bundle
\newcommand{\tlb}[1]{\mathsf{#1}} % tropical line bundle

\def\opn#1#2{\def#1{\operatorname{#2}}}
\opn\Vor{Vor}

\def\qq{\mathbb{Q}}

\def\rr{\mathbb{R}}

\def\zz{\mathbb{Z}}
\def\cc{\mathbb{C}}

\def\E{\mathbb{E}}

\def\ll{\mathcal{L}}

\def\pp{\mathcal{P}}

\def\aa{\mathcal{A}}
\def\oo{\mathcal{O}}

\def\hh{\mathcal{H}}

\def\Im{\mathrm{Im}\,}

\def\id{\mathrm{id}}

\def\an{\mathrm{an}}

\def\eps{\epsilon}

\def\Vor{\mathrm{Vor}}

\def\h{\mathrm{h}}

\def\trop{\mathrm{trop}}

\def\Hdg{\mathrm{Hdg}}

\begin{document}

%\maketitle

\begin{abstract}	
A theta divisor on the universal principally polarised abelian variety can be extended to a compactification either by taking the Zariski closure, or by taking the unique extension which is pure of weight 2. For the latter, following ideas of Yuan and Zhang, we need to pass to the category of adelic- or b-divisors. We show that the two choices of extension differ by a tropicalisation of the Riemann theta function. We prove an extension of Moret-Bailly's ``key formula'' that features the pure weight 2 extension of the theta divisor, and discuss various arithmetic applications, including a ``universal'' formula for the N\'eron--Tate height of a point. A key technical input is the systematic use of the theory of logarithmic abelian varieties due to Kajiwara, Kato, and Nakayama. 
\end{abstract}

\maketitle

\tableofcontents

\section{Introduction}

A \emph{theta divisor} on a principally polarised abelian variety $A$ is a symmetric ample divisor $\Theta$ representing the polarisation, in the sense that the polarisation is equal to the map 
\begin{equation} \label{eq:representing}
A \to A^\vee, \quad a \mapsto \ca O(T_a^*\Theta - \Theta). 
\end{equation}
The set of theta divisors for a fixed polarisation is a torsor under the 2-torsion $A[2]$, and each theta divisor is of \emph{pure weight 2}, by which we mean that, denoting by $[n]\colon A \to A$ the multiplication-by-$n$ map for any integer $n$, we have 
\begin{equation} \label{eq:pure}
\forall n \in \zz \, : \, [n]^*\Theta \sim n^2\Theta. 
\end{equation}
Here $\sim$ denotes linear equivalence.

\subsection{The universal abelian variety}

In this paper we study theta divisors over the moduli space of principally polarised abelian varieties (ppav), and their extendability as pure weight 2 objects over the boundary of a compactification. 

Let $g$ be a positive integer. Writing $A_g$ for the moduli stack over $\zz$ of ppav and writing $A_{g,1} \to A_g$ for the universal abelian variety, a first observation is that $A_{g,1}$ does \emph{not} carry a universal theta divisor (unless $g = 1$). Following Faltings--Chai in their book~\cite[\S IV.7]{fc}, we write $N_{g}$ for the stack of pairs of a principally polarised abelian variety and a theta divisor in the above sense. Then $N_g \to A_g$ is a torsor under $A_{g,1}[2]$, and this torsor is split if and only if $g =1$. 

We write $\pi\colon N_{g,1} \to N_g$ for the universal abelian variety over $N_g$. By construction, it carries a universal theta divisor $\Theta$, and also a universal symmetric rigidified line bundle $L$ representing the polarisation. The divisor $\Theta$ is a relative effective Cartier divisor. We can recover  $L$ from $\Theta$ by the formula 
\begin{equation}\label{eq:6}
L = \lb O(\Theta) \otimes \pi^*e^*\lb O(-\Theta),
\end{equation}
where $e$ is the unit section of $\pi$. Conversely, we can recover $\Theta$ from $L$ by observing that $\pi_*L$ is a line bundle on $N_g$, so locally a choice of generating section determines a divisor on $N_{g,1}$, and this divisor is independent of the choice of generating section. 

The purity property \eqref{eq:pure} of any theta divisor on a single ppav translates into a corresponding property for the universal rigidified line bundle~$L$:
\begin{equation} \label{eqn:purity}
\forall n \in \zz \, : \, [n]^*L = L^{\otimes n^2}.  
\end{equation}

Finally, let $\omega_g$ be the determinant of the Hodge bundle on $N_g$. We have the following ``key formula'' due to Moret-Bailly \cites{mb_fonct, mb} governing the variation of the theta divisor in families: 
\begin{equation} \label{eq:cle}
   L^{\otimes 8} \cong  \oo(8 \, \Theta)  \otimes \pi^* \omega_g^{\otimes -4} . 
   \end{equation}
It is known that the exponent~$8$ can not be improved, see \cite[Remark~I.5.2]{fc}.

\subsection{Extending over a compactification}

Our goal in this paper is to extend the above identities to suitable toroidal compactifications of the moduli stack $N_{g,1}$ over $\zz$. First of all, the determinant of the Hodge bundle $\omega_g$ extends naturally to any toroidal compactification of $N_g$, by pulling back the determinant of the sheaf of relative differentials along the zero section of the universal semiabelian scheme. In the following we write $\bar \omega_g$ for this natural extension of the Hodge bundle.

Throughout our introduction we denote by $\bar N_g$ the so-called Delaunay compactification of $N_g$, and by $\bar N_{g,1}$ the so-called mixed Delaunay--Voronoi compactification of $N_{g,1}$. We refer to Section~\ref{sec:choosing-decomposition} for their constructions, which are very similar to their classical counterparts for $A_g$ and $A_{g,1}$ found in \cite[\S VI.1]{fc}. The map $\pi$ extends uniquely to a toroidal map $\pi\colon \bar N_{g,1} \to \bar N_g$. 

Considering the theta divisor momentarily as a \emph{Weil} divisor on $N_{g,1}$, the simplest way to extend it to $\bar N_{g,1}$ is to take its Zariski closure; we denote this Zariski closure by~$\bar \Theta$.  On the other hand, if we focus on the purity property $[n]^*L = L^{\otimes n^2}$ from~\eqref{eqn:purity}, we are led to a rather different extension.

Indeed, observe that the multiplication-by-$n$ map $[n]\colon N_{g,1} \to N_{g,1}$ only defines a \emph{rational map} $\bar N_{g,1} \dashrightarrow \bar N_{g,1}$. It does extend though to a toroidal morphism from a suitable toroidal blowup:
\begin{equation*}
\begin{tikzcd}
  \bar N_{g,1}'  \arrow[rd, "{[n]}"] \ar[d] &  \\
\bar N_{g,1}\arrow[r, dashed, "{[n]}"]& \bar N_{g,1}.
\end{tikzcd}
\end{equation*}
The rational map $[n]\colon \bar N_{g,1}' \dashrightarrow \bar
N_{g,1}'$ is again not a morphism, but this can be resolved by a
further toroidal blowup; and so on ad infinitum. From this one may get the idea
of trying to extend the rigidified line bundle $L$ as a pure weight 2 object, as well as the
universal theta divisor $\Theta$ as it appears in the key formula, not over a fixed compactification $\bar N_{g,1}$, but rather over a \emph{limit} of a sequence of toroidal blowups of $\bar N_{g,1}$.

For the line bundle $L$, the above idea has been made precise by Yuan--Zhang \cite{yz}, using their theory
of \emph{adelic} divisors and \emph{adelic} line bundles. The resulting objects are not divisors or line bundles in the classical sense, but rather certain limits of sequences of such. Focusing instead on the theta divisor $\Theta$, an alternative framework, which captures the toroidal nature of the situation, is the theory of \emph{toroidal b-divisors}, due to the
first and second named author \cite{botero_burgos}. We upgrade here
this notion to the setting of algebraic stacks with log structure. 

A precise comparison of the two languages is given in Section~\ref{sec:log_b_divisors} (along with an extension of the adelic formalism of Yuan--Zhang to algebraic stacks with log structure), but for the purpose of this introduction we will not distinguish between them. One key observation is that any continuous conical function $f$ on the \emph{tropicalisation} of $\bar N_{g,1}$ (see Section~\ref{sec:proof_main}) induces an adelic or toroidal b-divisor $\on{div} f$, and that one can twist any adelic or  b-line bundle by such divisors~$\on{div} f$.

\subsection{Extending the theta divisor}

Extending the line bundle $L$ in either of these languages as a pure weight~$2$ object can be done in a direct and canonical manner, as we will see in Section~\ref{sec:proof_main}. For the purposes of the introduction we call the resulting pure object $L^{\mathrm{pure}}$. Thus we have
\begin{equation} \label{eq:pure_inv}  \forall n \in \zz \, : \, [n]^*L^{\mathrm{pure}} = L^{\mathrm{pure}, \otimes n^2}.  
\end{equation}
 In the adelic setting, this object is denoted $\bar L$ by Yuan--Zhang. They show the existence and uniqueness of an adelic extension of $L$ satisfying \eqref{eq:pure_inv}  in \cite[Theorem~6.1.3]{yz}. 
 We will offer a different approach using the notion of \emph{log b-line bundles}; this gives the adelic line bundle $\overline{L}$ of Yuan--Zhang an additional toroidal structure.

Extending the theta divisor $\Theta$ as a pure weight 2 object turns out to be more involved. 

To get it in the setting of adelic or toroidal b-divisors, we need the notion of \emph{tropical theta functions}.  Tropical theta functions were introduced by Mikhalkin--Zharkov \cite{tropi-theta} for tropical curves over valuation rings, essentially by tropicalising the power series defining the classical Riemann theta function to obtain a piecewise linear (PL) function on their skeleton. 

In Section~\ref{sec:tropical_RT} we generalise these constructions to include suitable tropical theta characteristics. This will
yield a PL function $\theta^{\mathrm{pl}}$ on a suitable cover of
the tropicalisation of $\bar N_{g,1}$ satisfying a tropical
cocycle condition. We also construct a variant $\theta^{\mathrm{sm}}$ of $\theta^{\mathrm{pl}}$,
which satisfies the same tropical cocycle condition, but is smooth
rather than PL. In essence $\theta^{\mathrm{sm}}$ is constructed by
interpolating the values of $\theta^{\mathrm{pl}}$ on lattice
refinements of $\bar N_{g,1}$. 

Since $\theta^{\mathrm{pl}}$ and
$\theta^{\on{sm}}$ satisfy the same tropical cocycle condition,
their difference 
\[ \theta^{\mathrm{inv}} \coloneqq \theta^{\mathrm{pl}}
- \theta^{\mathrm{sm}} \] is invariant under translations and descends
to give a continuous conical function 
on the tropicalisation of $\bar N_{g,1}$. By what we have said above, to
the continuous conical function  $\theta^{\mathrm{inv}}$ we can canonically
associate an adelic or toroidal b-divisor $\on{div}
  \theta^{\mathrm{inv}}$. 

We then consider the adelic or toroidal b-divisor
\[ \Theta^{\mathrm{pure}} \coloneqq \bar \Theta +  \on{div} \theta^{\mathrm{inv}} . \]
As we will see, the function $\theta^{\mathrm{inv}}$ is non-negative, which implies that one may view $\Theta^{\mathrm{pure}}$ as an \emph{effective} object in the category of adelic or toroidal b-divisors.

With the above notation, the following is our main result.

\begin{theorem*} 
\begin{itemize}
\item[(i)] The multiple $8 \, \bar \Theta$ of the Zariski closure $\bar \Theta$ of the theta divisor canonically has the structure of a Cartier divisor on $\bar N_{g,1}$. 
\item[(ii)] The adelic or b-divisor $8 \, \Theta^{\mathrm{pure}} = 8 \, \left( \bar \Theta +   \on{div} \theta^{\mathrm{inv}} \right)$ is pure of weight 2, in the sense that the associated adelic or b-line bundle $\oo(8 \, \Theta^{\mathrm{pure}})$ is equal to $L^{\mathrm{pure},\otimes 8}$, up to a pullback along the map $\pi$. More precisely, the key formula \eqref{eq:cle} extends as an isomorphism of adelic or b-line bundles
\begin{equation}
L^{\mathrm{pure}, \otimes 8} = \lb O(8 \, \Theta^{\mathrm{pure}}) \otimes \pi^*\bar \omega_g^{\otimes -4}
\end{equation}
over $\bar N_{g,1}$. 
\end{itemize}
\end{theorem*}
In particular, as $\bb Q$-adelic line bundles we have for all $n \in \zz$ that
\[[n]^*\lb O(\Theta^{\mathrm{pure}}) = \lb O(n^2\Theta^{\mathrm{pure}})  \otimes \pi^*\bar \omega_g^{- \frac{n^2-1}{2}}. 
\]

In the case $g=1$, working over $\cc$, and with $\Theta$ given by the unit section, the pure object $8 \, \Theta^{\mathrm{pure}}$ already plays a key role in the work \cite{bkk} by Burgos--Kramer--K\"uhn, where it equals the b-divisor $\on{b-div}(\theta_{1,1}^8)$ up to a pullback from the base. Here $\theta_{1,1}$ is the unique Riemann theta function in degree one whose zero locus is given by the unit section. 

Away from characteristic two, the multiple $8 \, \bar \Theta$ in fact has a \emph{unique} structure of Cartier divisor on $\bar N_{g,1}$. The space $\bar N_{g,1}$ is not normal over characteristic two; it has two connected components over $\mathbb Z[1/2]$, and these meet in characteristic two. As such the map from Cartier divisors to Weil divisors may not be injective.

We will prove our main result in Section~\ref{sec:AV_theta}; see in particular Theorem~\ref{thm:bar_theta_def_equivalence} and Theorem~\ref{thm:adelic_key_restate}.

\subsection{Main tools}

We make systematic use of logarithmic geometry, in particular the theory of logarithmic abelian varieties due to Kajiwara--Kato--Nakayama, which gives  modular interpretations of all toroidal compactifications of moduli spaces of abelian varieties. In order to make our paper reasonably self-contained, we outline in Sections \ref{sec:background_log}--\ref{sec:toroidal_comp} the various definitions and constructions we will use, while trying to give plenty of motivation and examples. In places we mildly generalise known results, for example by working with theta level structures. 

One key technical definition is that of a \emph{log b-line bundle} on
a log scheme $X$. This is a pair $(\llb L, s)$ where $\llb L$ is a
\emph{log line bundle} on $X$ (in the sense of Molcho--Wise \cite{molcho-wise}), 
and $s$ is a \emph{conical section} of the
\emph{tropicalisation} of $\llb L$. In Section~\ref{sec:log_b_divisors} we construct a fully faithful functor from the category of log b-line bundles on $X$ to the category of adelic line bundles on $X$ in the sense of Yuan--Zhang \cite{yz}.

The machinery of log b-line bundles allows us on the one hand to bring in deep theory on log line bundles from the work of Kajiwara--Kato--Nakayama, and on the other hand to make things very explicit by giving precise formulae for the necessary conical functions. The tropical theta function with characteristics $\theta^{\mathrm{pl}}$ appears naturally as an admissible principal polarisation function in the theory of Kajiwara--Kato--Nakayama, and the pure object $L^{\mathrm{pure}}$ appears in our approach as the log b-line bundle $(\llb L,-\theta^{\on{sm}})$ where $\llb L$ is the universal log line bundle, and $\theta^{\on{sm}}$ is the smooth theta function introduced earlier.

Logarithmic base changes of $\bar N_{g,1} \to \bar N_g$ along test curves yield Alexeev--Nakamura models \cite{an} of degenerating abelian varieties. This allows us to bring in ideas from \cite{djs_canonical} using Berkovich analytic spaces to control precisely the orders of vanishing of rational sections.  One technical issue is that the log smooth morphism $\bar N_{g,1} \to \bar N_g$ is not known to be flat for $g \ge 5$, and hence logarithmic base changes need not be schematic base changes. We work around this by introducing a certain root stack $\bar N_g^n$ of $\bar N_{g}$ such that, after logarithmic base change to $\bar N_g^n$, the map $\bar N_{g,1} \to \bar N_g$ becomes flat, and further logarithmic and schematic base changes coincide.

\subsection{Applications} We present two applications of our main result, one of complex analytic, and one of arithmetic nature. First of all, we establish a connection between the normalised tropical theta function with characteristics $\theta^{\on{inv}}$ and the degeneration of the natural hermitian metric on $\oo(\Theta)$ over $\cc$, which is described in terms of a certain normalised version of the classical Riemann theta function with characteristics. See Corollary~\ref{cor:degeneration_theta}. Second, we present a ``universal'' formula for the N\'eron--Tate height of a point on a principally polarised abelian variety over a number field, see Theorem~\ref{univ_NT_formula}. It would be interesting to see to what extent our formula can be used to carry out effective calculations of N\'eron--Tate heights.

\subsection{Overview of the paper} We start in Section~\ref{sec:background_log} with background on logarithmic and tropical geometry. Then in Section~\ref{sec:LAV} we review log abelian varieties and their tropicalisation. In Section~\ref{sec:toroidal_comp} we review classical material on toroidal compactifications, and construct our toroidal compactifications of $N_{g,1}$. In Section~\ref{sec:proof_main} we will then introduce and study the pure extension $L^{\mathrm{pure}}$ of the universal polarising line bundle as a log b-line bundle over a toroidal compactification of~$N_g$. We also discuss here the intimately related pure extension of the Poincar\'e line bundle. We show in Section~\ref{sec:tropical_RT} that the pure extension $L^{\mathrm{pure}}$ can be described in concrete terms using tropical theta functions with characteristics. In Section~\ref{sec:log_b_divisors} we connect the language of log b-line bundles with the language of log b-divisors and log adelic divisors. In Section~\ref{sec:AV_theta} we conduct our study of the pure extension of the theta divisor and prove our main theorem. In Section~\ref{sec:arith_cases} we present our applications. \\

\noindent \textbf{Terminology.} Our notion of convex function will be that from analysis. In particular, for us, a minimum of countably many affine functions will be concave. This is a convention opposite to the one used in for example \cite{fc}. \\

\noindent \textbf{Acknowledgements.} 
  DH would like to thank Leo Herr for discussions around flatness, as well as Aitor Iribar Lopez and Jeremy Feusi for discussions around automorphisms of cone stacks and tropical biextensions. RJ would like to thank Farbod Shokrieh for many inspiring discussions related to the theme of this paper.

\section{Background on logarithmic and tropical geometry} \label{sec:background_log}

\subsection{Log schemes}
For the convenience of the reader we recall in outline some background in log geometry. The basic references are \cite{kato_log_str_of_FI} and \cite{oguslogbook}. 

For us \emph{monoids} are always commutative. A  monoid $P$ is \emph{sharp} if its units $P^\times$ are trivial, it is \emph{integral} if the natural map $P \to P^\gp$ to the associated group is injective, and it is \emph{saturated} if it is integral and if for $x \in P^\gp$ there exists an integer $n >0$ such that $nx \in P$, we have $x \in P$. 

\begin{definition}\label{def:log_scheme}
A log scheme is a pair $(X, \M_X, \alpha)$ where $X$ is a scheme,
$\M_X$ is a sheaf of monoids in the \'etale topology on $X$, and
$\alpha\colon \M_X \to \ca O_X$ is a morphism of sheaves of monoids,
where $\ca O_X$ is given the multiplicative monoid structure. We
require that $\alpha^{-1}\ca O_X^\times \to \ca O_X^\times$ is an
isomorphism (in particular $\M_X^\times=\alpha^{-1}\ca O_X^\times $). The \emph{ghost sheaf} on $X$ is the quotient $\ghost_X =
\M_X /\M_X^\times$.  
\end{definition}

\begin{example}
The \emph{trivial} log structure on $X$ is given by taking $\M_X = \ca O_X^\times$ together with its inclusion into $\ca O_X$. 
\end{example}

\begin{example}\label{eg:log_str_from_open}
Let $X$ be a scheme and $j\colon U \to X$ an open subscheme. Let $\M_X = \ca O_X \cap j_*\ca O_U^\times$ be the subsheaf of functions on $X$ which are units on $U$. Suppose that $U$ is the complement of a reduced simple normal crossings (snc) divisor $D$; then the stalk of $\ghost_X$ at a point is a free monoid of rank equal to the number of branches of $D$ through that point; we say $X$ has \emph{snc log structure}. 
\end{example}

\begin{example}\label{eg:log_str_toric}
Let $P$ be a monoid and $R$ a ring. The associated monoid ring is
$R[P]$, and $\on{Spec} R[P]$ has a natural log structure coming from
the map $P \to R[P]$. More precisely, there is a natural map $i$ from
the constant sheaf $\underline P$ associated to $P$ to the structure sheaf $\lb O_{\on{Spec} R[P]}$, and the log structure is given by the pushout $\underline P \oplus_{i^{-1} \lb O_{\on{Spec} R[P]}^\times} \lb O_{\on{Spec} R[P]}^\times$ in the category of sheaves of monoids on $\on{Spec} R[P]$. If $P$ is finitely generated, saturated, and
torsion-free (a \emph{toric monoid}) then $\on{Spec} R[P]$ is a toric
scheme over $\on{Spec} R$. In this case, if we take $j\colon U \to
\on{Spec} R[P]$ to be the inclusion of the open locus on which the log
structure is trivial, then this log structure is equal to that from
Example \ref{eg:log_str_from_open}. 
\end{example}

\begin{example}\label{eg:log_str_toroidal}
If $k$ is a field and $j\colon U \to X$ is a toroidal embedding over $k$, then $X$ carries a natural log structure which can be described in two different ways. First, applying Example \ref{eg:log_str_from_open} to the map $j$ equips $X$ with a log structure. Equivalently, by definition $X$ is \'etale locally modeled on toric varieties, and each of these toric varieties has a natural log structure by Example \ref{eg:log_str_toric}, which patch to a log structure on $X$. 
\end{example}
The above examples give the impression that log geometry is about boundary divisors, but the next example shows that the machinery is much more flexible, in a way which is critical to its application to moduli problems. 
\begin{example}\label{eg:log_str_pt}
Let $k$ be a field and $P$ a sharp monoid. Define a map of monoids on $X = \on{Spec} k$ by the formula 
\begin{equation}
\M_X = P \oplus k^\times \to k = \ca O_X \, ; \quad (a, u) \mapsto 0^au,
\end{equation}
where $0^a$ is $1$ if $a = 0 \in P$, and is 0 otherwise. Then $\M_X \to \ca O_X$ gives a log structure, with $\ghost_X = P$. 
\end{example}

\subsection{Morphisms of log schemes}

\begin{definition}\label{def:fs}
A morphism of log schemes $f\colon (X, \M_X)\to (Y, \M_Y)$ consists of a morphism of schemes $f\colon X \to Y$ together with a   morphism of sheaves of monoids $\M_Y \to f_*\M_X$ compatible with the maps to $\ca O_X$ and $\ca O_Y$. We say $f$ is \emph{strict} if the log structure on $X$ is pulled back from that on $Y$; more precisely, if $\M_X = f^{-1}\M_Y \oplus_{f^{-1} \ca O^\times_Y} \ca O^\times_X$. 
\end{definition}

\begin{example}
Let $k$ be a field, let $X$ be $\on{Spec} k$ with the log structure from Example \ref{eg:log_str_pt} with $P = \bb N$, and let $Y = \bb A^1_k$ with toric log structure; equivalently with log structure coming from the divisor $0$. Then there is a unique strict map $X \to Y$ sending $X$ to the origin. 
\end{example}

To minimise pathologies, we generally work with fine saturated (fs) log schemes. 
\begin{definition}
A log scheme $X$ is \emph{fs} if it admits a strict fppf surjection 
\begin{equation}
\bigsqcup_i U_i \to X
\end{equation}
and for each $i$ there is a strict morphism $U_i \to V_i$ where $V_i$ is a toric log scheme. 
\end{definition}
Equivalence to the more standard definitions follows from \cite[Prop I.1.3.5 and Theorem III.1.2.7 (4)]{oguslogbook}; in characteristic 0 we may replace fppf by \'etale. We say $X$ is \emph{log flat} if the maps $U_i \to V_i$ above can be chosen to be flat. 

From now on, by log scheme we always mean an fs log scheme (this holds
in the examples given above, assuming that in
Example \ref{eg:log_str_from_open} we take $U$ to be the complement of a
normal crossings divisor, and in Example \ref{eg:log_str_pt} we take $P$ a
toric monoid).  

We list those properties of morphisms of log schemes which will be particularly important for us. Let $f\colon X \to Y$ be a morphism of log schemes. 
\begin{enumerate}
\item we say $f$ is \emph{proper} if the underlying morphism of schemes is proper; 
\item we say $f$ is \emph{log \'etale} if it is \'etale-locally modelled (see \cite[Theorem 3.3.1]{oguslogbook} for details) on a morphism of toric varieties induced by a map $P \to Q$ of monoids which is injective and for which the cokernel of $P^\gp \to Q^\gp$ is finite and is of order invertible in $\ca O_X$; 
\item we say $f$ is \emph{log smooth} if it is \'etale-locally modelled on a morphism of toric varieties induced by a map $P \to Q$ of monoids which is injective and for which the order of the torsion subgroup of the cokernel of $P^\gp \to Q^\gp$ is invertible in $\ca O_X$. 
\end{enumerate}

\begin{remark}
\begin{enumerate}
\item Every log \'etale morphism is log smooth.
\item A strict morphism is log \'etale if and only if the underlying morphism of schemes is \'etale, and is log smooth if and only if the underlying morphism of schemes is smooth. 

\item Given a toric morphism of toric varieties such that the order of the cokernel of the associated lattice morphism is invertible, the associated map of log schemes is log smooth. In particular this holds for any toric map of toric varieties in characteristic 0, and also for any birational toric map of toric varieties over any base. In particular, toric blowups of toric varieties give examples of log
  \'etale maps which are not strict and not flat.  

\item If $X$ has snc log structure and $z$ is a closed stratum and
  $\tilde X$ is the blowup of $X$ along $z$, then the inverse image of
  the boundary of $X$ is again snc, and we equip $\tilde X$ with the
  corresponding snc log structure. The map $\tilde X \to X$ is then
  proper and log \'etale, and not strict unless $z$ was a divisor.  
\end{enumerate}
\end{remark}

\subsection{Log algebraic spaces and stacks}\label{sec:log_spaces_and_stacks}

We recall from \cite[\href{https://stacks.math.columbia.edu/tag/025Y}{Definition 025Y}]{stacks-project} that an algebraic space is a sheaf $\cat{Sch}^\op \to \cat{Set}$ for the \'etale topology whose diagonal is representable and which admits an \'etale cover by a scheme. There are (at least) two natural ways to generalise this to log schemes. The first is to equip the algebraic space with a log structure exactly as in Definition \ref{def:log_scheme}; indeed, this generalises to algebraic stacks. We will refer to these as \emph{algebraic spaces (or algebraic stacks) with log structure}. 

The second, more interesting, option is to consider functors 
$\cat{LogSch}^\op \to \cat{Set}$ which are sheaves for the strict \'etale topology, have diagonal representable by a log scheme, and which admit a \emph{log} \'etale cover by a log scheme. 

Every algebraic space with log structure gives naturally such a functor, but the converse is far from true; the key point is that log \'etale morphisms are not in general flat, so admitting a log \'etale cover by a log scheme is a much weaker condition than admitting an \'etale cover by a scheme. 

Log abelian varieties (see Section~\ref{sec:LAV}) are perhaps the key motivating examples of log algebraic spaces in the second sense. We discuss here a simpler example (which will also play an important role later). 

\begin{example}\label{eg:log_trop_tori}
Let $X$ be a log scheme. We define 
\begin{equation}
(\bb G_\LOG)_X\colon \cat{LogSch}_X^\op \to \cat{Set}; \quad T \mapsto \M_T(T)^\gp
\end{equation}
and 
\begin{equation}
(\bb G_\trop)_X\colon \cat{LogSch}_X^\op \to \cat{Set}; \quad T \mapsto \ghost_T(T)^\gp, 
\end{equation}
respectively the \emph{log line} and \emph{tropical line} over $X$. We claim that the
former is a log algebraic space in the second sense (but not in the
first sense). Indeed, consider the subfunctor $\bb P^1_X \subseteq
(\bb G_\LOG)_X$ sending $T$ to the set of those sections $m \in
\M_T(T)^\gp$ which are locally comparable to $0$:

\begin{equation*}
\begin{split}
\bb P^1_X\colon & \cat{LogSch}_X^\op \to \cat{Set}; \\ 
& T \mapsto \{m \in \M_T(T)^\gp: \exists U \sqcup V \to T: m|_U \in \M_U(U) \text{ and } m^{-1}|_V \in \M_V(V)\}.
\end{split}
\end{equation*}
Here $U \sqcup V \to T$ is a strict \'etale surjection. 
On the one hand, $\mathbb P^1_X$ is simply the projective line over $X$, with log structure given by pulling back the log structure from $X$ and adding in the $0$ and $\infty$ sections. On the other hand, the inclusion $\bb P^1_X \to (\bb G_\LOG)_X$ is a log \'etale morphism, locally modelled on subdividing the fan of a toric variety. 

We note that the natural map $\bb G_\LOG \to \bb G_\trop$ is a $\bb G_m$-torsor, in particular not \'etale; and $\bb G_\trop$ is not a log algebraic space in either sense. 
\end{example}

\subsection{Cone complexes and cone stacks}
We recall here definitions related to rational polyhedral cone complexes and cone stacks.
The definitions in this section are mostly taken from \cite{CCUW}.

A \emph{strongly convex rational polyhedral cone} is a pair $(N, \sigma)$ consisting
of a finitely generated free abelian group $N$ and a finite
intersection of halfspaces $\sigma$ in $N_{\R} = N \otimes_{\Z} \R$
defined over $\Q$, such that $\sigma$ contains no non-trivial linear subspaces of $N_{\R}$. The \emph{dimension} of $\sigma$ is the dimension of its affine span in $N_{\R}$. We say $\sigma$ is \emph{full-dimensional} if $\dim \sigma = \dim N_{\R}$.

Denote by $\on{RPC}$ the category of full-dimensional strongly convex rational
polyhedral cones together with $\Z$-linear morphisms. For simplicity, we refer to the objects of $\on{RPC}$ simply as \emph{cones}.

Recall that to any cone $(N, \sigma)$ one can associate a toric monoid $S_{\sigma} = \sigma^{\vee} \cap M$, where $M = N^{\vee}$ is the dual lattice, and $\sigma^{\vee} \subset M_{\R}$ is the cone dual to $\sigma$. Then $\sigma \mapsto S_{\sigma}$ defines an equivalence of categories between $\on{RPC}$ and the category of sharp toric monoids.

There is a topological realisation functor from $\on{RPC}$ to the category of topological spaces,
\[
\on{RPC} \longrightarrow \on{TOP}\; , \; (N, \sigma) \longmapsto \sigma,
\]
where $\sigma$ on the RHS is viewed as a topological space. 

A map $\tau \to \sigma$ in $\on{RPC}$ is a \emph{face map} if it induces an isomorphism of $\tau$ onto a face of $\sigma$. A face map is moreover \emph{proper} if it induces an isomorphism onto a proper face. 

\begin{definition}
A \emph{rational polyhedral cone complex $\Sigma$} consists of a collection of cones $\ca C = \left\{\sigma_{\alpha}\right\}$ in $\on{RPC}$  together with a collection of face morphisms $\ca F = \left\{\phi_{\alpha \beta} \colon \sigma_{\alpha} \to \sigma_{\beta}\right\}$, closed under composition, such that the following axioms are satisfied:
\begin{enumerate}
\item[(i)] for all $\sigma_{\alpha} \in \ca C$, the identity map $\id\colon \sigma_{\alpha} \to \sigma_{\alpha}$ is in $\ca F$. 
\item[(ii)] every face of a cone $\sigma_{\beta}$ is the image of exactly one face morphism $\phi_{\alpha \beta} \in \ca F$. 
\item[(iii)] there is at most one morphism $\sigma_{\alpha} \to \sigma_{\beta}$ between two cones $\sigma_{\alpha}$ and $\sigma_{\beta}$ in $\ca F$. 
\end{enumerate}
A map $f \colon \Sigma_1 \to \Sigma_2$ between two rational polyhedral cone complexes is a choice, for each $\sigma_{\alpha_1} \in \Sigma_1$, of a cone $\sigma_{\alpha_2} \in \Sigma_2$ and a morphism $f_{\alpha_1 \alpha_2} \colon \sigma_{\alpha_1} \to \sigma_{\alpha_2}$ which does not factor through any proper face map $\tau \to \sigma_{\alpha_2}$ 
 such that for any face map $\phi_{\alpha_1\beta_1} \in \ca F_1$ there exists a face map $\phi_{\alpha_2\beta_2}\colon \sigma_{\alpha_2} \to \sigma_{\beta_2}$ in $\ca F_2$ making the following diagram commute: 
\[
\begin{tikzcd}
\sigma_{\alpha_1}  \arrow[r, "f_{\alpha_1\alpha_2}"] \arrow[d, "\phi_{\alpha_1\beta_1}"]  &   \sigma_{\alpha_2} \arrow[d, "\phi_{\alpha_2\beta_2}"'] \\
\sigma_{\beta_1} \arrow[r, "f_{\beta_1\beta_2}"'] &   \sigma_{\beta_2}\\
\end{tikzcd}.
\]
\end{definition}
The category of rational polyhedral cone complexes is denoted by $\on{RPCC}$. For simplicity, we refer to the objects in $\on{RPCC}$ simply as \emph{cone complexes}. 

As before, we have a topological realisation functor 
\[
\on{RPCC} \longrightarrow \on{TOP}\; , \; \Sigma \longmapsto |\Sigma|,
\]
where $|\Sigma| \coloneqq \varinjlim_{\ca F}\sigma_{\alpha}$, the colimit taken in the category $\on{TOP}$. 
The topological space  $|\Sigma|$ is often referred to as the \emph{support of} $\Sigma$ (see e.\,g.\,\cite{KKMSD}). 

We have a full and faithful embedding 
\[
\on{RPC} \to \on{RPCC}
\] defined by associating to a cone $\sigma$ the cone complex consisting of all of its faces. This is again denoted by $\sigma$. 

A map of cone complexes $f \colon \Sigma_1 \to \Sigma_2$ is called \emph{strict} if all the maps $f_{\alpha_1\alpha_2} \colon \sigma_{\alpha_1} \to \sigma_{\alpha_2}$ are isomorphisms. We denote by $\on{RPCC}^{\rm{face}}$, resp.\,$\on{RPC}^{\rm{face}}$, the category of cone complexes, resp. cones, with strict morphisms. 

\begin{definition}
A \emph{cone stack} is a category fibered in groupoids $\tau \colon W \to \on{RPC}^{\rm{face}}$. A \emph{morphism of cone stacks} $(W, \tau) \to (W', \tau')$ is a functor $\varphi \colon W \to W'$ and for all $\alpha \in W$ a map $\varphi_{\alpha} \colon \tau(\alpha) \to \tau'(\varphi(\alpha))$ such that the following conditions are satisfied:
\begin{enumerate}
\item[(i)] for all $\alpha \to \beta$ in $W$, the diagram
\[
\begin{tikzcd}
\tau(\alpha)  \arrow[r, "\varphi_\alpha"] \ar[d]& \tau'(\varphi(\alpha)) \ar[d]\\
\tau(\beta) \arrow[r, "\varphi_\beta"]  &   \tau'(\varphi(\beta)) \\
\end{tikzcd}
\]
commutes.
\item[(ii)] for all $\alpha$ in $W$, the image of the map $\varphi_\alpha \colon \tau(\alpha) \to \tau'(\varphi(\alpha))$ is not contained in any proper face.
\end{enumerate}
\end{definition}
\begin{remark}
For $\tau \colon W \to \on{RPC}^{\rm{face}}$ to be a category fibered in groupoids means that the following axioms are satisfied:
\begin{enumerate}
\item[(i)] for all $\alpha \in W$ and for any face map $\gamma \to \tau(\alpha)$ there exists $u \colon \beta \to \alpha$ in $W$ such that $\tau(\beta) = \gamma$ and $\tau(u) = \gamma \to \tau(\alpha)$.
\item[(ii)] if $u \colon \eta \to \alpha$ and $\nu \colon \beta \to \alpha$ are maps in $W$ such that the map $\tau(\eta) \to \tau(\alpha)$ factors through $\tau(\beta) \to \tau(\alpha)$, then there exists a unique $\omega \colon \eta \to \beta$ such that $u = \nu \circ \omega $ and $\tau(\omega)$ is the face map $\tau(\eta) \to \tau(\beta)$.

\end{enumerate}
\end{remark}
\begin{remark} A cone complex can be seen as a particular case of a cone stack by taking $W$ the face poset of the complex.
\end{remark}

\begin{remark}
The reader may be surprised by the lack of mention of a Grothendieck topology in the above definition of cone stack; this is because a cone admits no covers in the face topology, see \cite[Proposition 2.3]{CCUW}.
\end{remark}

\subsection{Tropicalisation} \label{sec:tropicalisation}

Tropicalisations are often seen as purely combinatorial objects, but can also be given the structure of algebraic stacks with log structure. While making things somewhat more technical, this has the great advantage that a log scheme and its tropicalisation live in the same category, and there is a natural morphism between them. The main reference for this section is again \cite{CCUW}. 

Let $R$ be a ring.
If $N$ is a lattice with dual $M$, and if $(N, \sigma)$ is a strongly convex rational polyhedral cone, then
$\sigma^\vee \cap M$ is a toric monoid, with associated toric variety
$X_\sigma \coloneqq \on{Spec} R[\sigma^\vee \cap M]$. The dense torus
is $T_\sigma\coloneqq \on{Spec} R[(\sigma^\vee \cap M)^\gp]$,

and the quotient 
\begin{equation}
\ca A_\sigma \coloneqq [X_\sigma/T_\sigma]
\end{equation}
is an algebraic stack with log structure, known as the \emph{Artin
  cone} associated to~$\sigma$. This is an algebraic stack of
dimension 0, and its $R$-points are naturally in bijection with the
torus orbits in $X_\sigma$.

The construction of the Artin cone can be globalised, in the sense that to a cone stack $\Sigma$ there is naturally associated an algebraic stack $\ca A_\Sigma$ with log structure, the \emph{Artin fan} of $\Sigma$. This construction induces an equivalence between cone stacks and suitable 0-dimensional algebraic stacks with log structure. Because of this equivalence, after this subsection we will use the notations $\Sigma$ and $\ca A_\Sigma$ interchangeably. 

In the affine toric case, there is a natural map $X_\sigma \to \ca A_\sigma$, which we call the \emph{tropicalisation map}, and $\ca A_\sigma$ is the \emph{tropicalisation} of $X_\sigma$. This naturally extends to toric varieties; if $F$ is a fan then there is a natural tropicalisation map $X_F \to \ca A_F$, which is a strict smooth map of stacks with log structure. 

For $X$ an algebraic stack with log structure, we define a \emph{tropicalisation} to be any pair of a cone stack $\Sigma$ and a strict flat surjective map $X \to \ca A_\Sigma$. Tropicalisations are not unique, and making `good' choices is an interesting problem in general.

Here surjectivity is imposed to avoid situations where the tropicalisation is much more complicated than the log scheme itself, and flatness leads to many technical simplifications; for example, subdivisions (defined in the next section) are always birational, and the locus in $X$ where the log structure is trivial is dense open. 

\begin{remark} \label{rem:clemens}
If $X$ has snc log structure, then there is a natural choice of
tropicalisation given by the Clemens complex; this has one ray for
each boundary divisor in $X$, one $2$-dimensional cone for each
connected component of the intersection of two boundary divisors,
etc. In this case the tropicalisation map is strict, surjective, and
smooth. However, this is not the only choice of tropicalisation, and
indeed we will make different choices for our log compactifications of
the moduli spaces $N_g$ and $N_{g,1}$, see Section~\ref{sec:universal_Ng}, in particular Example \ref{eg:example_genus_one}. 
\end{remark}

\subsection{Subdivisions}\label{sec:subdivisions}

Let $X$ be an algebraic stack with log structure. A \emph{subdivision} of $X$ is a proper representable log \'etale monomorphism $X' \to X$ which is a universal surjection for base changes to integral log schemes (see \cite{Hu-Schimpf}). The base-change of a subdivision is a subdivision. 

If $\Sigma$ is an Artin fan then by \cite[\S 4.3]{abramovich2017boundedness} subdivisions of $\Sigma$ are naturally in bijection with subdivisions of the associated cone stack; these are just surjective maps which are subdivisions on each cone in the sense of \cite{fulton_toric}. In this case subdivisions are always birational.

In particular, if $f\colon X \to \Sigma$ is a tropicalisation then a
subdivision of $\Sigma$ induces a subdivision of $X$ by base-change
along $f$. Such a subdivision is again birational (as the flat
pullback of a birational map). 

The key example to keep in mind is the subdivision of the fan of a toric variety, yielding a birational modification of the toric variety itself.

\subsection{Piecewise linear functions and divisors} \label{sec:piecewise_linear}

\subsubsection{Strict piecewise linear functions}
Let $X$ be an algebraic stack with log structure. The \emph{sheaf of strict piecewise linear (sPL) functions} on $X$ is the sheaf $\ghost_X^\gp$, also denoted $\on{sPL}_X$. A \emph{strict piecewise linear (sPL) function} on $X$ is a global section of $\ghost_X^\gp$. We denote the group of sPL functions on $X$ by $\on{sPL}(X)$. Thus we have $\on{sPL}(X) = H^0(X, \ghost_X^\gp)$.

\begin{example} If $\Sigma$ is an Artin fan then an sPL function on $\Sigma$ is the same as a continuous function on the associated cone stack which is linear on each cone. 
\end{example}

\begin{definition} \label{def:line_bundle_for_sPL}
From the exact sequence 
\begin{equation}\label{eq:basic_log_exact_seq}
1 \to \ca O_X^\times \to \M_X^\gp \to \ghost_X^\gp \to 0
\end{equation}
we see that any $\beta \in \ghost_X^\gp(X)$ induces a line bundle $\ca O_X(\beta)$ on $X$; namely, the inverse images of local restrictions of $\beta$ that lift to $\M_X^\gp$ form an $\ca O_X^\times$-torsor. Then we take the associated line bundle. Note that there is a choice of sign in passing from $\ca O_X^\times$-torsors to line bundles; we choose to glue in the $\infty$ section, so that non-negative sPL functions correspond to effective divisors. The corresponding isomorphism class is obtained by applying the connecting homomorphism $H^0(X, \ghost_X^\gp) \to H^1(X, \ca O_X^\times)$. 
\end{definition}

\begin{definition}\label{def:sPL_to_div}
Suppose that the log structure on $X$ is trivial on some dense open $U \hra X$ (this will usually be the case for us), then the line bundle $\ca O(\beta)$ comes with a canonical trivialisation on $U$, inducing in turn a Cartier divisor $\on{div}\beta$ on $X$, supported outside $U$. Putting this together, we have a natural map 
\begin{equation}
\on{div}\colon \on{sPL}(X) \to \on{Div}(X),
\end{equation}
taking non-negative sPL functions to effective divisors. We can tensor both sides with $\bb Q$ or $\bb R$ to obtain maps taking rational- or real-valued sPL functions to divisors with rational or real coefficients.
\end{definition}
 If $X$ is a toric variety, this all reduces to the standard correspondence between PL functions on the fan and torus-equivariant Cartier divisors. However, unlike in the toric case, in general not every divisor on $X$ will be linearly equivalent to one coming from an sPL function, for example if the log structure on $X$ is trivial and not every divisor on $X$ is principal. 

\subsubsection{Piecewise linear and conical functions}

In this section we define general piecewise linear functions and conical functions on algebraic stacks with log structure, generalising the natural notions on (the fans of) toric varieties. 

If $X$ is an atomic log scheme\footnote{Informally this means that $X$ is very small; the formal definition is in \cite[Definition 2.2.4]{abramovich-wise-birat}, examples are affine toric varieties, and log schemes whose underlying scheme is strictly henselian local. } then a PL function on $X$ is defined to be a continuous PL function on the cone $\ghost_X^\vee$ (we do not require it to be linear on any particular subdivision, but we do require the existence of a finite subdivision such that it is strict, i.e. linear on each cone). A \emph{conical function} on $X$ is defined to be a function $f$ from the rational points of the cone $\ghost_X^\vee$ to $\bb R$ with the property that $f(\lambda x) = \lambda f(x)$ for all $\lambda \in \bb Q_{\ge 0}$ and rational points $x$ in $\ghost_X^\vee$ (i.e., homomorphisms $x\colon \ghost_X \to \bb Q_{\ge 0}$). We say that a given conical function is \emph{continuous} if there exists a continuous function on the real points of the cone whose restriction to the rational points is the given function. 

If $Y \to X$ is a map of atomic log schemes then there is a natural pullback map on PL and (continuous) conical functions. 
If $X$ is an algebraic stack with log structure, and $U = \bigsqcup_i U_i \to X$ is a strict smooth
cover by atomic log schemes, and 
\begin{equation}\label{eq:cone_groupoid}
V=\bigsqcup_j V_j \to U \times_X U
\end{equation} is a smooth cover of the fibre product by atomic log schemes, then we may define a PL/conical/continuous conical function on $X$ to be a PL/conical/continuous conical function on each $U_i$ compatible with pullback to the $V_j$. In this way we define the sheaves:
\begin{itemize}
\item $\on{sPL}_X$ of strict piecewise linear functions (required to be linear on the cones of $X$, and taking \emph{integral} values on lattice points); 
\item $\on{PL}_X$ of piecewise linear functions (required to be linear on the cones of some subdivision of $X$, and required to take \emph{rational} values on lattice points);
\item $\on{Con}_X$ of real-valued conical functions; 
\item $\on{CCon}_X$ of real-valued continuous conical functions. 
\end{itemize}
All these notions are independent of the choices of covers $U$ and $V$. There are natural pullback maps on these sheaves. We have inclusions
\begin{equation}
\on{sPL}_X \to \on{PL}_X \to \on{CCon}_X \to \on{Con}_X. 
\end{equation}

We can give a more concrete combinatorial interpretation of these functions. We fix covers $U$ and $V$ as above, yielding a diagram of atomic log schemes 
\begin{equation}
\bigsqcup_j V_j \rightrightarrows \bigsqcup_i U_i. 
\end{equation}
This induces a diagram of cones 
\begin{equation}\label{eq:cone_diagram}
\bigsqcup_j \ghost^\vee_{V_j} \rightrightarrows \bigsqcup_i \ghost^\vee_{U_i}, 
\end{equation}
where all the maps are face inclusions because the maps $V_j \to U_i$ are strict. We let $T$ be the colimit of the diagram of topological realisations in the category of topological spaces; the pair consisting of $T$ and the diagram \eqref{eq:cone_diagram} is a \emph{generalised cone complex} in the sense of \cite[\S 2.6]{ACP}. We say a real-valued function on $T$ is continuous conical/conical/PL/sPL if it is so after pullback to each cone $\ghost^\vee_{U_i}$. Since functions on the colimit are exactly functions on the $\ghost^\vee_{U_i}$ whose pullbacks to the $\ghost^\vee_{V_j}$ agree, we deduce
\begin{lemma}
There are natural bijections
\begin{equation*}
\begin{split}
\sPL(X) &\to \sPL(T)\\
\on{PL}(X) &\to \on{PL}(T)\\
\on{Con}(X) &\to \on{Con}(T)\\
\on{CCon}(X)& \to \on{CCon}(T).\\
\end{split}
\end{equation*}
\end{lemma}

In particular suppose that $X = \ca A_\Sigma$ for some cone stack $\Sigma$. One may then choose the $U_i$ to correspond to the cones of $\Sigma$ and the $V_j$ to correspond to the pairwise intersections of cones (which are again cones). Then $T$ is the colimit in the category of topological spaces of the diagram $\Sigma$, and the functions on $\ca A_\Sigma$ are the functions on $T$. 

Given a cone stack $\Sigma$ and a map $X \to \ca A_\Sigma$ then we can write down sPL (respectively PL, or conical, or continuous conical) functions on $X$ by writing them down on $\Sigma$ and pulling back. We do not claim that every such function on $X$ arises in this way, but this approach will be enough to construct all the examples we need. 

We next prove some basic properties of the above spaces of functions. 
Let $X$ be an algebraic stack with log structure.
\begin{lemma}
If $X$ is noetherian and $f$ is a PL function on $X$ then there exist a positive integer $n$ and a subdivision $X' \to X$ such that the pullback of $nf$ to $X'$ is strict PL. 
\end{lemma}
\begin{proof}
Locally this is clear (by definition any PL function on a cone becomes strict upon some multiplication and subdivision). Since $X$ is noetherian  we can find a global subdivision refining any finite collection of local subdivisions. 
\end{proof}

\begin{lemma}\label{lem:PL_approx}
Suppose that $X$ is noetherian, then $\on{PL}(X)$ is dense in $\on{CCon}(X)$ where the latter is equipped with the topology of uniform convergence. 
\end{lemma}
\begin{proof}
This follows from the Lattice Stone-Weierstrass Theorem applied to the
generalised cone complex $T$. To apply this theorem have we to check that
PL functions separate points; this follows from the fact that if $p$
and $q$ are points of a cone which do not lie on a common ray  then there exists an integer $n$ such
that $p$ and $q$ lie in cells $\sigma_p$ and $\sigma_q$ of the $n$-th
iterated barycentric subdivision, with the property that $\sigma_p\cap
\sigma_q = \{0\}$.

\end{proof}

\subsection{Log b-divisors}\label{sec:log_divisors-on_stacks}

In this section we consider an algebraic stack with log structure $X$ which is log flat (for
example, toroidal). This implies that the largest open set $U \hookrightarrow X$ where the log structure is trivial, is dense in $X$, and that any subdivision of $X$ is birational, in fact an isomorphism over $U$. 

Following Definition~\ref{def:sPL_to_div} any sPL
function $\alpha$ on $X$ induces a Cartier divisor $\divisor \alpha$
on $X$, supported outside $U$. If $\alpha$ is any PL function on $X$
then it induces a Cartier divisor on any subdivision $X' \to X$ on
which $\alpha$ is strict.

We discuss how these divisors relate as the subdivision varies. We let $R(X)$ denote the set of all subdivisions of $X$, with morphisms over $X$,  and $R_1(X)$ the subset of all subdivisions all of whose cones are unimodular (which means they are generated by a subset of an integral lattice basis; equivalently, the associated affine toric variety is smooth).

Any map of subdivisions $X_2 \to X_1$  is proper and birational, so Cartier divisors can be pulled back and Weil divisors can be pushed forward. 
If $X' \to X$ is a subdivision, we write $\operatorname{Ca-Div}(X')$ for the group of $\bb Q$-Cartier $\bb Q$-divisors on $X'$, where any component which meets $U$ is required to have integral coefficient. This follows~\cite[\S 2.2.3]{yz}. Similarly we write $\operatorname{W-Div}(X')$ for the group of $\bb R$-Weil divisors on $X'$ where any component which meets $U$ is required to have integral coefficient. 

The following definition generalises the notion of toroidal b-divisors for toroidal embeddings over an algebraically closed field of characteristic zero from \cite{botero_burgos}.

\begin{definition} \label{def:log_b_Cartier}
We define the set of \emph{log Cartier b-divisors} as the colimit of Cartier divisors over subdivisions, and the set of \emph{log Weil b-divisors} as the limit of Weil divisors over subdivisions, i.e.,
\[
\logcbDiv(X) \coloneqq \varinjlim_{X' \in R(X)}\operatorname{Ca-Div}(X')
\]
and 
\[
\logwbDiv(X) \coloneqq \varprojlim_{X' \in R(X)}\operatorname{W-Div}(X').
\]
\end{definition}

The maps $\on{div}$ are compatible with pullback of Cartier divisors, in the sense that for all maps $X_2 \to X_1$ in $R(X)$ the following square commutes: 
\begin{equation}
\begin{tikzcd}
\sPL(X_2)  \arrow[r, "\on{div}"] &   \on{Ca-Div}(X_2) \\
\sPL( X_1) \ar[u] \arrow[r, "\on{div}"] &   \on{Ca-Div}( X_1). \ar[u]
\end{tikzcd}
\end{equation}
Here the vertical arrows are given by pullback. 
This induces a map 
\begin{equation}
\on{div} \colon\on{PL}(X) \longrightarrow \logcbDiv(X). 
\end{equation}

The maps $\on{div}$ are compatible with pushforward of Weil divisors in the case that $X_1$ and $X_2$ are both unimodular, in the sense that the following square commutes: 
\begin{equation}
\begin{tikzcd}
\sPL(X_2) \ar[d] \arrow[r, "\on{div}"] &   \on{W-Div}(X_2) \ar[d]\\
\sPL( X_1)  \arrow[r, "\on{div}"] &   \on{W-Div}( X_1). 
\end{tikzcd}
\end{equation}
Here the left vertical arrow takes an sPL function on $X_2$ to the unique $\sPL$ function on $X_1$ with the same values on the rays (this exists since $X_1$ is unimodular).

Given a conical function $c$ on $X$, i.e., a global section of $\on{Con}_X$, for any $X' \in R_1(X)$, we can define a real-valued sPL function $c_{X'}$ on $X'$ by interpolating the values of $c$ on the rays of $X'$, inducing an element of $\on{W-Div}( X')$. The compatibility with pushforward implies that these divisors assemble to form an element of $\logwbDiv(X)$. In this way we define an additive group homomorphism 
\begin{equation}\label{eq:con_to_div}
\divisor\colon \on{Con}(X) \longrightarrow \logwbDiv(X)
\end{equation}
from the set of conical functions on $X$. 
\begin{lemma}\label{lem:con_injective}
We write $\on{Div}(U)$ for the group of Weil divisors on $U$ with integral coefficients. Then the map 
\begin{equation*}
\on{Div}(U) \oplus \on{Con}(X) \to \logwbDiv(X)
\end{equation*}
induced by Zariski closure of divisors and by the map $\divisor$ from \eqref{eq:con_to_div}, is injective. If $X$ is log smooth then it is an isomorphism.  

\end{lemma}
\begin{proof}
For injectivity, note that divisors in the image of $\divisor$ are always supported in the complement of $U$. Assume that $X$ is log smooth. For surjectivity, we can decompose any log Weil b-divisor $\bb D$ into a summand $\bb D_B$ supported on the boundary, and a summand $\bb D_U$ consisting of divisors whose generic point is contained in $U$. Then $\bb D_U$ lies in the image of $\on{Div}(U)$. To show that $\bb D_B$ lies in the image of $\on{Con}(X)$ we may work locally (and then patch, by injectivity), assuming that $X$ is atomic with cone $\sigma$. As $X$ is log smooth, the irreducible components of the boundary divisors of models of $X$ correspond to unique rays in $\sigma$ with rational slope, and if a component $Y$ has multiplicity $m$ in $\bb D_B$ then we take the conical function to have value $m$ on the first lattice point on the ray corresponding to $Y$. This recipe (uniquely) determines a conical function on $\sigma$ which maps to $\bb D_B$. 
\end{proof}
\begin{remark} Without assuming that $X$ is log smooth, rays in $\sigma$ would still determine divisors in the boundary of models of $X$, but these divisors might not be irreducible, and the above argument (and indeed the surjectivity in the lemma) would fail. For example, one could take $X = \bb A^2$ where the log structure is trivial outside the coordinate axes and has constant rank~$1$ over the coordinate axes. Then $X$ is log flat but not log smooth (over the point with trivial log structure), and the union of the coordinate axes, which corresponds to the unique ray in the associated cone, is not irreducible.
\end{remark}
\begin{remark}\label{rem:phi_D}
Assuming $X$ is log smooth, if $D$ is a Weil divisor supported outside $U$ then Lemma~\ref{lem:con_injective} implies it comes from a unique PL function. We denote that function by $\phi_D$. 
\end{remark}

\subsection{Log b-line bundles}  \label{sec:b-line-bundles}

Let $X$ be an algebraic stack with log structure.
The purpose of this section is to define and study the notion of log b-line bundle on $X$. A key input will be the notion of log and tropical line bundles. When $X$ is log smooth, there is a natural map from log Weil b-divisors to isomorphism classes of log b-line bundles.

\begin{definition}  A \emph{log line bundle} on $X$ is an $\M_X^\gp$ torsor. A \emph{tropical line bundle} on $X$ is an $\ghost_X^\gp$ torsor. We write $\on{LogPic}(X)$ resp.\ $\on{TroPic}(X)$ for the groups of isomorphism classes of log resp.\ tropical line bundles. The exact sequence \eqref{eq:basic_log_exact_seq} induces an exact sequence 
\begin{equation}\label{eq:pic_to_logpic_seq}
\on{sPL}(X) \to \on{Pic}(X) \to \on{LogPic}(X) \to \on{TroPic}(X). 
\end{equation}
The maps $\on{Pic}(X) \to \on{LogPic}(X)$ and $\on{LogPic}(X) \to \on{TroPic}(X)$ are given by extension of scalars. Given a line bundle $\ca L$ on $X$, the \emph{logification} of $\ca L$, notation $\ca L^{\LOG}$, is the extension of scalars of $\ca L$ along the map $\oo_X^\times \to \M_X^\gp$.
Given a log line bundle $\llb L$ on $X$, the \emph{tropicalisation} of $\llb L$, notation $\llb L^\trop$, is the extension of scalars of $\llb L$ along the map $\M_X^\gp \to \ghost_X^\gp$.
\end{definition}

\begin{definition}\label{def:log-b-line-bundle}

Given a tropical line bundle $\tlb L$ on $X$, a \emph{conical section} of $\tlb L$ is a section of $\tlb L \otimes_{\sPL_X} \on{Con}_X$.  A \emph{log b-line bundle} on $X$ is a pair $(\llb L, c)$ where $\llb L$ is a log line bundle on $X$ and $c$ is a conical section of its tropicalisation $\llb L^\trop$.

An isomorphism of log b-line bundles consists of compatible isomorphisms between the log and tropical line bundles, where the latter preserves the conical sections. Tensor product is defined by tensoring the log line bundles and adding the conical sections. We write $\on{b-Pic}(X)$ for the group of isomorphism classes of log b-line bundles on $X$. 
\end{definition}

\begin{remark}
If $X$ is of the form $\ca A_\Sigma$, where $\Sigma$ is a cone stack, then a tropical line bundle $\tlb L$ admits a purely combinatorial interpretation as torsor in the star topology for the sheaf of sPL functions, cf.\ \cite[Section 2.4]{cavalieri_gross}. Similarly the sheaves of piecewise linear and conical functions have intuitive combinatorial meanings as sheaves in the star topology. 
\end{remark}

There is a natural map
\begin{equation} \label{eq:Pic_to_b-Pic}
\on{Pic}(X) \to \on{b-Pic}(X); \quad \lb L \to (\lb L^\LOG, 0)
\end{equation}
sending a line bundle $\lb L$ to the induced log line bundle $\lb L^{\LOG}$ together with the conical section 0 of the trivial tropical bundle, noting that the tropicalisation of $\lb L^{\LOG}$ is trivial by exactness of \eqref{eq:pic_to_logpic_seq}. 

\begin{lemma} The natural map $\on{Pic}(X) \to \on{b-Pic}(X)$ is injective. Its image is exactly the subgroup of log b-line bundles $(\ca L, c)$ where $c$ is an sPL section; in other words, a trivialisation of the associated tropical bundle. 
\end{lemma}
\begin{proof} Suppose two line bundles $\lb L$ and $\lb L'$ become isomorphic as log line bundles, then from the exactness of \eqref{eq:pic_to_logpic_seq} there exists a PL function $\alpha$ and an isomorphism $\lb L \iso \lb L'(\alpha)\coloneqq \lb L' \otimes_{\lb O_X} \ca O_X(\alpha)$.  If $\lb L$ and $\lb L'$ become isomorphic as log b-line bundles then $0_{\ca L^\trop}  = 0_{{\ca L'}^\trop} + \alpha$, which implies that $\alpha = 0$. 
\end{proof}
\begin{remark} \label{rem:Pic_subdivisions} If $X' \to X$ is a subdivision, we likewise have a natural injective map $\on{Pic}(X') \to \on{b-Pic}(X)$, compatible with the map $\on{Pic}(X) \to \on{b-Pic}(X)$. Ranging over all subdivisions of $X$, we obtain the subgroup of log b-line bundles $(\ca L, c)$ where $c$ is a PL section.
\end{remark}
The construction in Definition~\ref{def:line_bundle_for_sPL} generalises to log line bundles.
\begin{definition} \label{def:line_bundle_for_adm}  Let $\llb L$ be a log line bundle on $X$ and $f$ an sPL section of its tropicalisation~$\llb L^\trop$. Then $f$ is a trivialisation, which implies that $\llb L$ is obtained by extension of scalars from a canonical $\oo_X^\times$-torsor. We write $\lb L_f$ for the associated line bundle. The log b-line bundles $(\llb L,f)$ and $(\lb L_f^{\LOG},0)$ are isomorphic, that is, the image of $\lb L_f$ in $\on{b-Pic}(X)$ is equal to $(\llb L,f)$.
\end{definition}

We also have a natural map  
\begin{equation} \label{eqn:con_to_b-pic}
\on{Con}(X) \to \on{b-Pic}(X); \quad c \mapsto (\lb O^\LOG, c). 
\end{equation}
\begin{remark} The map $\on{Con}(X) \to \on{b-Pic}(X)$ is \emph{not} in general injective, since there may exist sPL functions which map to zero in $\on{Pic}(X)$. If $c$ is any such function, then the b-line bundle $(\ca O^\LOG, c)$ is trivial. 
\end{remark}

We would like to put a topology of uniform convergence on the spaces of conical sections of a tropical line bundle, as follows.

\begin{definition}\label{def:conical_top}

Suppose $X$ is noetherian, and let $\tlb L$ be a tropical line bundle on $X$. Choose a cover of $X$ by finitely many
atomic log schemes $U_i$, with corresponding cones $\sigma_i$. Assume that the cover trivialises $\tlb L$, and choose a trivialising section $\ell_i$ on each $U_i$. Each
extremal ray $\rho$ of $\sigma$ has a first integral point $p_\rho$;
we let $H_i\subset \sigma_i$ be the convex hull of the $p_\rho$. We equip the space of global sections $\Gamma(X, \on{Con}_X \otimes_{sPL_X} \tlb L)$ with the
topology of uniform convergence for the restrictions of conical
sections, after subtracting $\ell_i$ to obtain functions, to the $H_i$. One checks that this does not depend on the
choice of finite cover $U_i$ or the choice of trivialising sections $\ell_i$. 
\end{definition}

We have the following generalization of Lemma~\ref{lem:PL_approx}.
\begin{lemma}\label{lem:PL_approx_bundle}
Suppose that $X$ is noetherian, then $\Gamma(X, \on{PL}_X \otimes_{sPL_X} \tlb L)$ is dense in $\Gamma(X, \on{Con}_X \otimes_{sPL_X} \tlb L)$ where the latter is equipped with the topology of uniform convergence described above. 
\end{lemma}
\begin{proof}
It suffices to prove that if $\tlb L$ admits a continuous conical section $s$, then it admits a PL section $\frak s$. Indeed, then $s - \frak s$ is a continuous conical function, and by Lemma \ref{lem:PL_approx} we can approximate it uniformly by a sequence of PL functions $\frak f_i$; then the sequence of PL sections $\frak s + \frak f_i$ converges uniformly to $s$. 

Choose a cover of $X$ by finitely many
atomic log schemes $U_i$, and let $V_j$ be a cover of $U \times_X U$ by finitely many atomic log schemes.
Let $\tilde X \to X$ be the barycentric subdivision; then each cone is simplicial. Since the maps $V_j \to U_i$ are strict they correspond to face inclusions, and hence the barycentric subdivision of $X$ simply corresponds to taking the barycentric subdivision of the cone of each $U_i$ and $V_j$. We restrict the conical section $s$ to the rays of the barycentric subdivision, then since this is simplicial there is a unique sPL section $\frak s$ on $\tilde X$  given by linearly interpolating between the rays. \end{proof}

\section{Log abelian varieties and their tropicalisation}\label{sec:LAV}

In this section we recall (and at one point very slightly generalise) the theory of log abelian varieties from the works of Kajiwara--Kato--Nakayama, in particular \cite{LAV2, LAV7}.  The basic starting point for this story is the classical problem of compactifying degenerating families of abelian varieties; it has long been known that one cannot in general construct a proper model with a group structure, and much work has been devoted to finding good models under more relaxed conditions. In particular, if one weakens the properness condition one obtains the theory of N\'eron models, and if one does not require a group structure then there is a large literature on various toroidal compactifications that can be built. The idea behind the works of Kajiwara--Kato--Nakayama on log abelian varieties is to weaken in a different direction, by building a model that is proper, and has a group structure, but is not a scheme (or algebraic space or algebraic stack); rather, it is a more exotic object, a \emph{log algebraic space in the second sense} as described in Section \ref{sec:log_spaces_and_stacks}. The full theory is quite technical, and we do not reproduce it here; rather, we hope to give enough details that a reader can follow our arguments, taking on trust many deep results from the works of Kajiwara--Kato--Nakayama.

\subsection{Log 1-motifs}\label{sec:log_motif}
Log 1-motifs are in some sense the building blocks of log abelian varieties; every log abelian variety over a geometric point is obtained by taking a quotient of a log 1-motif. We begin by setting up their definition.

\subsubsection{Log tori and semiabelian schemes}

Let $S$ be a log scheme. Recall from Section \ref{sec:log_spaces_and_stacks} the log algebraic spaces $\bb G_\LOG$ and $\bb G_\trop$, defined by $\bb G_\LOG(S) = \M_S^\gp(S)$ and $\bb G_\trop(S) = \ghost_S^\gp(S)$. Given a torus $T$, set $T_\LOG \coloneqq \Hom(\Hom(T, \bb G_m), \bb G_\LOG)$, a ``log torus''; for example, if $T = \bb G_m^n$, then $T_\LOG = (\bb G_\LOG)^n$. Note that we have a natural map $T \to T_\LOG$.

Now let $G$ be a semiabelian scheme over a log scheme $S$, and suppose that $G$ fits into an exact sequence 
\begin{equation}
1 \to T \to G \to B \to 0
\end{equation}
with $T$ a torus and $B$ abelian (in particular, $G$ has constant toric rank). We define $G_\LOG$ to be the pushout of $T_\LOG$ with $G$ over $T$, yielding a commutative diagram with exact rows
\begin{equation}
 %\begin{tikzcd}
 \xymatrix{
  1 \ar[r] & T \ar[r] \ar[d] & G \ar[r]  \ar[d] & B \ar@{=}[d]  \ar[r] & 1\\
  1 \ar[r] & T_\LOG \ar[r]  & G_\LOG \ar[r] & B  \ar[r] & 1. \\  
  }
%\end{tikzcd}
\end{equation}

We define $G_\trop = G_\LOG/G$. For example we have $(\bb G_m)_\trop = \bb G_\trop$.   
We define $X = \Hom(T, \bb G_m)$, the character lattice of $T$. Then $G_\trop = \Hom(X,\bb G_\trop)$.
Notice that $G_\trop = T_\trop$. Finally we have canonical identifications
\begin{equation} \label{eqn:describe_X}
X = \Hom(T, \bb G_m) = \Hom(T_\LOG, \bb G_\LOG) = \Hom(T_\trop, \bb G_\trop) = \Hom(G_\trop, \bb G_\trop).
\end{equation}

\subsubsection{Log 1-motifs}

\begin{definition}
A \emph{log 1-motif} over the log scheme $S$ is a triple $(Y, G, u)$ where
\begin{enumerate}
\item
$Y$ is a sheaf on $S$ of abelian groups, locally free of finite rank as $\bb Z$-module; 
\item $G$ is a semiabelian scheme over $S$, which is an extension of an abelian scheme $B$ over $S$ by a torus $T$ over $S$;
\item $u\colon Y \to G_\LOG$ is a homomorphism of sheaves of abelian groups. 
\end{enumerate}
\end{definition}
We generally write a log $1$-motif $(Y, G, u)$ as $[u\colon Y \to G_\LOG]$.  Associated to a log $1$-motif $[u\colon Y \to G_\LOG]$ over $S$ we have a natural ``monodromy pairing''
\begin{equation}\label{eq:trop_monodromy}
\Span{-,-}\colon X \times Y \to \bb G_\trop; \, (x, y) \mapsto x_\trop(y_\trop), 
\end{equation}
where $x_\trop\colon T_\trop \to \bb G_\trop$ is the induced map from $x$ and $y_\trop$ is the image of $y$ in $G_\trop = T_\trop$ via~$u$. 

A morphism of log 1-motifs is a morphism of triples in the natural way (we remark that any map $G_\LOG \to G'_\LOG$ is induced by a map $G \to G'$). 

\subsubsection{Dual log 1-motif}
Let $[u\colon Y \to G_\LOG]$ be a log 1-motif. Define $B^\vee$ to be the dual abelian variety to $B$, define $T^\vee = \Hom(Y, \bb G_m)$ to be the torus with character lattice $Y$, and set $1 \to T^\vee \to G^\vee \to B^\vee \to 1$ to be the dual semiabelian variety, defined as the sheafification of the group of isomorphism classes of pairs $(F, h)$ where $F$ is a $\bb G_m$-extension of $B$ in the category of sheaves of abelian groups, and $h\colon Y \to F$ is a group homomorphism lifting the composite $Y \to G_\LOG\to B$.  This comes with a natural map $u^\vee\colon X \to G_\LOG^\vee$, and we define the dual log 1-motif to be $[u^\vee\colon X \to G_\LOG^\vee]$.

\begin{definition}\label{def:polarisation_on_motif}
Following \cite[Definition 2.8]{LAV2}, a \emph{polarisation} on a log 1-motif $[u\colon Y \to G_\LOG]$ is a map of complexes
\begin{equation}
h\colon [u\colon Y \to G_\LOG] \to [u^\vee\colon X \to G_\LOG^\vee]
\end{equation}
satisfying the following four conditions, of which the first two will be particularly important for us: 
\begin{enumerate}
\item the induced map $B \to B^\vee$ is a polarisation on $B$; 
\item the induced map $h_{-1}\colon Y \to X$ is injective with finite cokernel;
\item for all non-zero sections $y$ of $Y$, the pairing $\langle h_{-1}(y), y\rangle$ is not trivial at any geometric point of $S$; 
\item the homomorphism $T_\LOG\to  T_\LOG^\vee$ induced by $h_0 \colon G_\LOG\to  G_\LOG^\vee$ coincides with the one induced by $h_{-1} \colon Y \to X$.
\end{enumerate}
The polarisation is called \emph{principal} if both maps $h_0$ and $h_{-1}$ are isomorphisms. 
\end{definition}

\subsubsection{Associated log and tropical abelian varieties}\label{sec:BM_and_the LAV}
Let $[u\colon Y \to G_\LOG]$ be a log 1-motif. We write $G_\LOG^{(Y)}\subseteq G_\LOG$ for the subsheaf of elements with \emph{bounded monodromy}\footnote{This condition is automatic if the log structure on the base $S$ is valuative; for example, if we are working over a discrete valuation ring with standard log structure. }: 
\begin{equation}
G_\LOG^{(Y)}(U) = \{f \in G_\LOG(U) : \text{ the induced function }\phi\colon X \to \bb G_\trop \text{ is bounded by  }Y\}, 
\end{equation}
where $f\colon U \to G_\LOG$ induces a map $X \to \bb G_\trop(U)$ by evaluation at the tropicalisation of $f$ using the equalities \eqref{eqn:describe_X}. Here a function $\phi\colon X \to \bb G_\trop$ is \emph{bounded by} $Y$ if for all $u \in U$ and $x \in X_u$ there exist $y$, $y' \in Y_u$ such that 
\begin{equation}
\langle x, y \rangle \mid \phi(x) \mid \langle x, y' \rangle, 
\end{equation}
where given a monoid $M$ and elements $m, n \in M$, we write $m \mid n$ if there exists $a \in M$ with $n = a + m$. 
A motivation for this condition in terms of deformation theory has been given by Molcho--Wise, see~\cite[\S 4.18]{molcho-wise}.

This brings us to our first example of a log abelian variety: the quotient $G_\LOG^{(Y)}/Y$. In fact every log abelian variety is fibrewise of this form: 
\begin{definition}\label{def:LAV}
Let $S$ be a log scheme.
A \emph{log abelian variety} $A/S$ is a sheaf of abelian groups for the strict \'etale topology  on the category of fs log schemes over $S$, such that 
\begin{enumerate}
\item for all $s \in S$ there exists a polarisable log 1-motif $[u\colon Y \to G_\LOG]$ such that the quotient $G_\LOG^{(Y)}/Y$ is isomorphic to $A_s$;
\item \'etale locally on $S$ there exist a semiabelian group scheme $G/S$, finitely generated free $\bb Z$-modules $X$ and $Y$ (viewed as constant sheaves on $S$), a bilinear form 
\begin{equation}
\langle -,-\rangle  \colon X \times Y \to \bb G_\trop ,
\end{equation}
and an exact sequence 
\begin{equation}
0 \to G \to A \to \Hom(X, \bb G_\trop)^{(Y)}/\bar Y \to 0, 
\end{equation}
where $\bar Y$ is the image of $Y$ in $\Hom(X, \bb G_\trop)^{(Y)}$. 
\end{enumerate}
\end{definition}
We remark that the log 1-motif in (1) above is unique up to unique isomorphism by \cite[Theorem 3.4]{LAV2}. 
\begin{definition}\label{def:trop_AV}
Given a log abelian variety $A/S$, the associated tropical abelian variety is the quotient $A_\trop \coloneqq A/G  = \Hom(X, \bb G_\trop)^{(Y)}/\bar Y = G^{(Y)}_\trop/\bar Y$. 
\end{definition}

\begin{remark}\label{rem:auxillary_data}
From a log abelian variety $A/S$ we obtain various auxiliary data. First, a semiabelian group scheme $G/S$ together with a strict open immersion $G \to A$, where the log structure on $G$ is strict over $S$, and whose formation is compatible with arbitrary base change. Secondly, we obtain locally constructible sheaves $\bar X$ and $\bar Y$ on the big strict \'etale site on $S$; i.e., they are functors $\cat{LogSch}^\op_S \to \cat{Ab}$ which are sheaves for the strict \'etale topology, and which are representable by strict \'etale algebraic spaces over $S$. Recall that the sheaf $\bar Y$ is defined as the image of $Y$ in $\Hom(X, \bb G_\trop)^{(Y)}$; similarly $\bar X$ is defined as the image of $X$ in $\Hom(Y, \bb G_\trop)^{(X)}$. Over each geometric point of $S$, the stalks of $\bar X$ and $\bar Y$ coincide with the groups $X$ and $Y$ of the associated log 1-motifs from Section \ref{sec:log_motif}. 
There is a canonical pairing $\bar X \times \bar Y \to \bb G_\trop$, and over each geometric point this specialises to the tropical monodromy pairing \eqref{eq:trop_monodromy}. 
\end{remark}

\begin{remark}\label{rem:rank_of_LAV}
The \emph{rank} of a log abelian variety $A/S$ is equivalently either the rank of $\bar X$, the rank of $\bar Y$, or the torus rank of the semiabelian part $G$; this is an upper-semi-continuous function on $S$. 
\end{remark}

\begin{example}
Every abelian variety is a log abelian variety, with $\bar X$ and $\bar Y$ the zero groups. 
\end{example}

\begin{example}
If $[u\colon Y \to G_\LOG]$ is a polarisable log 1-motif then $G_\LOG^{(Y)}/Y$ is a log abelian variety. 
\end{example}

\subsection{Polarisations}\label{sec:log_polarisations}

The theory of dual log abelian varieties is only developed pointwise,
so the theory of polarisations in that context is constructed  via
biextensions (in the sense of \cite[Expose VII, Definition 2.1]{SGA7I}). For consistency, we adopt the same approach in the
classical setting. 
\begin{definition}
A \emph{polarisation} on an abelian scheme $A\to S$ is a symmetric biextension $ P$ of $(A,A)$ by $\bb G_m$ such that, locally on $S$, there exists a relatively ample line bundle $ L$ on $A/S$ such that $ P$ is equal to the biextension $m^* L \otimes p_1^* L^\vee \otimes p_2^* L^\vee \otimes e^* L$. A polarisation induces \footnote{Classically, a polarisation is \emph{defined} to be this morphism, and the biextension $ P$ is constructed by pulling back the Poincar\'e biextension on $A \times A^\vee$ along $(\on{id}, \lambda)$.} a finite morphism $A \to A^\vee$, and the polarisation is \emph{principal} if this is an isomorphism. 
\end{definition}

If $[u\colon Y \to G_\LOG]$ is a log 1-motif with polarisation 
\begin{equation}
h\colon [u\colon Y \to G_\LOG] \to [u^\vee\colon X \to G_\LOG^\vee]
\end{equation}
then there is an induced map of associated log abelian varieties
\begin{equation}
h\colon G_\LOG^{(Y)}/Y \to (G_\LOG^\vee)^{(X)}/X. 
\end{equation}
Recalling that $A \coloneqq G_\LOG^{(Y)}/Y$ is a log abelian variety, the quotient $(G_\LOG^\vee)^{(X)}/X$ is the dual log abelian variety, denoted $A^\vee$; the polarisation is principal if and only if $A \to A^\vee$ is an isomorphism. 
By \cite[Theorem 7.4 (3)]{LAV2} there is a homomorphism 
\begin{equation}
A^\vee \to \ca Ext(A, \bb G_\LOG), 
\end{equation}
which we may compose with the polarisation to obtain a homomorphism 
\begin{equation}
A \to \ca Ext(A, \bb G_\LOG), 
\end{equation}
and hence by \cite[Proposition 2.3]{LAV5} a symmetric $\bb G_\LOG$-biextension on $(A, A)$. 
\begin{definition}\label{def:polarisation_on_LAV}
Following \cite[\S 1.3]{LAV5}, a \emph{polarisation} on a log abelian variety $ A\to S$ is a symmetric biextension $\llb P$ of $(A,A)$ by $\bb G_\LOG$ which on every geometric point comes from a polarisation of the corresponding log 1-motif in the sense described just above. We say $\llb P$ is \emph{principal} if it is so on every geometric fibre over~$S$. 
\end{definition}
In particular, if $A$ is a polarised log abelian variety then there is a natural map of locally constructible sheaves $\bar Y \to \bar X$, and if the polarisation is principal then this map is an isomorphism.

\subsection{Recovering a family of real tori} \label{sec:real_torus} \label{sec:Berk_anal}

The purpose of this section is to explain how to obtain families of real tori from log abelian varieties, and to make a connection with the related point of view of \emph{Berkovich analytification} of abelian varieties.

Let $(A/S, \lambda)$ be a polarised log abelian variety. Let $k$ be a field, and let $P$ be the log point whose underlying scheme is $\on{Spec} k$, and whose log structure is given by $\bb R_{\ge 0} \oplus k^*$ (see Example~\ref{eg:log_str_pt}). Fix a map $P \to S$, and pull back to get a polarised log abelian variety over $P$. 
Let $A_\trop$ be the associated tropical abelian variety. Let $G$ be the associated semiabelian variety over $P$, and let $T$ be the torus of $G$. Write $X=\Hom(T,\bb G_m)$ as usual.

\begin{lemma} \label{torus_from_log_AV}
$A_\trop(P) = \Hom_P(P, A_\trop)$ is naturally a real torus of dimension equal to the rank of $T$. 
\end{lemma}
\begin{proof} We have seen that $A_\trop = T_\trop/ Y$. 

Now 
\begin{equation}
T_\trop(P) =\Hom(P, T_\trop) = \Hom(X, \bb R). 
\end{equation}
On the other hand, $Y$ is a subgroup of $T_\trop(P) = \Hom(X,\bb R)$, and the condition that the polarisation map $ Y \to  X$ be an isomorphism over $\qq$ ensures that $ Y $ is a lattice inside $\Hom(X, \bb R)$. 
\end{proof}
 
 Suppose for the remainder of this section that $S$ is the spectrum of a complete discrete valuation ring $R$ with residue field $k$, and let $P \to S$ be the natural map. In this case $G$ is the connected component of the N\'eron model of the generic fiber of $A/S$. Let $F$ be the fraction field of $R$ and suppose that $A_F$ is a polarised abelian variety over $F$.  Let $\mathbb{F}$ be the completion of an algebraic closure of $F$. In the next lines we will freely use theory and terminology related to Berkovich analytic spaces as developed in, e.g., \cite{berkovich}, \cite{gubler_tropical}, \cite{gubler_can_meas}. 
 
 Let $A^\an$ denote the Berkovich analytification of $A$ over $\mathbb{F}$. The space $A^\an$ naturally contains the real torus $A_\trop(P) = \Hom(X, \bb R)/Y$ constructed above as a subset. This subset is called the \emph{canonical skeleton} of $A^\an$. We denote the canonical skeleton of $A^\an$ by $\Sk(A)$. We have a canonical \emph{tropicalisation homomorphism} $\val \colon A^\an \to \Sk(A)$ and it turns out that $\val$ is a deformation retraction.  
 
 Let $\ca P \to S$ be an integral, projective flat model of $A_F$. Let $\widetilde{\bb F}$ denote the residue field of $\bb F$. There exists a canonical \emph{reduction map} $\mathrm{red}_{\ca P} \colon A^\an \to \ca P_{\widetilde{\bb F}}$. It is surjective, and if $\xi$ is a generic point of $\ca P_{\widetilde{\bb F}}$, there exists a unique $x \in A^\an$ such that $\mathrm{red}_{\ca P}(x) = \xi$. We call this point the \emph{Shilov point} corresponding to $\xi$.
 The space $A^\an$ also naturally contains the set $A(\mathbb{F})$ of $\mathbb{F}$-valued points of $A$.

\section{Toroidal compactifications} \label{sec:toroidal_comp}

The purpose of this section is to present toroidal compactifications of the moduli spaces needed for our arguments, using the language of logarithmic abelian varieties, which will be important for the proofs of our main results. We begin by recalling the classical (non-logarithmic) versions of the relevant moduli problems.

\subsection{Classical moduli problems}

Let $g$ be a positive integer.

\begin{definition}[$A_{g, n}$]
$A_g$ is the algebraic stack whose objects are principally polarised abelian schemes of relative dimension~$g$. 
Given $n \ge 0$, we define $A_{g,n}$ to have objects being tuples consisting of an object $(A/S, P)$ of $A_g$ together with $n$ sections of $A \to S$; in particular $A_{g, 0} = A_g$ and $A_{g,1}$ is the universal abelian variety over $A_g$. 
\end{definition}
Given an abelian scheme $A \to S$ we write $e \colon S \to A$ for the zero section.
\begin{definition}[$N_{g, n}$]\label{def:Ng}
$N_g$ is the algebraic stack whose objects are pairs $(A,  L)$ where
$A$ is a principally polarised abelian variety (ppav) and $  L$ is a
symmetric rigidified line bundle on $A$ representing the
polarisation. More formally, it is a tuple $(A/S,   P,   L, \epsilon ,
\iota)$ where $A/S$ is abelian, $  P$ is a principal polarisation, $
L$ is a relatively ample line bundle on $A/S$, $\epsilon\colon e^*  L
\to   O_S$ is an isomorphism, $\iota\colon [-1]^*  L \to   L$ is an
isomorphism which pulls back along $e$ to the identity, and we require
that $  P \isom m^*  L \otimes p_1^*  L^\vee \otimes p_2^*  L^\vee
\otimes e^*  L$ as rigidified line bundles (or equivalently as
biextensions) on $A \times A$. 

Given an integer $n \ge 0$, we define $N_{g,n}$ to have objects being tuples consisting of an object of $N_g$ together with $n$ sections of $A \to S$; in particular $N_{g, 0} = N_g$ and $N_{g,1}$ is the universal abelian variety over $N_g$. 
\end{definition}
\begin{remark}\label{rem:N_g_comments}
\begin{itemize}
\item[(i)] The stack $A_{g,n}$ from above should not be confused with the stack $A_{g,n}$ appearing in \cite[Definition 4.4]{fc}. In the latter, the ``$n$'' stands for a level structure. 
\item[(ii)] Equivalently, an object of $N_g$ is a pair of a ppav and a symmetric relative theta divisor (a flat relatively ample effective Cartier divisor which defines the principal polarisation).  First, given such a pair $(p\colon A \to S, \Theta)$ we define $  L =   O_A(\Theta) \otimes p^*e^*  O_A(\Theta)^\vee$. Conversely, given $  L$  we know that $p_*  L$ is invertible; locally on $S$ we can choose a trivialising section, yielding an effective Cartier divisor. A different choice of trivialising section will yield the same Cartier divisor, and hence these glue to a global divisor, the theta divisor. 

\item[(iii)] The condition that $  P \isom m^* L \otimes p_1^* L^\vee \otimes p_2^* L^\vee \otimes e^* L$ as rigidified line bundles on $A \times A$ may be equivalently replaced by the condition that $  L^{\otimes 2} \isom \Delta^*  P$ where $\Delta\colon A \to A \times A$ is the diagonal.

\item[(iv)] The forgetful morphism $N_g$ to $A_g$ is finite flat of degree $2^{2g}$, and is an fppf torsor under $A_{g,1}[2]$, the moduli of ppavs together with a 2-torsion point. The torsor is trivial if and only if $g=1$. 

\item[(v)] It is shown in \cite[Propositions 1.1.4 and 1.2.6]{mb_fonct} that $N_g$ has two irreducible components (which meet in characteristic two, and which correspond to the ``parity'' of the theta divisor), and that the set of global sections of $\bb G_m$ on each component is just $\{-1, +1\}$. 

\item[(vi)] By construction, the universal abelian variety $N_{g,1}$ carries a universal polarising line bundle $L$ and a universal theta divisor $\Theta$.

\end{itemize}
\end{remark}

\subsection{Logarithmic moduli problems} \label{sec:log_moduli}

The logarithmic moduli problems are almost identical to the classical
ones, except that we replace abelian varieties by log abelian
varieties, and $\bb G_m$-torsors by $\bb G_\LOG$-torsors. We warn the
reader that these moduli spaces are generally \emph{not} representable
by schemes or algebraic stacks; for this we must subdivide the moduli
problems, as explained in Section
\ref{sec:representable_subdivisions}. 

Let $g$ be a positive integer.
\begin{definition}
We define $A_g^\LOG$ to be the stack of principally polarised log abelian varieties (pplav), i.e., the fibred category over $\cat{LogSch}$ with objects $(A/S, \llb P)$, where
\begin{enumerate}
\item $A/S$ is a log abelian variety of relative dimension~$g$ (see Definition~\ref{def:LAV});
\item $\llb P$ is a principal polarisation on $A/S$ (see Definition~\ref{def:polarisation_on_LAV}). 
\end{enumerate}

Given $n \ge 0$, we define $A_{g,n}^\LOG$ to have objects being tuples consisting of an object of $A_g^\LOG$ together with $n$ sections of $A \to S$; in particular $A^\LOG_{g, 0} = A_g^\LOG$ and $A_{g,1}^\LOG$ is the universal log abelian variety over $A_g^\LOG$. 
\end{definition}

\begin{definition}\label{def:Nlog}
We define $N_g^\LOG$ to be the fibred category over $\cat{LogSch}$ with
objects $(A/S, \llb P, \llb L, \epsilon , \iota)$, where
\begin{enumerate}
\item $A/S$ is a log abelian variety of relative dimension~$g$;
\item $\llb P$ is a principal polarisation; 
\item $\llb L$ is a $\bb G_\LOG$-torsor on $A$; 
\item $\epsilon\colon e^*\llb L \to \bb G_\LOG$ is a trivialisation along the zero section;
\item $\iota\colon [-1]^*\llb L \to \llb L$ is an isomorphism which pulls back along $e$ to the identity;
\item\label{cond:rigi} we require that $m^*\llb L \otimes p_1^*\llb L^\vee \otimes p_2^*\llb L^\vee \otimes e^*\llb L$ is isomorphic, as rigidified log line bundles on $A \times A$, to $\llb P$. 
\end{enumerate}

Given $n \ge 0$, we define $N_{g,n}^\LOG$ to have objects being tuples consisting of an object of $N_g^\LOG$ together with $n$ sections of $A \to S$; in particular $N^\LOG_{g, 0} = N^\LOG_g$ and $N^\LOG_{g,1}$ is the universal log abelian variety over $N^\LOG_g$.
\end{definition}

Observe that $N^\LOG_{g,1}$ carries a universal $\bb G_\LOG$-torsor $\llb L$. Over $N_{g,1}$ it coincides with the universal polarising line bundle $  L$.

For $g \ge 2$ neither $A_g^\LOG$ nor $N_g^\LOG$ is representable by an algebraic stack with log structure, but this becomes true after suitable subdivision (as constructed in Section~\ref{sec:representable_subdivisions}). 

\begin{lemma}\label{2_tors_action}
The natural forgetful morphism $N_g^\LOG \to A_g^\LOG$ is a torsor under the 2-torsion $A^\LOG_{g,1}[2]$. Both $A_{g,1}[2]$ and $N_{g}$ are representable by finite flat algebraic spaces over $A_g$. 
\end{lemma}
\begin{proof}
The theorem of the cube \cite[Theorem 2.2]{LAV5} implies that if a $\bb G_\LOG$-torsor $\llb L$ on a log abelian variety $A/S$ with principal polarisation $\llb P$ satisfies condition \eqref{cond:rigi} in Definition~\ref{def:Nlog} then the same holds for any translate of $\llb L$. In particular, $A_{g,1}[2]$ acts by translations on the set of symmetric $\llb L$ satisfying \eqref{cond:rigi}; in other words, there is an action of $A_{g,1}[2]$ on $N_{g}$. Another application of the theorem of the cube implies (using symmetry of $\llb L$) that $[2]^*\llb L \cong \llb L^{\otimes 4}$, and hence $\Delta^*\llb P \cong \llb L^{\otimes 2}$. The action of $A_{g,1}[2]$ is therefore transitive. As it is clearly free, we find that $N_g$ is an $A_{g,1}[2]$-torsor. 
The fact that $A_{g,1}[2]$ is a finite flat algebraic space is \cite[Proposition 18.1]{LAV4}. This property transfers to any torsor. 
\end{proof}

\subsection{Representable subdivisions of logarithmic moduli problems}\label{sec:representable_subdivisions}

\subsubsection{Representable subdivisions of $A_g^\LOG$}
In Sections \ref{sec:subdiv_N_g} and \ref{sec:universal_Ng} we will
construct representable subdivisions of $N_g^\LOG$ and
$N_{g,1}^\LOG$. For context we first recall (and slightly generalise)
the analogous constructions from \cite{LAV7} for $A_g^\LOG$. Fix a
positive integer $g$ and $W$ a free finitely generated abelian group of rank $g$. Let
$S(W)$ be the free abelian group of  symmetric integral valued bilinear forms
$W \times W \to \zz$. We write $S_{\qq}(W)=S(W)\otimes_{\zz}\qq$ and
$S_{\rr}(W)=S(W)\otimes_{\zz}\rr$.   

\begin{definition}[Admissible Fan] \label{def:admissible_fan}
An admissible fan $\Sigma$ of $S_\rr(W)$ is a set of strongly convex
rational polyhedral cones in $S_\rr(W)$
satisfying the following conditions:
\begin{enumerate}
    \item For $\sigma \in \Sigma$, every face of $\sigma$ is in $\Sigma$.
    \item For $\sigma, \tau \in \Sigma$, the intersection $\sigma \cap
      \tau$ is a face of $\sigma$. 
    \item $\Sigma$ is stable under the action of
      $\operatorname{Aut}(W)$. Here, $\alpha \in
      \operatorname{Aut}(W)$ acts on $S_\rr(W)$ by $b
      \mapsto b\left(\alpha^{-1}(\cdot),
        \alpha^{-1}(\cdot)\right)$. 
    \item The number of $\operatorname{Aut}(W)$-orbits in $\Sigma$ is finite.
    \item For any $\sigma \in \Sigma$, any element of $\sigma$ is positive semi-definite, i.e., $b(w, w) \geq 0$ for any $b \in \sigma$ and any $w \in W$.
    \item The union of the $\sigma \in \Sigma$ is exactly the space of positive semi-definite symmetric bilinear forms with rational radical. 
\end{enumerate}
\end{definition}

Recall the sheaves $\bar X$ and $\bar Y$ attached to a pplav (see Remark \ref{rem:auxillary_data}).
\begin{definition}[{\cite[Definition~1.3]{LAV7}}]\label{def:A_g_compatibility}
A pplav $(A/S, \llb P)$ is \emph{compatible with $\Sigma$} if
strict-\'etale locally on $S$ there exists a surjective homomorphism
of sheaves of abelian groups $f\colon W \to \bar Y$ such that for all
geometric points $s$ of $S$, there exists $\sigma \in \Sigma$ such
that for all monoid homomorphisms $h\colon \ghost_{S,s} \to \bb N$ (we
think of these as points in the dual cone), the composition 
\begin{equation}\label{eq:sigma_compatibility_LAV}
b\colon W \times W \xrightarrow{f \times f} \bar Y_s \times \bar Y_s \to  \bar X_s \times \bar Y_s \xrightarrow{\langle \cdot, \cdot \rangle} \ghost_{S,s}^{\text{gp}} \xrightarrow{h^{\text{gp}}} \mathbb{Z}
\end{equation}
lies in $\sigma\cap S(W)$; here the map $\bar Y_s \times \bar Y_s \to  \bar X_s
\times \bar Y_s$ is given by the identity on the second factor, and
the isomorphism $\bar Y \isom \bar X$ induced by the principal
polarisation on the first factor.

\end{definition}
\begin{remark}\label{rem:integral_redundant_A}
If we simply required that $b$ lies in $\sigma$ instead of $\sigma\cap S(W)$ this would be equivalent, as $b$ is evidently integral-valued. 
\end{remark}

\begin{definition}
Given an admissible fan $\Sigma$, we write $A_g^\Sigma$ for the full subcategory of $A_g^\LOG$ consisting of tuples where $(A/S, \llb P)$ is compatible with $\Sigma$. 
\end{definition}

\begin{theorem}\label{thm:A_g}
$A_g^\Sigma$ is a proper algebraic stack with log structure. It is log
smooth. It is moreover smooth if and only if $\Sigma$ is smooth, that
is, if all cones in $\Sigma$ are unimodular (i.e., they are generated
by part of a $\zz$-basis of $S(W)$).
\end{theorem}
\begin{proof}
For a prime $p \ge 3$, the fact that the stack $A_g^\Sigma[p]$ of pplav with full level $p$ structure is a proper log smooth algebraic space away from $p$, is proven in \cite[Theorem 1.7]{LAV7}. By \cite[Proposition 18.1]{LAV4} the map $A_g^\LOG[p] \to A_g^\LOG$ is representable by algebraic spaces with log structure, and is \'etale away from $p$, hence the same holds for the map $A_g^\Sigma[p] \to A_g^\Sigma$. By choosing distinct primes $p$, $q \ge 3$ we obtain an \'etale cover of $A_g^\Sigma$ by algebraic spaces with log structure, hence $A_g^\Sigma$ is an algebraic stack with log structure, and it is proper since this holds away from $p$ and away from $q$. We finish by noting that a log smooth stack with unimodular cones is smooth since it is smooth-locally modelled on a smooth toric variety. 
\end{proof}

\begin{remark}\label{rem:cone-stack-Ag}
The cone complex $\Sigma$ carries a natural action of $\on{Aut}(W)$, and a quotient $\Sigma/\on{Aut}(W)$ in a suitable category might be called a \emph{ moduli space of tropical abelian varieties}, see for example \cite{brannetti2011tropical,brown2025bordifications,assaf2025tropicalizations}. 
\end{remark}

\subsubsection{Representable subdivisions of $N_g^\LOG$}\label{sec:subdiv_N_g}

We first generalize the definition of admissible cone decompositions, using freely ideas from \cite[\S IV.7]{fc}. 

Let $R$ be a ring, and let $G, H$ be two $R$-modules.
We say a map of sets $a\colon G \to H$  is \emph{quadratic} if the map $G \times G \to H; (x,y) \mapsto a(x+y) - a(x) - a(y) + a(0)$ is bilinear. If $H = R$ then we call $a \colon G \to R$ a \emph{quadratic function} on $G$. 

Given a free finitely generated abelian group $W$, and an element
$\rho \in W$, we define $Q_\rho(W)$ to be the set of $\zz$-valued quadratic
functions $a$ on $W$ satisfying $a(\rho - w)  = a(w)$, modulo
constants. This is a free finitely generated abelian group. If $\rho -
\rho' \in 2W$ then there is a canonical isomorphism
\begin{equation}\label{eq:iso-quadratic}
Q_\rho(W) \to Q_{\rho'}(W); \quad [a] \mapsto [x \mapsto a(x - (\rho' - \rho)/2)]. 
\end{equation} 

If $f\colon W \to W'$ is a surjection then pullback of functions induces an injective group homomorphism  
\begin{equation}\label{eq:Q_rho_pullback}
f^*\colon Q_{f(\rho)}(W') \to Q_\rho(W). 
\end{equation}

There is a natural map 
\begin{equation} \label{eq:natural_map}
Q_\rho(W) \to S(W); \quad [a] \mapsto [(x,y)  \mapsto a(x+y) - a(x) - a(y) + a(0)]. 
\end{equation}

\begin{lemma} \label{lem:iso_2}
After tensoring both sides over $\bb Z$ with $\bb Z[1/2]$, the map \eqref{eq:natural_map} becomes an isomorphism. Equivalently, the map \eqref{eq:natural_map} is injective, and has finite cokernel of order a power of $2$. 
\end{lemma}
\begin{proof}
Given $b \in S(W)$, we can define a quadratic function $a_b\colon W \to \bb Z[1/2]$ by $a_b(w) = \frac{1}{2}(b(w,w)-b(\rho,w))$. We have $a_b \in Q_\rho(W)$, and $a_b$ maps to $b$ by \eqref{eq:natural_map}. 
\end{proof}
Lemma~\ref{lem:iso_2} also shows that we may as well define $Q_\rho(W)$ as the set of quadratic functions $a$ on $W$ satisfying $a(\rho - w)  = a(w)$ and $a(0)=0$. This is the definition we will use below. The condition $a(0)=0$ is motivated by the fact that we are dealing with \emph{rigidified} line bundles. 

Inside $Q_{\rho}(W)_{\rr}\coloneqq
  Q_{\rho}(W)\otimes_{\zz}\rr$ we have the cone $C_{\rho}(W)$  of
  quadratic functions whose associated symmetric
bilinear form is positive semi-definite with rational radical. We will denote
$C_{\rho}(W)_{\zz}=C_{\rho}(W)\cap Q_{\rho}(W)$ for the monoid of
lattice points inside $C_{\rho}(W)$. The pullback map (\ref
{eq:Q_rho_pullback}) restricts to a pullback 
$f^*\colon
C_{f(\rho)}(W')_{\zz} \to C_{\rho}(W)_{\zz}$.

\begin{definition}\label{def:N_admissible}
An \emph{$N_g$-admissible cone decomposition} consists
of, for every free abelian group $W$ of rank $\le g$, and every $\rho
\in W$, a locally-finite decomposition $\Sigma_{W, \rho}$ of
$C_\rho(W)$ into strongly convex rational polyhedral cones, such that the face of a cone in $\Sigma_{W, \rho}$ lies in $\Sigma_{W, \rho}$, such that the intersection of two cones is a face of both, and such that whenever $f\colon W' \to W$ is a surjection of free abelian groups of rank $\le g$, and $\rho' \in 2W'$, the pullback of the cone decomposition for $(W, \rho)$ along the maps from \eqref{eq:iso-quadratic} and \eqref {eq:Q_rho_pullback} is equal to the cone decomposition for $(W', f(\rho) + \rho')$. 
\end{definition}

\begin{remark}
\begin{enumerate}
\item
An $N_g$-admissible cone decomposition seems to consist of a large amount of data, but this is not really the case -- it is completely determined by a cone decomposition of $C_\rho(\bb Z^g)$ as $\rho$ ranges over a set of coset representatives for $\bb Z^g/2\bb Z^g$. 
\item
When we talk about smoothness of an $N_g$-admissible cone decomposition, we always mean with respect to the integral structure induced by the groups $Q_\rho(W)$. It is explained in \cite[\S IV.7]{fc} how to pull back an admissible cone decomposition for $A_g$ to produce an $N_g$-admissible cone decomposition, but note that this process will not in general take smooth decompositions to smooth decompositions. On the other hand, any $N_g$-admissible cone decomposition can be refined to a smooth one. 
\item Analogously to Remark \ref{rem:cone-stack-Ag}, from an $N_g$-admissible cone decomposition one can build a \emph{tropical moduli stack of abelian varieties with theta level structure} as a colimit in a suitable category of the cone complexes
  $\Sigma_{W, \rho}^\circ$ along the maps $f^*$, where $\Sigma_{W, \rho}^\circ$ denotes the open sub-cone-complex of  $\Sigma_{W, \rho}$ where all cones contained entirely in the boundary are  omitted. This is a disjoint union of two tropical Shimura varieties in the sense of \cite{assaf2025tropicalizations}. 
\end{enumerate}
\end{remark}

Next we want to generalise Definition \ref{def:A_g_compatibility} to this setting. First, given a tuple $(A/S, \llb P, \llb L, \epsilon, \iota)$ in $N_g^\LOG$, we construct a section of $\bar X/2\bar X$ where $\bar X$ is the sheaf from Remark \ref{rem:auxillary_data}. We do this on geometric points, leaving to the careful reader to check that it is compatible with specialisation and therefore leads to a global section. On a geometric fiber we have a semiabelian scheme $1 \to T \to G \to B \to 0$ and a map $c\colon Y \to T_\LOG$, with  $A = G_\LOG /Y$ (see Section \ref{sec:log_motif}). We pull back $\llb L$ to $T[2]$, obtaining a $\bb G_\LOG$-torsor on $T[2]$. Since $[-1]$ restricts to the identity on $T[2]$ the map $\iota$ can be seen as a section of $\bb G_\LOG$ over $T[2]$, and the relation between $\llb L$ and the biextension $\llb P$ means that it is actually a character, i.e., a group homomorphism $T[2] \to \bb G_\LOG$. Now for all positive integers $n$ we have
\begin{equation}
\Hom(T[n], \bb G_\LOG)  =\Hom(T_\LOG[n], \bb G_\LOG) = \Hom(T_\LOG, \bb G_\LOG)/n\Hom(T_\LOG, \bb G_\LOG),
\end{equation}
and recalling that $X = \Hom(T, \bb G_m) = \Hom(T, \bb G_\LOG)$ by \eqref{eqn:describe_X}, the homomorphism $T[2] \to \bb G_\LOG$ is exactly an element of $\bar X/2\bar X$ as required. 
\begin{definition}\label{def:L_to_rho}
Given $(A/S, \llb P, \llb L, \epsilon, \iota)$ in $N_g^\LOG$, we define $\kappa$ to be the section of $\bar X/2\bar X$ constructed above. 
\end{definition}

Now suppose that we are given a tuple $(A/S, \llb P, \llb L,\epsilon,\iota)$ in $N_g^\LOG$ and an $N_g$-admissible cone decomposition $\Sigma$ as in Definition~\ref{def:N_admissible}. 

\begin{definition}\label{def:N_g_compatibility}
We say that $(A/S, \llb P, \llb L)$ is \emph{compatible with $\Sigma$} if locally on $S$ there exists a surjective homomorphism of sheaves of groups $f\colon W \to \bar Y$ and an element $\rho \in W$ such that $\lambda(f(\rho)) \bmod 2\bar X = \kappa$ and 
such that for all geometric points $s$ of $S$, there exists $\sigma \in \Sigma_{W, \rho}$ such that for all monoid homomorphisms $h\colon \ghost_{S,s} \to \bb N$ (we think of these as points in the dual cone), the composition 
\begin{equation}\label{eq:sigma_compatibility_LAV_N}
b\colon W \times W \xrightarrow{f \times f} \bar Y_s \times \bar Y_s \xrightarrow{\lambda \times \text{id}} \bar X_s \times \bar Y_s \xrightarrow{\langle \cdot, \cdot \rangle} \ghost_{S,s}^\gp \xrightarrow{h^{\text{gp}}} \mathbb{Z}
\end{equation}
is such that $\frac12(b(w, w) - b(\rho, w))$ lies in $\sigma\cap
Q_{\rho}(W)$. Here $\lambda \colon \bar Y \to \bar X$ is the isomorphism induced from the principal polarisation.
\end{definition}

\begin{remark}In fact, the quadratic function $\frac12(b(w, w) - b(\rho, w))$ is automatically integral valued, so the condition that $\frac12(b(w, w) - b(\rho, w))$ lie in $\sigma\cap Q_{\rho}(W)$ may equivalently be replaced simply by requiring that $\frac12(b(w, w) - b(\rho, w))$ lies in $\sigma$. This takes a little more work to prove than the analogous statement for $A_g$ in Remark \ref{rem:integral_redundant_A}, but we take some time to explain this since it also sheds light on the shape of the condition.

The point is to see how a log abelian variety together with a symmetric polarising log line bundle induces a quadratic function on $\bar Y_s$ (whereas a polarisation induces a symmetric bilinear form on $\bar Y_s$). We work over a geometric point for simplicity, so that we may assume our log abelian variety $A$ has constant degeneration, hence admits a presentation as the quotient of a log 1-motif $[u\colon Y \to G_\LOG]$, and we set $X = \on{Hom}(G_\LOG, \bb G_\LOG)$ so that $A^\trop = \on{Hom}(X, \bb G_\trop)^{(Y)}/Y$ (see Definition \ref{def:trop_AV}); we set $G_\trop = \on{Hom}(X, \bb G_\trop)^{(Y)}$. Then the log biextension $\frak P$ on $A \times A$ induces a tropical biextension $\frak P^\trop$ on $A^\trop$, which pulls back to a tropical biextension on $G_\trop$, and the latter is \emph{uniquely} trivial by \cite[\S 5 and Proposition A.7]{wise_monodromy}. 

Since our tropical biextension is obtained by pullback from $A^\trop \times A^\trop$ it also comes with trivialisations along $G_\trop \times Y$ and $Y \times G_\trop$ (which agree on $Y \times Y$). Comparing these trivialisations yields a function $Y \times G_\trop \to \bb G_\trop$, which corresponds to an additive map $\lambda\colon Y \to X$. Composing with the standard pairing $Y \times X \to \bb G_\trop$ yields the bilinear form of \eqref{eq:sigma_compatibility_LAV_N}. Equivalently, a descent datum for $\frak P^\trop$ from $G_\trop \times G_\trop$ to $A^\trop \times A^\trop$ with respect to the unique biextension trivialisation of the pullback is given by a function $p\colon Y \times Y \times G_\trop \times G_\trop \to \bb G_\trop$, and $\lambda$ is determined by the formula $h(\lambda(y)) = p(y,0;0,h)$. 

Suppose now that we are given a rigidified log line bundle $\llb L$ on $A$ together with an isomorphism of biextensions
\begin{equation}
\delta(\llb L) \coloneqq m^*\llb L \otimes p_1^*\llb L^\vee \otimes  p_2^*\llb L^\vee \otimes e^*\llb L \to \llb P, 
\end{equation}
descending to an isomorphism $\delta(\llb L^\trop) \to \llb P^\trop$ on $A^\trop \times A^\trop$. Pulling back $\tlb L\coloneqq \llb L^\trop$ to $G_\trop$ and choosing a trivialisation yields a biextension trivialisation of $\delta(\tlb L)$ (which we have previously noted is unique, and indeed one can check directly that a different trivialisation of $\tlb L$ gives the same trivialisation of $\delta(\tlb L)$). After fixing such a trivialisation, the descent datum for $\tlb L$ to $A^\trop$ is given by a function $a\colon Y \times G_\trop \to \bb G_\trop$ satisfying the cocycle relation $a_{y+z}(h)=a_z(h) + a_y(h+z)$, and a small calculation yields the relation
\begin{equation}
\label{eq:p_vs_a}
p(y_1, y_2;h_1,h_2)
=
a(y_1 + y_2, h_1 + h_2) - a(y_1, h_1) - a(y_2, h_2) + a(0, 0). 
\end{equation}
A different trivialisation of the pullback of $\tlb L$ to $G_\trop$ necessarily differs by addition of a function $c\colon G_\trop \to \bb G_\trop$. Any such function is automatically linear, and changes the cocycle $a$ to $a(y,h) + c(u(y))$; in particular this is compatible with the relation \eqref{eq:p_vs_a}. Another small calculation shows that $a(0,h) = 0$, and that $q(y) \coloneqq a(y,0)$ defines a function $Y\to \bb G_\trop$ whose ``second difference'' is the bilinear form from before: 
\[
q(y_1+y_2)-q(y_1)-q(y_2) + q(0)
= \langle \lambda(y_1),y_2\rangle;
\]
 this makes $q$ into a quadratic function. Suppose that we in addition fix a symmetry isomorphism $\iota\colon [-1]^*\llb L \to \llb L$, and require that the trivialisation of the pullback of $\tlb L$ to $G_\trop$ is compatible with $\iota$; this means exactly that there exists $r \in Y$ such that $q(r - y) = q(y)$. In Definition \ref{def:N_g_compatibility} it is really the function $q\colon Y\to \bb G_\trop$ which we are requiring to take integer values, but this is automatic. 
\end{remark}

\begin{definition}
We write $N_g^\Sigma$ for the full subcategory of $N_g^\LOG$ whose objects are compatible with $\Sigma$. 
\end{definition}

With these definitions in hand, we claim:
\begin{theorem}\label{thm:N_g}
Fix an $N_g$-admissible cone decomposition $\Sigma$. Then the fibred category $N_g^\Sigma$ over $\cat{LogSch}$ whose objects are log abelian varieties of dimension~$g$ together with a principally-polarising rigidified symmetric line bundle which is compatible with $\Sigma$, forms a proper algebraic stack with log structure. It is smooth over $\bb Z[1/2]$ if and only if the cone decomposition $\Sigma$  is smooth. 
\end{theorem}
\begin{proof}
From Lemma \ref{2_tors_action} we know $N_g^\LOG$ is relatively representable by finite algebraic spaces over $A_g^\LOG$. Now $A_g^\LOG$ is proper (by \cite[Theorem 1.6]{LAV7} for the case with level structure, or by Theorem \ref{thm:A_g} in general), hence $N_g^\LOG$ is proper. The same then holds for any subdivision, and it is clear from Definition \ref{def:N_g_compatibility} that an $N_g$-admissible cone decomposition is a (complete) subdivision. Representability and (log) smoothness can be checked locally, and locally $N_g^\Sigma$ is a root stack (of degree some power of $2$) over a fibre product $A_g^{\Sigma'}\times_{A_g^\LOG} N_g^\LOG$. Here we use the isomorphism of Lemma \ref{lem:iso_2} to transfer $\Sigma$ to a cone decomposition for $A_g$. These constructions preserve representability by algebraic stacks, and both smoothness and log smoothness outside characteristic 2. 
\end{proof}

\begin{example}[Comparison to modular curves]\label{eg:example_genus_one}

In this example we assume the characteristic not to be~$2$. In genus one it is well-known that $A_1 = Y_1(1)$ (the ``$j$-line'') can be compactified to a modular curve $X_1(1)$ whose coarse moduli space is $\bb P^1$. Moreover, the torsor $N_1$ over $A_1$ under the 2-torsion of the universal elliptic curve is trivialised by the subspace of odd theta characteristics (which give a section $A_1 \to N_1$); this is very different from what happens in higher genus. This gives us an identification between $N_1$ and the 2-torsion in the universal elliptic curve; we obtain $N_1 = Y_1(1) \sqcup Y_1(2)$. This has a unique smooth compactification $X_1(1) \sqcup X_1(2)$. 

In genus one the spaces $A_1^\LOG$ and $N_1^\LOG$ are already representable; no non-trivial subdivisions exist for dimension reasons. Moreover, they coincide with $X_1(1)$ and with $X_1(1) \sqcup X_1(2)$ respectively. In Figure \ref{fig:genus_1} we illustrate the tropical and logarithmic moduli spaces, using colour coding to demonstrate which boundary divisors correspond to which rays of the tropicalisation. We also include in our picture the Clemens complexes of these snc log schemes, to illustrate that the tropicalisation $N_1^\trop$ does \emph{not} coincide with the Clemens complex. 

We refer to \cite[Section~1.2]{mb_fonct} for more details, and the case that the characteristic is~$2$.

\begin{center}
\begin{figure}
\begin{tikzpicture}[scale = 1]
\draw[thick](0,0) to (3,0);
\draw[thick](2.75,0) node{$\bullet$};
\draw[thick](0,1) to (3,1);
\draw[thick,red](2.75,1) node{$\bullet$};
\draw (-0.6,0) node{$A_1^{\log}$};
\draw (-0.6,2.8) node{$N_1^{\log}$};
\draw [thick] plot [smooth, tension=1.5] coordinates { (3,2) (0,2) (2.75,3) (0,4)};
\draw[thick,blue](2.75,1.95) node{$\bullet$};
\draw[thick,blue](2.75,3) node{$\bullet$};
\draw[thick,->](1.5,0.8) to (1.5,0.2);
\draw (1,-1) node{\emph{log moduli spaces}};

\draw[thick] (4,0) node{$\longrightarrow$};
\draw[thick] (4,2.8) node{$\longrightarrow$};

\draw(5,3) node{$\bullet$};
\draw[thick,blue](5,3) to (6.5,3.5);
\draw[thick,blue](5,3) to (6.5,2.5);
\draw(5,2) node{$\bullet$};
\draw[thick,red](5,2) to (6.5,2);
\draw(5,0) node{$\bullet$};
\draw[thick](5,0) to (6.5,0);
\draw[thick,->](5.5,1.3) to (5.5,0.6);
\draw (5.7,-1) node{\emph{Clemens complex}};

\draw[thick] (7.5,0) node{$=$};
\draw[thick] (7.5,2.8) node{$\longrightarrow$};

\draw (8.5,3) node{$\bullet$};
\draw[thick,blue](8.5,3) to (10,3.5);
\draw[thick, violet](8.5,3) to (10,2.5);
\draw (8.5,0) node{$\bullet$};
\draw[thick](8.5,0) to (10,0);
\draw (10.8,0) node{$A_1^{\trop}$};
\draw (10.8,2.8) node{$N_1^{\trop}$};
\draw (10.5,-1) node{\emph{tropical moduli spaces}};
\draw[thick,->](9.2,1.7) to (9.2,0.6);

\end{tikzpicture}
\caption{Log and tropical moduli in genus one}
\label{fig:genus_1}
\end{figure}
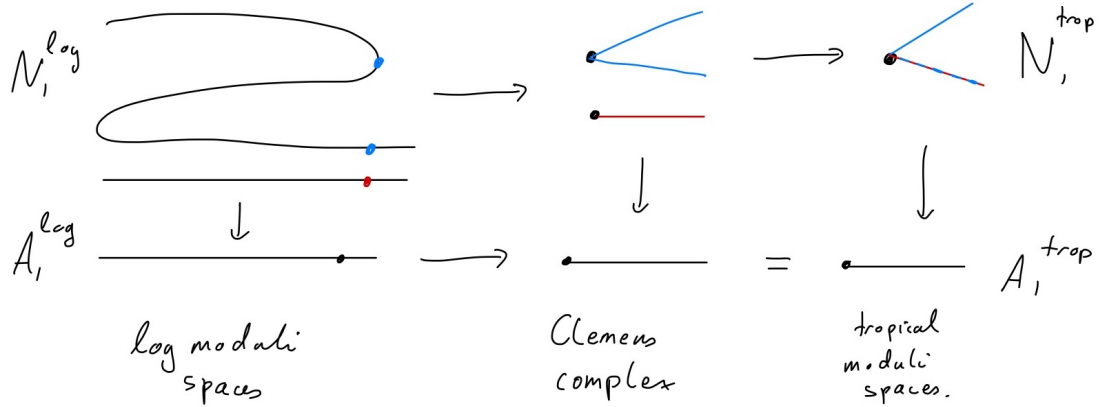
\end{center}

\end{example}

\subsubsection{A representable subdivision of $A_{g,1}^\LOG$}

Here we explain how to use the ideas from \cite{LAV7} to construct a toroidal compactification of $A_{g,1}$ as a subdivision of $A_{g,1}^\LOG$. This may be seen as a translation into the language of log abelian varieties of some of the constructions in \cite[Chapter VI]{fc}.

We let $\tilde S(W) =  S(W) \times \Hom(W, \zz)$.  We write $\tilde S_{\qq}(W) = \tilde
S(W) \otimes_{\zz}\qq $ and $\tilde S_{\rr}(W) = \tilde
S(W) \otimes_{\zz}\rr $.
An element 
$(b,l)\in \tilde S_{\rr}(W)$ is called \emph{positive semi-definite} if $b$
is positive semidefinite with rational radical and $l$ vanishes on the
radical of $b$.

\begin{definition}[Admissible Fan]\label{def:admissible_fan_tilde}
An admissible fan $\widetilde \Sigma$ in $\tilde S_\rr(W)$ is a set of
strongly convex rational  polyhedral cones in $\tilde S_\rr(W)$
satisfying the following conditions: 
\begin{enumerate}
    \item For $\sigma \in \widetilde \Sigma$, every face of $\sigma$ is in $\widetilde \Sigma$.
    \item For $\sigma, \tau \in \widetilde \Sigma$, the intersection
      $\sigma \cap \tau$ is a face of $\sigma$.
    \item $\widetilde \Sigma$ is stable under the action of
      $\operatorname{Aut}(W)\ltimes W$. Here, $\alpha \in
      \operatorname{Aut}(W)\ltimes W$ acts on $\tilde
      S_\rr(W)$ as follows: if $f \colon W \isom W$ then $f(b, l) =
      (b \circ f^{-1}, l \circ f^{-1})$, and if $w \in W$ then $w(b,
      l) = (b, v \mapsto l(v) + b(w,v))$.  
    \item The number of $\operatorname{Aut}(W)\ltimes
      W$-orbits in $\widetilde \Sigma$ is finite. 
    \item For any $\sigma \in \widetilde \Sigma$, any element of $\sigma$ is positive semi-definite. 
    \item For each positive semi-definite $(b,l)\in \tilde S_{\rr}(W)$
      there exists a unique $\sigma \in \widetilde \Sigma$ for which $(b, l)$
      is contained in the interior of $\sigma$. 
\end{enumerate}
\end{definition}

Let $\Sigma $ be an admissible fan in $S_{\rr}(W)$ and $\widetilde \Sigma
$ an admissible fan in $\tilde S_{\rr}(W)$. We say that $\tilde
\Sigma $ and $\Sigma $ are compatible if for every cone $\tilde \sigma
\in \widetilde \Sigma $ there exists a cone $\sigma \in\Sigma $ such that
the image of $\tilde \sigma $ under the projection $\tilde
S_{\rr}(W)\to S_{\rr}(W)$ is contained in $\sigma $.
We fix a pair of compatible  admissible fans $\widetilde \Sigma$ and
$\Sigma $ in $\tilde S_\rr(W)$ and $S_{\rr}(W)$. Below we will define what it means for a triple $(A/S, \llb P, a)$, where $(A/S, \llb P$) is a pplav and $a\colon S \to A$ a section, to be compatible with $\widetilde \Sigma$. Recall from Definition~\ref{def:LAV} that locally on $S$ we have an exact sequence 
\begin{equation}
0 \to G \to A \to \Hom(X, \bb G_\trop)^{(Y)}/\bar Y \to 0,
\end{equation}
where $G$ is the semiabelian part of $A$ and $\Hom(X, \bb G_\trop)^{(Y)}$ is those homomorphisms bounded by $Y$ in the sense of Section \ref{sec:BM_and_the LAV}. The map $a\colon S \to A$ induces by composition an $S$-point of $\Hom(X, \bb G_\trop)^{(Y)}/\bar Y$. Locally we can choose a lift to an $S$-point of $\Hom(X, \bb G_\trop)$, in other words, to a map $X \to \ghost_S^\gp$. The principal polarisation gives us an isomorphism $\lambda\colon Y \isom X$, and so composing gives a map $ a^* \colon Y \to \ghost_S^\gp$. We can use this to build, for any surjection $W \xrightarrow{f} Y $ and any monoid homomorphism $h\colon \ghost_{S,s} \to \bb N$,   a map 
\begin{equation}\label{eq:map_from_point}
W \xrightarrow{f} Y  \xrightarrow{a^*} \ghost^\gp_{S, s} \xrightarrow{h^\gp} \bb Z. 
\end{equation}

\begin{definition}[{\cite[\S 2.3]{LAV4}}]\label{def:univ_oompatibility}
We say \emph{$(A/S, \llb P, a)$  is compatible with $\widetilde \Sigma$} if for all geometric points $s$ of $S$, on some strict \'etale neighbourhood of $s$ there exists a surjection $f\colon W \to Y$ and a cone $\tilde \sigma \in \widetilde \Sigma$ such that for all monoid homomorphisms $h\colon \ghost_{S,s} \to \bb N$:
\begin{enumerate}
\item \label{item:1} the composition 
\begin{equation}\label{eq:quad_form}
b\colon W \times W \xrightarrow{f \times f} Y \times Y \xrightarrow{\lambda \times \text{id}} X \times Y \xrightarrow{\langle \cdot, \cdot \rangle} \ghost_{S,s}^\gp \xrightarrow{h^{\text{gp}}} \mathbb{Z}
%W \times W \xrightarrow{f \times f} \bar Y_s \times \bar Y_s \xrightarrow{\phi \times \text{id}} \bar X_s \times \bar Y_s \xrightarrow{\langle \cdot, \cdot \rangle} \ghost_{S,s}^\gp \xrightarrow{h^{\text{gp}}} \mathbb{Z}
\end{equation}
lies in the image of $\tilde \sigma$ in $S_\bb Q(W)$ (this is the same condition as in Definition \ref{def:A_g_compatibility}); 
\item writing $l\colon W \to \bb Z$ for the map defined
  in \eqref{eq:map_from_point}, the pair $(b, l)$ lies in $\tilde \sigma$. 
\end{enumerate}
\end{definition}
It is automatic that the composite function in
\eqref{eq:map_from_point} vanishes on the radical of the quadratic
form defined in \eqref{eq:quad_form}.  

\begin{remark} The above ``compatibility'' is the same as the condition in \cite[\S 2.3]{LAV4}.
\end{remark}

\begin{definition}
Given $\widetilde \Sigma$ as above, the category of triples $(A/S, \llb P, a)$ compatible with $\widetilde \Sigma$ is denoted $A_{g,1}^{\widetilde \Sigma}$. 
\end{definition}
If $\Sigma$ and $\widetilde \Sigma $ are compatible,  then
there is a natural map $A_{g,1}^{ \widetilde \Sigma} \to
A_g^\Sigma$ thanks to condition (\ref{item:1}) in Definition
\ref{def:univ_oompatibility}.

\begin{theorem}\label{thm:rep_of_A_g1}
Fix a compatible pair $\Sigma $ and $\widetilde \Sigma$. The morphism $A_{g,1}^{\widetilde \Sigma} \to A_g^\Sigma$ is representable by algebraic spaces, log smooth, proper, and over the interior coincides with the universal abelian variety. 
\end{theorem}
\begin{proof}
This follows from \cite[Theorem 8.1]{LAV4}. 
\end{proof}

\begin{remark}
In \cite[\S VI]{fc} the stack $A_{g,1}^{\widetilde \Sigma}$ is denoted $\bar Y$. 
\end{remark}

\subsubsection{A representable subdivision of $N_{g,1}^\LOG$}\label{sec:universal_Ng}

Assume $W$ and $\rho \in W$ are given as above. Given $a \in Q_\rho(W)$, we let $b_a$ denote the symmetric bilinear form associated to $a$ via the map
\begin{equation} 
Q_\rho(W) \to S(W); \quad [a] \mapsto [(x,y)  \mapsto a(x+y) - a(x) - a(y) + a(0)]. 
\end{equation}
That is, $b_a$ is the unique symmetric bilinear form $W \times W \to \zz$ satisfying $a(w) = b_a(w- \rho, w) /2$. 

We define $\tilde C_\rho(W)
\subset Q_\rho(W)_{\rr} \times \Hom(W, \rr)$  to be the subspace of
pairs $(a,l)$ such that $b_{a}$ is positive
semi-definite with rational radical, and $l$ vanishes on the radical
of $b_{a}$. We put $\tilde C_\rho(W)_{\zz}=\tilde C_\rho(W)\cap
(Q_\rho(W) \times \Hom(W, \bb \zz))$.

Given $ w\in W$ we define a translation
\begin{equation}\label{eq:W_translation}
\tilde C_\rho(W) \to \tilde C_\rho(W); \, (a,l) \mapsto (a, v \mapsto l(v) + b_a(v, w))
\end{equation}
where $b_a$ is the bilinear form associated to $a$. If $f \colon W \to W'$ is a surjection then we further define 
\begin{equation}\label{eq:Q_rho_pullback_univ}
f^*\colon \tilde C_{f(\rho)}(W') \to \tilde C_\rho(W); \, (a, l) \mapsto (a \circ f, l \circ f). 
\end{equation}
Both the translation and the pull-back map are compatible with the
integral structure.  

\begin{definition}\label{def:N_admissible_univ}
An \emph{$N_{g,1}$-admissible fan} consists of, for every free abelian
group $W$ of rank $\le g$, and every $\rho \in W$, a locally-finite
decomposition $\widetilde \Sigma_{W, \rho}$ of $\tilde C_\rho(W)$
 into strongly convex rational polyhedral cones, such that every face
 of a cone in $\widetilde \Sigma_{W, \rho}$ lies in $\widetilde \Sigma_{W,
   \rho}$, and such that the intersection of two cones is a face of
 both. We require that whenever $f\colon W \to W'$ is a surjection of
 free abelian groups of rank $\le g$, and $\rho' \in 2W'$,  the
 pullback of the cone decomposition for $(W, \rho)$ along the map
 (\ref{eq:Q_rho_pullback_univ}) is equal to the cone decomposition
 for $(W', f(\rho) + \rho')$. Moreover, given any $W$, $\rho$  and any
 $w \in W$, we require $\widetilde \Sigma_{W, \rho}$ to be invariant under
 the translation by $w$ of \eqref{eq:W_translation}.  
\end{definition}

\begin{example}
In Section \ref{sec:choosing-decomposition} we construct explicit $N_{g,1}$-admissible fans by using ``shifted'' analogues of Namikawa's mixed Voronoi-Delaunay decomposition.
\end{example}

\begin{definition}\label{def:Ng1_compatibility}
We define compatibility between an object of $N_{g,1}^\LOG$ and a cone decomposition $\widetilde \Sigma$ analogously to Definitions \ref{def:N_g_compatibility} and  \ref{def:univ_oompatibility}. We write $N_{g,1}^{\widetilde \Sigma}$ for the full subcategory of $N_{g,1}^\LOG$ consisting of those tuples compatible with $\widetilde \Sigma$. 
\end{definition}

Given an $N_{g}$-admissible cone decomposition $\Sigma $ and an
$N_{g,1}$-admissible cone decomposition $\widetilde \Sigma $ in $\tilde C_{\rho
}(W)$, we say that $\Sigma $ and $\widetilde \Sigma $ are compatible  if
for every $W$ and $\rho $ and every  cone $\tilde \sigma 
\in \widetilde \Sigma _{W,\rho }$ there exists a cone $\sigma \in\Sigma_{W,\rho } $ such that
the image of $\tilde \sigma $ under the projection $\tilde
C_{\rho}(W)\to C_{\rho}(W)$ is contained in $\sigma $. Given a pair $\Sigma, \widetilde \Sigma $ of
compatible admissible cone decompositions, there is a natural map
$N_{g,1}^{\widetilde \Sigma} \to N_g^\Sigma$. 

\begin{theorem} Let $\Sigma$ and $\widetilde \Sigma$ be a pair of
  compatible admissible cone decompositions. 
The morphism $N_{g,1}^{\widetilde \Sigma} \to N_g^\Sigma$ is representable
by algebraic spaces, log smooth, proper, and over the interior coincides with the
universal abelian variety. 
\end{theorem}

\begin{proof}
This can be checked locally on $N_g^\Sigma$, where it follows from Theorem \ref{thm:rep_of_A_g1} since locally these spaces are isomorphic. 
\end{proof}

\subsection{Extending the polarising line bundle}

\subsubsection{Admissible polarisation functions for $N_{g,1}$}\label{sec:universal-theta}

The next definition is a small variation on \cite[Definition VI.1.5]{fc}. We fix an $N_{g,1}$-admissible cone decomposition $\widetilde \Sigma$. 
\begin{definition}\label{def:Ng-adm-pola}
An \emph{$N_{g,1}$-admissible principal polarisation function $\phi $ with
  respect to $\widetilde \Sigma$} is a continuous function $\phi_{W,
  \rho}\colon \tilde C_\rho(W)\to \bb R$ for each finitely
generated free abelian group $W$ and element $\rho \in W$ such that  
\begin{enumerate}
\item 
$\phi_{W, \rho}$ has integral values in $\tilde C_\rho(W)_{\zz}$ and is
linear on each cone of $\widetilde \Sigma$;  
\item the $\phi_{W, \rho}$ are compatible with the isomorphisms $C_\rho(W) \isom C_{\rho'}(W)$ whenever $\rho-\rho' \in 2W$; 
\item $\phi_{W, \rho}$ is concave; 
\item given a surjection $f\colon W \to W'$, we have $\phi_{W,\rho} f^* = \phi_{W', f(\rho)}$, where $f^*$ is the pullback defined in \eqref{eq:Q_rho_pullback_univ};
\item given $w \in W$, the function $\tilde C_\rho(W) \to \bb R; (a, l) \mapsto \phi(a,l) - \phi(a,  l+ b_a(w, -))$ is equal to the function sending $(a, l)$ to $a(w)+l(w)$. Here $b_a$ is the symmetric bilinear form associated to $a$.
\end{enumerate}
\end{definition}
\begin{remark} \label{rem:extending_L} The condition that each $\phi_{W, \rho}$ has integral values in
  $\tilde C_\rho(W)_\zz$ implies that the polarisation function $\phi $
  defines, by applying ``Mumford's construction'', a line bundle $\overline{L}$ over $N_{g,1}^{\widetilde \Sigma}$ that
  extends the universal polarising line bundle $L$ over
  $N_{g,1}$. Note that, in \cite[Definition VI.1.5]{fc} it is only
  required that each $\phi_{W, \rho}$ has rational values with bounded
  denominators. This would imply that $\phi$ determines an
  extension of some power $L^{\otimes m}$ of $L$, or what is the same,
  an extension of $L$ as a $\Q$-line bundle. We require integer values
  because this will allow us to use the cohomology of the extended
  line bundle and because we will see in Section~\ref{sec:tropical_RT}
  that such functions exist. 
\end{remark}
\begin{remark} The term ``principal'' in the above definition refers to the fact that we
  are extending the universal line bundle $L$ that determines the
  principal polarisation.
\end{remark}

\begin{lemma}\label{lem:tropical-section}
Let $\phi$ be an $N_{g,1}$-admissible principal polarisation function with respect to $\widetilde \Sigma$. Then $-\phi$ is a section of the tropicalisation $\llb L^\trop$ of the universal log line bundle $\llb L$ on $N_{g,1}^{\widetilde \Sigma}$.
\end{lemma}
\begin{proof} 
In \cite[Prop. 1.6(2)]{LAV5}  a characterisation is given for a function $\tilde C(W) \to \bb R$  to give a section of the tropicalisation of $\llb L^{\otimes 2}$; it is the same as our definition of an $N_{g,1}$-admissible principal polarisation function, except that it restricts to the case $\rho = 0$ (in other words, it is the more classical notion of an $A_{g,1}$-admissible principal polarisation function, as found in \cite[VI.1.5]{fc}), and the cocycle $a(w)+l(w) = \frac12b_a(w, w) - \frac12b_a(\rho, w) + l(w)$ in (5) is replaced by $b_a(w, w) + 2l (w)$; in other words, it is multiplied by $2$ and the term $- \frac12b_a(\rho, w)$ is omitted. 

The bundle $(\llb L^\trop)^{\otimes 2}$ has exactly $2^{\on{rk} W}$ square roots, corresponding to the $2^{\on{rk} W}$ choices of $\rho' \in W/2W$. If we work with rational coefficients, the cocycle $b_a(w, w) + 2l (w)$ for $(\llb L^\trop)^{\otimes 2}$ is uniquely divisible by 2, yielding the cocycle $\frac12b_a(w, w)  + l(w)$. However, this cocycle does not take integral values (unless $\rho = 0$), and so it cannot be a cocycle for $\llb L^\trop$. However, the cocycle $a(w)+l(w) =\frac12b_a(w, w) - \frac12b_a(\rho, w) + l(w)$ does take integral values, hence defines a tropical line bundle on $N_{g,1}^{\widetilde \Sigma}$. If we double this cocycle it yields $b_a(w, w) - b_a(\rho, w) + 2l(w)$ which differs from $b_a(w, w) + 2l (w)$ by a global linear integral-valued function, hence these cocycle conditions define the same tropical line bundle; this shows that $\frac12b_a(w, w) - \frac12b_a(\rho, w) + l(w)$ is a cocycle for a square root of $(\llb L^\trop)^{\otimes 2}$. 

We know that there are $2^{\on{rk} W}$ square roots of $(\llb L^\trop)^{\otimes 2}$, and we have written down $2^{\on{rk} W}$ tropical line bundles that square to $(\llb L^\trop)^{\otimes 2}$; their cocycles differ by global linear functions that do not take integral values, hence these are genuinely distinct tropical bundles. It remains to verify that the cocycle condition for $\llb L^\trop$ is given by the choice $\rho = \rho'$, but this follows from the compatibility in the last line of Definition \ref{def:N_g_compatibility}. 
\end{proof}

\begin{definition} \label{def:line_bundle_for_adm_bis} Assume given an $N_{g,1}$-admissible principal polarisation function $\phi$. Following Definition~\ref{def:line_bundle_for_adm}, we write $\lb L_\phi$ for the inverse image of the global section of $\llb L^\trop$ determined by $\phi$ under the natural map $\llb L \to \llb L^\trop$. This is a line bundle on  $N_{g,1}^{\widetilde \Sigma}$. Over $N_{g,1}$ it coincides with the universal polarising line bundle~$L$.
\end{definition}

\begin{remark} The line bundle $\lb L_\phi$ is equal to the relatively ample invertible sheaf $\bar L$ on $N_{g,1}^{\widetilde \Sigma}$ obtained by applying ``Mumford's construction'' to $L$ and the admissible principal polarisation function $\phi$, as referred to in Remark~\ref{rem:extending_L}. 
\end{remark}

\section{Pure extensions} \label{sec:proof_main}
In this section we construct the promised pure weight~$2$ extension of the universal polarising line bundle on $N_{g,1}$ using the tropical theta function $\theta^{\rm{sm}}$, and relate it to the universal biextension. We continue with the setting and notation of Section~\ref{sec:toroidal_comp}. 

\subsection{Pure extension of the universal polarising line bundle}

Let $m$ be a positive integer. We have a ``multiplication-by-$m$'' map 
\[ [m]\colon \tilde C_{\rho}(W) \to \tilde C_{\rho}(W) \]
defined by $(a, l) \mapsto (a, ml)$ that respects the integral structure. 

Suppose we are given an $N_{g, 1}$-admissible cone decomposition
$\widetilde\Sigma$, and an associated admissible polarisation function
$\phi$ (see Definition~\ref{def:Ng-adm-pola}). We define
$\widetilde\Sigma_m$ to be the cone decomposition of $\tilde C_{\rho}(W)$
obtained by pulling back $\widetilde\Sigma$ along $[m]$ (i.e., taking the
inverse image of each cone; these again form an admissible cone
decomposition), and  taking the minimal common refinement with $\tilde
\Sigma$. We define
$[m]^*\phi = \phi \circ [m]$, which is an admissible principal
polarisation function for $\widetilde\Sigma_m$.  

Since $\widetilde\Sigma_m$ is a subdivision of $\widetilde\Sigma$, we obtain a birational map $\tau\colon N_{g,1}^{\widetilde\Sigma_m} \to N_{g,1}^{\widetilde\Sigma}$. Next we have the ``multiplication-by-$m$'' map 
\[
[m]\colon N_{g,1} \to N_{g,1}
\]
 sending $(A,  L, x) \mapsto (A,  L, mx)$. 
Let $\llb L$ be the universal log line bundle on $N_{g,1}^\LOG$ (see Section~\ref{sec:log_moduli}).
Recall the resulting line bundle $\lb L_\phi$ on $N_{g,1}^{\widetilde\Sigma}$ from Definition~\ref{def:line_bundle_for_adm_bis}.
\begin{lemma}
The map $[m]\colon N_{g,1} \to N_{g,1}$ extends uniquely to a map 
\begin{equation}\label{eq:mult_by_N}
[m]\colon N_{g,1}^{\widetilde\Sigma_m} \to N_{g,1}^{\widetilde\Sigma}. 
\end{equation}
The log line bundle $[m]^*\llb L$ has a continuous conical section given by $[m]^*\phi$, inducing a line bundle $\lb L_m$ on $ N_{g,1}^{\widetilde\Sigma_m}$. We have a natural isomorphism of line bundles $[m]^*(\lb L_\phi)  = \lb L_m$. 
\end{lemma}
\begin{proof}
If $(A, \llb L, x)$ is compatible with $\widetilde\Sigma$ (in the sense of Definition~\ref{def:Ng1_compatibility}) then $(A, \llb L, mx)$ is compatible with $m\widetilde\Sigma$, showing that the map \eqref{eq:mult_by_N} exists. The uniqueness is clear. For the second part, tropicalisation commutes with pullback so there is a natural isomorphism
\begin{equation}
[m]^*(\llb L^\trop)\cong ([m]^*\llb L)^\trop, 
\end{equation}
and the two line bundles we want to identify are both obtained by pulling back $[m]^*\phi$ to the log line bundle $[m]^*\llb L$. 
\end{proof}
In the following we let $L$ denote the universal polarising line bundle on $N_{g,1}$.

\begin{lemma}\label{lem:log_biext}
Let $\llb L$ be the universal log line bundle on $N_{g,1}^\LOG$.
There is a unique isomorphism of log line bundles $[m]^*\llb L\isom \tau^* \llb L^{\otimes m^2}$ compatible with the isomorphism $[m]^*L  \isom  L^{\otimes m^2}$ on $N_{g,1}$. This induces an isomorphism $[m]^*(\llb L^\trop) \isom (\llb L^\trop)^{\otimes m^2}$. 
\end{lemma}

\begin{proof}Using that $\llb L$ is symmetric and rigidified, the statement for $\llb L$ follows from the theorem of the cube \cite[Theorem 2.2]{LAV5} just as in the classical case. The tropical version is an immediate consequence. 
\end{proof}

\begin{definition} \label{def:dual_quadratic} Let $V$ be a finite dimensional real vector space. Let $b \colon V \times V \to \rr$ be a symmetric bilinear form. The \emph{quadratic form} associated to $b$ is the map $q \colon V \to \rr$ given by $v \mapsto b(v,v)$. 
We have an injective linear map $V/\on{rad}(b) \to V^\vee$ given by $v \mapsto b(v,-)$. Let $\on{co-rad}(b) \subset V^\vee$ be the space of linear functionals on $V$ that vanish on $\on{rad}(b)$. It is clear that $\on{co-rad}(b)$ is equal to the image of $V$ in $V^\vee$ under the map $v \mapsto b(v,-)$. By its identification with $V/\on{rad}(b)$, we have a canonical ``dual'' symmetric bilinear form $b^*$ on the subspace $\on{co-rad}(b) \subset V^\vee$. We write $q^*$ for the associated quadratic form on $\on{co-rad}(b)$. 
\end{definition}

Let $\rho \in W$.  We recall that $\tilde C_\rho(W)$ is the subset of $C_\rho(W) \times W_\rr^\vee$ consisting of pairs $(a,l)$ where $l$ vanishes on the radical of $b_a$.
\begin{definition} \label{def:eq_theta}  Let $q_a$ denote the quadratic form associated to $b_a$, that is $q_a(w)=b_a(w,w)$ for $w \in W$. We define
  $\theta_\rho^{\on{sm}} \colon\tilde C_\rho(W) \to \rr$  to be the
  smooth conical function given by 
\begin{equation}
\theta_\rho^{\on{sm}}(a,l) := -\frac{1}{2}q_a^*(l) + \frac{1}{2}
l(\rho) - \frac{1}{8}q_a(\rho)  , \quad a \in C_\rho(W)_\rr , \, l
\in \on{co-rad}(b_a) . 
\end{equation}
Here $q_a^*$ is the dual quadratic form of $q_a$, as in Definition~\ref{def:dual_quadratic}. We write $\theta^{\on{sm}}$ for the resulting function on $\bigsqcup_{\rho \in W} \tilde C_\rho(W)_\rr$. Note that $\theta_\rho^{\on{sm}}(a,b_a(\rho/2,-))=0$, which explains the insertion of the ``constant'' term $\frac{1}{8}q_a(\rho)$.
\end{definition}
Recall that we are given an $N_{g, 1}$-admissible cone decomposition $\widetilde\Sigma$.
\begin{lemma} \label{lem:cocycle_smooth} The continuous conical function $\theta^{\on{sm}}$ satisfies the cocycle condition
\[ \theta_\rho^{\on{sm}}(a, l) - \theta_\rho^{\on{sm}}(a, l + b_a(w,-))  = a(w) + l(w) , \quad w \in W ,  \, a \in C_\rho(W) , \, l \in \on{co-rad}(b_a) .\]
In particular, the continuous function $-\theta^{\on{sm}}$ is a conical section of the tropicalisation $\llb L^\trop$ of the universal log line bundle $\llb L$ on $N_{g,1}^{\widetilde \Sigma}$. 
\end{lemma}
\begin{proof} We compute
\begin{equation*}
\begin{split}
- \theta_\rho^{\on{sm}}(a,&  l + b_a(w,-))   + \theta_\rho^{\on{sm}}(a, l)\\
 & =  \frac12q_a^*(l + b_a(w,-)) - \frac12l(\rho) - \frac12b_a(w,\rho) - \frac12q_a^*(l) + \frac12l(\rho)\\
& = b_a^*(b_a(w,-),l) + \frac{1}{2}q_a^*(b_a(w,-)) -\frac{1}{2}b_a(w,\rho)\\
& =   l(w)  + \frac{1}{2} q_a(w)  -b_a(\rho/2,w)\\
& = a(w) + l(w).
\end{split}
\end{equation*}
The last statement follows the same argument as in the proof of Lemma~\ref{lem:tropical-section}. 
\end{proof} 

Let $\llb L$ be the universal log line bundle on $N_{g,1}^\LOG$. By Lemma~\ref{lem:cocycle_smooth} we have a canonical log b-line bundle $(\llb L, -\theta^{\on{sm}}) $ over  $N_{g,1}^{\widetilde\Sigma}$.
\begin{theorem} \label{thm:main}
For all $m \in \zz$ there is a unique isomorphism of log b-line bundles
\begin{equation}
[m]^*(\llb L, -\theta^{\on{sm}}) \cong (\llb L, -\theta^{\on{sm}} )^{\otimes m^2}
\end{equation}
over $N_{g,1}^{\widetilde\Sigma}$ compatible with the isomorphism $[m]^* L \cong L^{\otimes m^2}$ over $N_{g,1}$. In other words, the log b-line bundle $(\llb L, -\theta^{\on{sm}}) $ is a pure weight 2 extension of the rigidified line bundle~$L$ in the space of log b-line bundles over $N_{g,1}^{\widetilde\Sigma}$. 

\end{theorem}
\begin{proof}
Since $q_a^*$ is quadratic, the difference of $[m]^*(\llb L, -\theta^{\on{sm}})$ and $(\llb L^{\otimes m^2}, -m^2 \theta^{\on{sm}})$ is equal to $([m]^*\llb L \otimes \llb L^{\otimes -m^2}, (a,l) \mapsto (m^2-m) l(\rho))$. Now $[m]^*\llb L \otimes \llb L^{\otimes -m^2}$ is trivial as log line bundle by Lemma \ref{lem:log_biext}. From the cocycle for $\llb L$ from Definition \ref{def:Ng-adm-pola}(5), 
we compute that the tropicalisation of $[m]^*\llb L \otimes \llb L^{\otimes -m^2}$ is given by the cocycle $$\left[a(mw) + ml(mw) - m^2(a(w) + l(w))\right] = \frac{1}{2}(m^2 - m)b_a(w,\rho);$$ we can verify that this is trivial either because it is the trivialisation of a trivial log line bundle, or because $\frac{1}{2}(m^2 - m)b_a(w,\rho)$ is independent of $l$. The function $(a,l) \mapsto \frac{1}{2}(m^2-m) l(\rho)$ is a global linear integral function whose associated cocycle is $\frac{1}{2}(m^2 - m)b_a(w,\rho)$, hence it determines a constant section of the trivial tropical bundle; it also vanishes at $l = 0$, hence it corresponds to the constant section $0$. This means that the difference $([m]^*\llb L \otimes \llb L^{\otimes -m^2}, (a,l) \mapsto \frac{1}{2}(m^2-m) l(\rho))$ is canonically trivial as log b-line bundle. 
\end{proof}

\begin{theorem} \label{thm:theta-eq-cocycle-new}  Let $\phi$ be any admissible principal polarisation function for $\widetilde\Sigma$, and set $c= \phi - \theta^{\on{sm}}$. For all $w \in W$ we have
\[ c(a, l + b_a(w,-))  = c(a, l), \quad a \in C_\rho(W), \, l \in \on{co-rad}(b_a), \]
so that $c$ is a continuous conical function on $N_{g,1}^{\widetilde \Sigma}$. 
We have an isomorphism of  log b-line bundles $(\llb L, -\theta^{\on{sm}})   \cong (\lb L_{-\phi}^{\LOG},0) \otimes (\oo,c)$ over $N_{g,1}^{\widetilde \Sigma}$. 
\end{theorem}
\begin{proof} Lemma~\ref{lem:cocycle_smooth} shows that $\phi$ and $\theta^{\on{sm}}$ satisfy the same cocycle condition (for the cocycle condition of $\phi$, refer to Definition~\ref{def:Ng-adm-pola}). The invariance property for their difference $c= \phi - \theta^{\on{sm}}$ is then clear.  

As to the second statement, referring to Definition~\ref{def:line_bundle_for_adm} we have isomorphisms of log b-line bundles
\[ (\lb L_{-\phi}^{\LOG},0) \otimes (\oo,c) \cong (\llb L,-\phi) \otimes (\oo,c) \cong (\llb L, -\phi +c) = (\llb L, - \theta^{\on{sm}}). \]
This finishes the proof of the theorem.
\end{proof}

\begin{remark}
The above theorem and its proof generalise naturally to any symmetric (hence pure weight 2) log line bundle and any section of its tropicalisation that is fibrewise pure of weight 2. 
\end{remark}

\subsection{The log b-Poincaré bundle}

Let $S$ be a scheme. On a ppav $A/S$, we recall that the Poincaré bundle $ P$ on $A \times_S A$ is a $\bb G_m$-biextension. We can rephrase this by saying that it comes with isomorphisms
\begin{equation}\label{eq:biext_isos}
\begin{split}
& P|_{A \times e} \isom \ca O_A, \;  P|_{e \times A} \isom \ca O_A, \; \\
&(m\times \on{id})^* P \isom pr_{13}^* P \otimes pr_{23}^* P,\; (\on{id}\times m)^*  P \isom pr_{12}^* P \otimes pr_{13}^* P, 
\end{split}
\end{equation}
and that these isomorphisms satisfy certain compatibilities corresponding to commutativity and associativity, see \cite[Exposés VII and VIII]{SGA7I}. Assume that $A/S$ admits a polarising rigidified line bundle $L$. Letting $\Delta \colon A \to A \times_S A$ denote the diagonal embedding, we have canonical isomorphisms 
\[ P = m^*L \otimes pr_{1}^* L^\vee \otimes pr_2^* L^\vee \, , \quad \Delta^*P = L^{\otimes 2}. \]

Suppose now that $A/S$ is a pplav over a log scheme~$S$, then by \cite{LAV5} it comes with a Poincaré log line bundle $\llb P$. This is a $\bb G_\LOG$-biextension, a property which can be rephrased by writing certain explicit isomorphisms as in \eqref{eq:biext_isos} and corresponding associativity and commutativity axioms. 

This motivates the following definition. 
\begin{definition} \label{def:log_Poinc}
Suppose that $\bar A/S$ is a representable subdivision of a pplav, and that $c$ is a conical section of the tropicalisation of the pullback $\llb P$ of the log Poincaré bundle, so that $(\llb P, c)$ is a log b-line bundle on $\bar A$. We define the structure of a \emph{biextension} on $(\llb P, c)$ to be four isomorphisms 
\begin{equation}\label{eq:log_biext_isos}
\begin{split}
&(\llb P, c)|_{\bar A \times e} \isom (\ca O_{\bar A},0), \; (\llb P, c)|_{e \times {\bar A}} \isom (\ca O_{\bar A},0), \; \\
&(m\times \on{id})^* (\llb P, c) \isom pr_{13}^*(\llb P, c) \otimes pr_{23}^*(\llb P, c),\; \\
&(\on{id}\times m)^* (\llb P, c) \isom pr_{12}^*(\llb P, c) \otimes pr_{13}^*(\llb P, c)
\end{split}
\end{equation}
compatible with the corresponding isomorphisms on $\llb P$ described above, and satisfying the same commutativity and associativity rules as in the classical case. 
\end{definition}
\begin{remark}\label{rk:biext}
A less conceptual but more concise way to phrase Definition~\ref{def:log_Poinc} would be to say that we require $\llb P$ to be a $\bb G_\LOG$-biextension and we require $c$ to satisfy the following equations:
\begin{equation}
\begin{split}
&c(p,0) = 0, \; c(0,q) = 0, \;  \\ 
&c(p+p',q) = c(p, q) + c(p', q), \; c(p, q + q') = c(p,q) + c(p, q'),
\end{split}
\end{equation}
where each equality holds on a suitable pullback of $\llb P$ to a suitable fibre power of $A$. This is what we will check in practice. 
\end{remark}

Let $g$ be a positive integer, let $W$ be a finitely generated free abelian group of rank $\le g$,  and let $\widetilde \Sigma_A$ be an $A_{g,1}$-admissible fan in $\tilde S_\bb Q(W)$ in the sense of Definition~\ref{def:admissible_fan_tilde}. Let $\widetilde \Sigma_N$ be an $N_{g,1}$-admissible fan with a map of cone complexes $\widetilde \Sigma_N \to \widetilde \Sigma _A$ (this can always be arranged by suitable refinements). This yields a proper, generically finite map 
\begin{equation}
\tau\colon N_{g,1}^{\widetilde \Sigma_N} \to A_{g,1}^{\widetilde \Sigma_A}. 
\end{equation}

The log Poincar\'e bundle $\llb P$ pulls back naturally to $A_{g,1}^{\widetilde \Sigma_A}\times_{A_{g}^{ \Sigma_A}}A_{g,1}^{\widetilde \Sigma_A}$. Recall that $N_{g,1}^{\widetilde \Sigma_N}$ carries a universal polarising log line bundle $\llb L$; these are related (after pulling back $\llb P$ to $ N_{g,1}^{\widetilde \Sigma_N}$) by the formula
\begin{equation}\label{eq:poincare_vs_polar}
\llb P = m^* \llb L \otimes p_1^*\llb L^\vee \otimes p_2^*\llb L^\vee \otimes e^*\llb L. 
\end{equation}

Recall from Lemma~\ref{lem:tropical-section} that if $\phi$ is any $N_{g,1}$-admissible polarisation function compatible with $\widetilde \Sigma_N$, then $-\phi$ is a section of $\llb L^\trop$, and $-\theta_\rho^{\on{sm}}(a,l) = \frac{1}{2}q_a^*(l) - \frac{1}{2} l(\rho) + \frac{1}{8}q_a(\rho) $ is a conical section of $\llb L^\trop$. We see using \eqref{eq:poincare_vs_polar} that
\begin{equation}
-\phi(l_1+l_2) + \phi(l_1) +  \phi(l_2) - \phi(0)
\end{equation}
is a global section of $\llb P^\trop$, and that 
\begin{equation}
\begin{split}
\frac12q^*_a(\tilde{l}) & - \frac12\tilde{l}(\rho) - \frac12q^*_a(l_1) + \frac12l_1(\rho) - \frac12q^*_a(l_2) + \frac12l_2(\rho) + \frac12q^*_a(0) \\
& = b_a^*(l_1,l_2),
\end{split}
\end{equation}
where $\tilde{l} \coloneqq l_1 + l_2$,
is a conical section of $\llb P^{\trop}$. Now the section $b_a^*(l_1,l_2)$ of $\llb P^{\trop}$ on $N_{g,1}^{\widetilde \Sigma_N}$ does not depend on $\rho$, and hence descends to a conical section of $\llb P^{\trop}$ on $A_{g,1}^{\widetilde \Sigma_A}$, which we denote simply by $b$ (indeed, it corresponds on the universal cover of the universal tropical abelian variety over $A_g^{\Sigma_A}$ to the universal symmetric bilinear form $b$). 

\begin{theorem}
The universal log b-line bundle $(\llb P, b)$ is a biextension. 
\end{theorem}
\begin{proof}
This is immediate from Remark \ref{rk:biext}, and the fact that $b$ is bilinear. 
\end{proof}
\begin{corollary}
Let $m$ be a positive integer, and write $m\colon A_{g,1} \to A_{g,1}$ for the induced endomorphism. Over $A_{g,1}^{\widetilde \Sigma_A}$ we have natural isomorphisms of log b-line bundles
\begin{equation}
\begin{split}
(m,1)^*(\llb P, b) &\isom (\llb P, b)^{\otimes m},\\
(1,m)^*(\llb P, b) &\isom (\llb P, b)^{\otimes m},\\
(m,m)^*(\llb P, b) &\isom (\llb P, b)^{\otimes m^2}.
\end{split}
\end{equation}
 Let $\Delta \colon N_{g,1} \to N_{g,1} \times_{N_g} N_{g,1}$ denote the diagonal. Over $N_{g,1}^{\widetilde \Sigma_N}$ we have a natural isomorphism of log b-line bundles $\Delta^*(\llb P, b) \isom (\llb L,-\theta^{\on{sm}})^{\otimes 2}$.
\end{corollary}

\section{Tropical theta functions} \label{sec:tropical_RT}

The aim of this section is to give a concrete construction of some of the abstract objects defined in Section~\ref{sec:proof_main}. 

Let $g$ be a positive integer,  and consider the universal abelian variety $N_{g,1}$ over $N_g$. 
We will define a suitable $N_{g,1}$-admissible principal polarisation function $\theta^{\mathrm{pl}}$, called the \emph{$pl$-tropical theta function}. Its linearity locus induces an $N_{g,1}$-admissible cone decomposition, which we will denote by $\widetilde \Sigma^{{\mathrm{mDV}}}$.  It can be seen as an $N_g$-analogue of the classical mixed Delaunay--Voronoi decomposition over $A_g$. By the results in Section \ref{sec:universal_Ng} and Section \ref{sec:universal-theta}, we obtain a log smooth, proper algebraic space $N_{g,1}^{\mathrm{mDV}} = N_{g,1}^{\widetilde{\Sigma}^\mathrm{mDV}} $ containing $N_{g,1}$, and a section of the tropicalisation of the universal log line bundle on $N_{g,1}^{\mathrm{mDV}}$.

\subsection{Tropical theta functions with characteristics}

As in the previous section, let $W$ be a free abelian group of rank $\leq g$ and let $\rho \in
W$. Recall that we have the group $Q_\rho(W)$ of quadratic functions $a \colon W \to \zz$ satisfying $a(\rho-w)=a(w)$ and $a(0)=0$. We have the cone $C_{\rho}(W)$ inside $Q_\rho(W)_\rr$ of
quadratic functions $a$ whose associated symmetric
bilinear form $b_a$ is positive semi-definite with rational radical.  Recall that $b_a$ is the
unique symmetric bilinear form $b \colon W_\rr \times W_\rr \to \rr$ such that
$a(v) = b(v-\rho,v)/2 $ for $v \in W_\rr$. We write $q_a(v)=b_a(v,v)$ for the quadratic form determined by~$b_a$.

\begin{definition}[Tropical theta function with characteristic] \label{def:theta-char} We define the \emph{$pl$-tropical theta function} on $\tilde{C}_\rho(W)$ by
\begin{equation} \label{def:theta_pl}  \theta^{\mathrm{pl}}_\rho(a,l)  \coloneqq \min_{w \in W} \left\{ a(w) + l(w) \right\} , \quad a \in C_\rho(W) , \,  l \in \on{co-rad}(b_a) . 
\end{equation}
Note that it takes integer values on $\tilde C_\rho(W)_{\zz} $. 
The fact that $a$ is bounded below ensures that for each $l \in
\on{co-rad}(b_a)$ the minimum exists. 
We define $\theta^{\mathrm{pl}}$ as the resulting function on $\bigsqcup_{\rho \in W} \tilde{C}_\rho(W)$.
\end{definition}
\begin{remark} The function $\theta_\rho^{\mathrm{pl}} \colon \tilde{C}_\rho(W)\to \rr$ is conical. Moreover, as a minimum of affine functions arising from lattice-indexed terms, $\theta_\rho^{\mathrm{pl}}$ is piecewise affine, continuous, and concave.  The regions of linearity of $\theta_\rho^{\mathrm{pl}} $ are given by a variant of the classical mixed Delaunay--Voronoi decomposition, as we will see in Section~\ref{sec:choosing-decomposition}.
\end{remark} 

\begin{remark} The function $\theta^{\mathrm{pl}}$ is a tropical analogue of the \emph{Riemann theta function with theta characteristics} from the classical theory of analytic theta functions.
\end{remark}

\begin{example} \label{rem:usual_theta}
For $\rho=0$ we obtain the ``usual'' tropical Riemann theta function as in \cite{tropi-theta}:
\[ \theta_0^{\mathrm{pl}}(a,l)  =  \min_{w \in W} \left\{ \frac{1}{2}q_a(w) + l(w) \right\} , \quad a \in C_0(W), \,  l \in \on{co-rad}(b_a). \]
\end{example}

\begin{lemma} \label{lem:cocycle_theta} Let $\rho \in W$. The $pl$-tropical theta function with characteristic $\rho$ in \eqref{def:theta_pl} satisfies the following cocycle condition: for all $v \in W$, for all $a \in C_\rho(W)$ and for all $l \in \on{co-rad}(b_a)$ we have 
\[ -\theta_\rho^{\mathrm{pl}}(a, l + b_a(v,-)) + \theta_\rho^{\mathrm{pl}}(a,l) =  a(v) + l(v) . \]
\end{lemma}
\begin{proof}
A small calculation shows that for $v, w \in W$ we have
\[ a(w-v) = a(w) - b_a(w - \frac{1}{2}v - \frac{\rho}{2},v) . \]

With this we calculate 
\begin{align*}  \theta_\rho^{\mathrm{pl}}  & (a, l + b_a(v,-))  = \\
 & =  \min_{w \in W} \left\{ a(w) + l(w) +b_a(v,w)  \right\}   \\
  & =  \min_{w \in W} \left\{ a(w-v) + l(w-v)  +b_a(v,w-v) \right\}    \\
  & =  \min_{w \in W} \left\{ a(w) - b_a(w - \frac{1}{2}v - \frac{\rho}{2},v) + b_a(v,w) + l(w-v) - b_a(v,v)  \right\}   \\
  & =  \min_{w \in W} \left\{ a(w) - \frac{1}{2}b_a(v,v) + \frac{1}{2}b_a(v,\rho)  + l(w-v) \right\}  \\
 & =    \; \theta_\rho^{\mathrm{pl}}(a,l) - a(v) - l(v) . 
 \end{align*} 
\end{proof}
Let $\rho \in W$. Recall the function $\theta_\rho^{\on{sm}} \colon\tilde C_\rho(W) \to \rr$ from Definition~\ref{def:eq_theta}.
\begin{definition} \label{def:inv_theta}
We define the \emph{invariant tropical theta function with characteristic $\rho$} as the difference
\begin{equation}\label{def:inv-tropi-theta}
\theta_\rho^{\mathrm{inv}}(a, l)  \coloneqq \theta_\rho^{\mathrm{pl}}(a,l) - \theta_\rho^{\on{sm}}(a,l) , \quad a \in C_\rho(W) , \,  l \in \on{co-rad}(b_a) .
\end{equation}
We denote by $\theta^{\mathrm{inv}}$ the resulting function on $\bigsqcup_{\rho \in W} \tilde C_\rho(W)$. 
\end{definition}
It follows from Theorem~\ref{thm:theta-eq-cocycle-new} that $\theta^{\mathrm{inv}}$ is invariant under translations by elements $b_a(v,-)$ for $v \in W$. 
\begin{example} \label{ex:inv_rho_vanish} We continue Example~\ref{rem:usual_theta} and consider the case that~$\rho=0$. We have $\theta_0^{\on{sm}}(a,l) = -\frac{1}{2}q_a^*(l)$. For $l=b_a(v,-)$ with $v \in W_\rr$ we find
\[ \begin{split} \theta_0^{\on{inv}}(a,l) & = \min_{w \in W} \left\{ \frac{1}{2}q_a(w) + l(w) \right\} + \frac{1}{2} q_a^*(l) \\
 & = \min_{w \in W} \left\{ \frac{1}{2}q_a(w) + b_a(v,w) + \frac{1}{2}q_a(v) \right\} \\
 & = \frac{1}{2} \min_{w \in W} q_a(v+w) , 
\end{split} \]
that is $\theta_0^{\on{inv}}(a,l)$ is half the distance-squared, in the semi-norm given by $a$, from $v \in W_\rr$ to the nearest lattice point in $W$.
\end{example}

The above example generalises as follows to arbitrary~$\rho \in W$.
\begin{proposition}  \label{prop:trop_theta_new}  Let $\rho \in W$, let $a \in C_\rho(W)$, and let $l \in \on{co-rad}(b_a)$. The formula
\[  
\theta_\rho^{\mathrm{inv}}(a,l) = \frac{1}{2} \min_{w \in W} q_a^*(l + b_a(w-\rho/2,-))
 \]

holds. In particular, for each given $a \in C_\rho(W)$ we have 
\[
\theta^{\mathrm{inv}}_\rho(a,-)  =  \theta^{\mathrm{inv}}_0(a,-) \circ T_{b_a(\rho/2,-)},
\]
 where for all $m \in W_\rr^\vee$ we denote by $T_m \colon W_\rr^\vee \to W_\rr^\vee$ the translation by~$m$. For each given $a \in C_\rho(W)$, the function $ \theta_\rho^{\mathrm{inv}}(a,l)$ has minimum value equal to zero, and it is attained at $l=b_a(\rho/2,-)$.
\end{proposition}
\begin{proof} We compute
\[ \begin{split}  \theta_\rho^{\mathrm{inv}} & (a,l)  = \min_{w \in W} \left\{ a(w) + l(w) + \frac{1}{2}q_a^*(l) -\frac{1}{2}l(\rho) + \frac{1}{8}q_a(\rho)    \right\} \\
& =  \min_{w \in W} \left\{ \frac{1}{2}q_a(w) - \frac{1}{2}b_a(\rho,w)  + l(w) + \frac{1}{2}q_a^*(l) -\frac{1}{2}l(\rho) + \frac{1}{8}q_a(\rho)   \right\} \\
& =  \frac{1}{2} \min_{w \in W} q_a^*(l + b_a(w,-) -b_a(\rho/2,-)). \end{split} \]

To see the last statement, note that $\theta^{\mathrm{inv}}_0(a,l) \ge 0$, and that $\theta^{\mathrm{inv}}_0(a,0) = 0$. 
\end{proof}

\begin{proposition} \label{prop:independence_mod_2} The function $ \theta_\rho^{\mathrm{inv}}$ only depends on the class of $\rho$ modulo $2\cdot W$. More precisely, we have $\theta^{\mathrm{inv}}_\rho(a,l) = \theta^{\mathrm{inv}}_{\rho'}(a',l)$ when $\rho - \rho' \in 2\cdot W$ and $a'(w) = a(w + (\rho-\rho')/2)$. As a result, the function $\theta^{\mathrm{inv}}$ descends to give a function on $\bigsqcup_{\overline{\rho} \in W/2\cdot W} \tilde C_\rho(W)$.
\end{proposition}
\begin{proof} Note that $b_a = b_{a'}$ and $q_a = q_{a'}$. For $\rho' - \rho = 2v$ with $v \in W$ we have  the equality $ b_a(\rho'/2,-) -b_a(\rho/2,-) =  b_a(v,-)$ in $W_\rr^\vee$. Now use the explicit formula in Proposition~\ref{prop:trop_theta_new}.
\end{proof}

\subsection{Choosing a decomposition}\label{sec:choosing-decomposition}

The goal of this section is to define an $N_{g,1}$-admissible cone decomposition induced from the tropical theta functions with characteristics defined in the previous section. It is a variation of the mixed Delaunay--Voronoi decomposition given by Namikawa in \cite[Section 9]{Na}, see also \cite[Section~1]{an}. 

As before, let $W$ be a free abelian group of rank $\leq g$, and let $\rho \in W$.
The following is a slight variation of the classical definition of a
Delaunay cell corresponding to a quadratic form, cf.\ \cite[Definition
1.2]{an}.

\begin{definition}
Fix an element $a \in C_{\rho}(W)$. Given $l \in W_{\R}^\vee$, we say that $y \in W$ is $l$-\emph{nearest} (with respect to $a$) if 
\[
\theta_\rho^{\mathrm{pl}}(a, l)  = a(y) + l(y).
\]
An $a$-\emph{Delaunay cell} is the closed convex hull in $W_\rr$ of all lattice elements in $W$ which are $l$-nearest for some given $l \in W_{\R}^\vee$.
\end{definition}

For a given $a$-Delaunay cell $D$ in $W_\rr$, a corresponding element $l \in W_{\R}^\vee$ is uniquely determined if and only if $D$ has maximal possible dimension. Furthermore, the collection of $a$-Delaunay cells constitutes a locally finite decomposition $\mathrm{Del}^W_\rho(a)$ of $W_{\R}$ into  convex lattice polytopes. By construction, the decomposition is invariant under translations by the lattice $W$. 

\begin{definition}
For a given $a$-Delaunay cell $D$ consider the set of all $l \in W_{\R}^\vee$ defining $D$. These form a locally closed convex polytope, whose closure is called the $a$-\emph{Voronoi} cell, denoted by $\widehat{D} = V(D)$. The collection of $a$-Voronoi cells constitutes the \emph{Voronoi cell decomposition} of $W_{\R}^\vee$. 
\end{definition}
\begin{remark} Voronoi polytopes are not lattice polytopes, unlike Delaunay polytopes. In fact,  Voronoi decompositions vary continuously with respect to $a$, unlike Delaunay decompositions.
\end{remark} 
\begin{remark}
As in the classical case, the $a$-Delaunay cells and the $a$-Voronoi cells are dual to each other in the sense of \cite[Lemma 1.5]{an}. 
\end{remark}

\begin{definition}\label{def:decomposition}
For $a \in C_{\rho}(W)$ we set $\sigma(a)^{\circ}$ to be the (open) cone defined by
\[
 \left\{a' \in C_{\rho}(W) \; | \; \mathrm{Del}^W_\rho(a) = \mathrm{Del}^W_\rho(a') \right\}
\]
and we define
\[
\sigma(a) \coloneqq \overline{\sigma(a)^{\circ}}.
\]
We further denote by $\Sigma_{\rho}(W)^{{\mathrm{Del}}}$ the collection 
\[
\Sigma_{\rho}(W)^{{\mathrm{Del}}} \coloneqq \left\{\sigma(a) \; | \; a \in C_{\rho}(W) \right\}.
\]
\end{definition}
The relative interior of $\sigma(a)$ is $\sigma(a)^{\circ}$, which justifies the notation.

\begin{proposition}\label{prop:adm-cone}
The sets $\Sigma_{\rho}(W)^{{\mathrm{Del}}}$, where $W$  ranges over all free abelian groups of rank $\leq g$ and all $\rho \in W$ form an $N_g$-admissible cone decomposition in the sense of Definition \ref{def:N_admissible}. We denote it by $\Sigma^{{\mathrm{Del}}}$. 
\end{proposition}
\begin{proof}
We need to show:
\begin{enumerate}
\item for each $W$ and $\rho$ as in the statement, the collection $\Sigma_{\rho}(W)^{{\mathrm{Del}}}$ forms a locally-finite decomposition of $C_{\rho}(W)$ into strongly convex rational polyhedral cones satisfying that the intersection of any two cones is a face of both.
\item let $f \colon W \to W'$ be a surjection and let $\rho' \in 2W'$. Recall the map
\[ f^*\colon \tilde C_{f(\rho)+\rho'}(W') \to \tilde C_\rho(W); \quad (a, l) \mapsto (a \circ f, l \circ f). \]
Then 
\[
\left(f^*\right)^{-1}\left(\Sigma_{\rho}(W)^{{\mathrm{Del}}}\right) = \Sigma_{f(\rho)+\rho'}(W')^{{\mathrm{Del}}}.
\]
\end{enumerate}
The first item follows as in the case of $A_g$ (see e.\,g.\,\cite[Theorem 9.9]{Na}, or \cite[Chapter VI, Remark~1.6~b)]{fc}). For the second item, we have to show that $a, a' \in \left(f^*\right)^{-1}(\sigma)$ for $\sigma \in \Sigma_{\rho}(W)^{{\mathrm{Del}}}$ if and only if $a, a' \in \sigma'$ for $ \sigma' \in \Sigma_{f(\rho) + \rho'}(W')^{{\mathrm{Del}}}$.

Assume that $a, a' \in \left(f^*\right)^{-1}(\sigma)$ for some $\sigma \in \Sigma_{\rho}(W)^{{\mathrm{Del}}}$. Then $a \circ f, a'\circ f \in \sigma$ which implies that the Delaunay decompositions $\mathrm{Del}^W_\rho(a \circ f)$ and $\mathrm{Del}^W_\rho(a' \circ f)$ coincide. Let then $l, l' \in W^\vee_{\R}$ be such that 
\begin{align}\label{eq:delaunay}
& \on{ConvHull}\left\{x \in W \; | \; \theta_\rho^{\mathrm{pl}}(a \circ f, l) = (a \circ f)(x) + l(x)\right\} \nonumber \\  &= \on{ConvHull}\left\{x' \in W \; | \; \theta_\rho^{\mathrm{pl}}(a' \circ f, l') = (a' \circ f)(x') + l'(x')\right\}.
\end{align}
Write $f^* \colon (W')^\vee \to W^\vee$ for the dual map (and also for the induced map after tensoring by $\R$). Set $\tilde{l} \coloneqq (f^*)^{-1}(l)$ and $\tilde{l}' \coloneqq (f^*)^{-1}(l')$ in $(W')^\vee_{\R}$.
Applying $f$ to \eqref{eq:delaunay} we obtain the equality of convex hulls
\begin{align}\label{delaunay2}
& \on{ConvHull}\left\{\tilde{x} \in W' \; | \; \theta_{f(\rho)}^{\mathrm{pl}}(a, \tilde{l}) = a (\tilde{x}) + \tilde{l}(\tilde{x})\right\} \nonumber \\  &= \on{ConvHull}\left\{\tilde{x}' \in W' \; | \; \theta_{f(\rho)}^{\mathrm{pl}}(a', \tilde{l}') = a' (\tilde{x}') + \tilde{l}'(\tilde{x}')\right\}.
\end{align}
Hence, by the isomorphism $Q_{f(\rho)}(W') \simeq Q_{f(\rho)+ \rho'}(W')$ in \eqref{eq:iso-quadratic} we conclude that 

\[
\mathrm{Del}_{f(\rho)+ \rho'}^{W'}(a) = \mathrm{Del}_{f(\rho)+ \rho'}^{W'}(a'),
\]
as we wanted to show. Finally, we note that the implications above are actually  equivalences, and we obtain the result. 
\end{proof}

We can now define an $N_{g,1}$-admissible cone decomposition by combining the $N_g$-admissible cone decomposition from Proposition \ref{prop:adm-cone} with the polyhedral decompositions of $W_{\R}^\vee$ induced by the Voronoi decomposition.

\begin{definition}\label{def:decomposition2}

Let $W$ and $\rho$ be as above. 
For a cone $\sigma \in \Sigma_{\rho}(W)^{{\mathrm{Del}}}$ and a Delaunay cell $D = \on{ConvHull}(x_1, \dotsc, x_{\ell})$ with $x_i \in W$, we define the mixed Delaunay--Voronoi cone $\widetilde{\sigma}(D)$ by
\[
 \left\{(a, l) \in \sigma \times W_{\R}^\vee \; | \; \frac{1}{2}b_a(x-x_i, x+x_i-\rho)+ l(x-x_i)  \geq 0 \; \forall  x \in W \; \forall i \in \{1, \dotsc, \ell\} \right\},
\]
where $b_a$ denotes the symmetric bilinear form associated to $a$.  The collection of mixed Delaunay--Voronoi cones $\left\{\widetilde{\sigma}(D)\right\}_{\sigma, D}$ as we vary over all $\sigma \in \Sigma_{\rho}(W)^{{\mathrm{Del}}}$ and all $a$-cells $D$ for $a \in \sigma^{\circ}$ forms a decomposition of $C_\rho(W) \times W_{\R}^\vee$. We denote it by $\widetilde{\Sigma}_{\rho}(W)^{{\mathrm{mDV}}}$.

\end{definition}
\begin{proposition}\label{rem:fiber-voronoi}
Consider the projection $p \colon C_{\rho}(W) \times W_{\R}^\vee \to C_{\rho}(W)$. For each cone $\sigma \in \Sigma_{\rho}(W)^{{\mathrm{Del}}}$ and each corresponding Delaunay cell $D$, the mixed Delaunay--Voronoi cone $\widetilde{\sigma}(D)$ is mapped onto $\sigma$. The fiber over $a \in \sigma^{\circ}$ is exactly the union of $a$-Voronoi cells $\widehat{D} = V(D)$ indexed by $x \in W $. 
\end{proposition}
\begin{proof} This follows from the equivalence
\[
a(x_i) + l(x_i) \leq a(x) + l(x) \Leftrightarrow \frac{1}{2}b_a(x-x_i, x+x_i-\rho)+ l(x-x_i)  \geq 0,
\]
where, again, $b_a$ is the symmetric bilinear form associated to $a$ as above. 
\end{proof}
\begin{proposition}\label{prop:decomposition-over}
Ranging over all free abelian groups $W$ of rank $\leq g$ and $\rho \in W$, the collection $\widetilde{\Sigma}^{{\mathrm{mDV}}}\coloneqq \left\{\widetilde{\Sigma}_{\rho}(W)^{{\mathrm{mDV}}}\right\}$ forms an $N_{g,1}$-admissible cone decomposition over $\Sigma^{{\mathrm{Del}}}$. Moreover, the $pl$-tropical theta function $\theta^{\mathrm{pl}}$ defines an $N_{g,1}$-admissible principal polarisation function with respect to $\widetilde{\Sigma}^{{\mathrm{mDV}}}$.  
\end{proposition}
\begin{proof}
The first part of the proposition follows from Proposition \ref{prop:adm-cone} and Proposition~\ref{rem:fiber-voronoi}.  
It remains to check items (1)--(5) in Definition~\ref{def:Ng-adm-pola}. As to item (1), we have seen above that the regions of linearity of $\theta^{\mathrm{pl}}$ define the $N_{g,1}$-admissible cone decomposition $\widetilde{\Sigma}^{{\mathrm{mDV}}}$. Proposition~\ref{prop:independence_mod_2} implies that item (2) holds. It is clear that item (3) is verified. From Definition~\ref{def:theta-char} one immediately checks that for all surjections $f \colon W \to W'$, for all $\rho \in W$, for all $a \in C_{f(\rho)}(W')$, and for all $l \in \Hom(W',\rr)$ we have the equality $\theta_{W,\rho}^{\mathrm{pl}}(a \circ f, l \circ f) = \theta^{\mathrm{pl}}_{W',f(\rho)}(a,l)$. This verifies item (4). Finally, the $W$-cocycle condition in item (5) is verified in Lemma~\ref{lem:cocycle_theta}. 
\end{proof}
\begin{definition} \label{def:comp_N} We denote by $N_g^{\on{Del}}$ the log smooth, proper algebraic stack corresponding to $\Sigma^{\on{Del}}$. 
Similarly, we denote by $N_{g,1}^{\mathrm{mDV}} $ the log smooth, proper algebraic stack corresponding to $\widetilde{\Sigma}^{{\mathrm{mDV}}}$. We have a natural structure map $\pi \colon N_{g,1}^{\mathrm{mDV}}  \to N_g^{\on{Del}}$.
\end{definition}
\begin{corollary} The function $-\theta^{\mathrm{pl}}$ is a section of the tropicalisation $\llb L^\trop$ of the universal log line bundle $\llb L$ on $N_{g,1}^{\mathrm{mDV}}$.
\end{corollary}
\begin{proof} Given Proposition~\ref{prop:decomposition-over}, this is now a special case of Lemma~\ref{lem:tropical-section}.
\end{proof}

\section{Log b-line bundles and log adelic line bundles} \label{sec:log_b_divisors}

The aim of this section is to connect the language of log b-line bundles, as developed in Section~\ref{sec:b-line-bundles}, with the language of adelic line bundles, as developed by Yuan--Zhang in their book \cite{yz}. The current paper is partly motivated by understanding, from a more explicit point of view, the following special case of a result from loc.\ cit. 
\begin{theorem} (Yuan--Zhang \cite[Theorem~6.1.3]{yz}) \label{thm:YZ_invariant} Let $S$ be a flat and quasi-projective integral scheme over $\zz$. Let $\pi \colon A \to S$ be an abelian scheme, and let $L$ be a symmetric, rigidified line bundle on $A$. There exists an adelic line bundle $\overline{L}$ on $A$ extending $L$ and satisfying $[2]^*\overline{L} \cong \overline{L}^{\otimes 4}$. The adelic line bundle $\overline{L}$ is uniquely determined by the rigidification. Moreover, for all $m \in \zz$ we have an isomorphism $[m]^*\overline{L} \cong \overline{L}^{\otimes m^2}$ of adelic line bundles on $A$.
\end{theorem}
Using our approach we can give an ``explicit'' description of the pure adelic line bundle $\overline{L}$ from Theorem~\ref{thm:YZ_invariant} in the case that the rigidified line bundle $L$ is \emph{relatively ample}, and fiberwise defines a \emph{principal polarisation}. In fact, we can work over the stack $N_{g,1}$.

The first step is to extend the notions of adelic divisors and adelic line bundles as developed in \cite{yz} to so-called \emph{arithmetic stacks}. After that, we add a log structure, and introduce and study the notions of \emph{log adelic divisors} and \emph{log adelic line bundles} on log flat arithmetic stacks with log structure. It turns out that once we restrict to log b-divisors and log b-line bundles with \emph{continuous} conical part, we have natural associated log adelic divisors and log adelic line bundles.

We will see in Theorem~\ref{thm:comparison_adelic_b_AV} that the log b-line bundle $(\llb L, -\theta^{\on{sm}})$ from Theorem~\ref{thm:main} is continuous, and has associated log adelic line bundle essentially equal to Yuan--Zhang's pure adelic line bundle from Theorem~\ref{thm:YZ_invariant}. In Corollary~\ref{cor:specialize_Riemann} we give an ``explicit'' description of Yuan--Zhang's pure adelic line bundle in terms of the tropical theta function with characteristic that we constructed in Section~\ref{sec:tropical_RT}.

\subsection{Adelic divisors and adelic line bundles}\label{sec:adelic-stacks}

The contents of this section are inspired by the language of adelic divisors and adelic line bundles on quasi-projective varieties over a field or over $\zz$, as developed by Yuan--Zhang in \cite{yz}. We extend here their notions to the setting of \emph{arithmetic stacks}. In the next section we discuss the case of arithmetic stacks with log structure. 

Let $R$ be a field or Dedekind domain, and let $\Omega$ be a finite (possibly empty) set of injective ring morphisms from $R$ into $\bb C$. 

\begin{definition} \label{def:arithmetic_stack}
An \emph{arithmetic stack} over $R$ is a normal Deligne--Mumford stack $X/R$ whose structure morphism is flat, separated and of finite type, and such that $X \otimes_\bb Z \bb Q$ is smooth over $\bb Q$. 
\end{definition}
Given an arithmetic stack $X/R$ and an injection $\sigma \in \Omega$, we have an associated complex orbifold $X_\sigma = (X \otimes_{R,\sigma} \bb C)(\bb C)$. 

An \emph{arithmetic ($\Q$)-divisor} on an arithmetic stack $X/R$ is the datum of a ($\Q$)-Cartier divisor $D$ on $X$ and for each $\sigma \in \Omega$ a Green's function $g_{D_\sigma}$ for $D_\sigma$ on $X_\sigma$. Here by a Green's function on a complex orbifold we mean a continuous function on the coarse space minus the support of $D_\sigma$, which \'etale locally has logarithmic singularities along $D_\sigma$. 

Suppose now that $X/R$ is a \emph{proper} arithmetic stack. We fix an arithmetic  \emph{boundary divisor} for $X$, i.e., an arithmetic divisor $\overline{B}= (B,(g_{B_\sigma})_\sigma)$ on $X$ where $B$ is an effective relative Cartier divisor on $X$. 

Let $U=X\setminus |B|$. We denote by $\on{Div}(X)$ resp.\ $\on{Div}(U)$ the groups of arithmetic divisors on $X$ resp.\ $U$, and by $\on{Div}(X)_{\Q}$ resp.\ $\on{Div}(U)_{\Q}$ the groups of arithmetic $\Q$-divisors on $X$ resp.\ $U$. We define $\on{Div}(X,U)$ to be the fibre product of the natural map $\phi \colon \on{Div}(X)_{\Q} \to \on{Div}(U)_{\Q}$ with the natural map $\psi  \colon \on{Div}(U) \to \on{Div}(U)_{\Q}$, i.e.
\[
\on{Div}(X,U) \coloneqq \on{ker}\left(\phi - \psi \; : \; \on{Div}(X)_{\Q} \oplus \on{Div}(U) \to \on{Div}(U)_{\Q}\right).
\]
In other words, $\on{Div}(X,U)$ is the group of pairs $(D, D')$, where $D \in \on{Div}(X)_{\Q}$ and $D' \in \on{Div}(U)$ have equal images in $\on{Div}(U)_{\Q}$. An element $(D,D')$ of $\on{Div}(X,U)$ is called a \emph{$(\Q,\Z)$-arithmetic divisor on} $(X,U)$. 

By construction, there are projection maps 
\begin{equation}\label{eq:mixed-coef}
\on{Div}(X,U) \to \on{Div}(X)_{\bb Q}\; , \quad \on{Div}(X,U) \to \on{Div}(U).
\end{equation}

Abusing notation, we will abbreviate an element $(D, D')$ of $\on{Div}(X,U)$ as $D$, and then write $D|_U$ for $D'$, viewed as an integral divisor on $U$.

We denote by $\mathcal{R}(X, B)$ the category of all proper morphisms $\tau\colon \tilde X \to X$ with $\tilde X$ again an arithmetic stack over $R$ and which are isomorphisms over $U$ and with $\tau^{-1}U$ dense in $\tilde X$. 

We can pull back $(\Q,\Z)$-arithmetic divisors between objects of $\mathcal{R}(X,B)$, and the group of \emph{model adelic divisors} on $(X,B)$ is defined to be the colimit over $\mathcal{R}(X, B)$ of the spaces of $(\Q,\Z)$-arithmetic divisors with maps given by pullback. 

On the set of model adelic divisors one defines the so-called \emph{boundary topology} in analogy with \cite[Section~2.4.1]{yz}. The ball with radius $\eps \in \qq_{>0}$ around a given model adelic divisor $\overline{D} = (D,(g_{D_\sigma}))$ is given by those model adelic divisors $\overline{D'} = (D',(g_{D'_\sigma}))$ satisfying
\[ -\eps \, \overline{B} \le \overline{D} - \overline{D'} \le \eps \, \overline{B}  , \]
that is, we have the conditions
\[ -\eps \, B \le D - D' \le \eps \, B , \quad \textrm{and} \quad  \forall \sigma \in \Omega \, : \, -\eps \, g_{B_\sigma} \le g_{D_\sigma} - g_{D'_\sigma} \le \eps \, g_{B_\sigma}  . \] 
Here when $E, E'$ are $\qq$-divisors we write $E \le E'$ if $E'-E$ is effective. The inequalities of Green's functions are supposed to hold on $U_\sigma(\bb C)$. Note that the first inequality implies that the support of $D - D'$ is contained in that of $B$, and then the second inequality implies that for all $\sigma \in \Omega$ we have that the function $g_{D_\sigma} - g_{D'_\sigma}$ extends continuously over $U_\sigma(\bb C)$.

\begin{definition} \label{def:adelic}
The space of \emph{adelic divisors on $(X,B)$}, denoted by $\operatorname{Ad-Div}(X,B)$, is defined to be the completion of the space of model adelic divisors with respect to the boundary topology.

By \eqref{eq:mixed-coef} we have a natural restriction map $\operatorname{Ad-Div}(X,B) \to \operatorname{Div}(U)$.

Similarly we have a category $\on{Ad-}\!\mathcal{P}\!\on{ic}(X,B)$ of \emph{adelic line bundles on $(X,B)$} as follows. 
An object of $\on{Ad-}\!\mathcal{P}\!\on{ic}(X,B)$ is a pair $(\ca L, (X_i,\ca L_i,\ell_i)_{i \ge 1})$ consisting of
\begin{itemize}
\item[(i)] a line bundle $\ll$ on $U$;
\item[(ii)] objects $X_i$ of $\ca R(X,B)$;
\item[(iii)] hermitian $\qq$-line bundles $\ca L_i$ on $X_i$; 
\item[(iv)] isomorphisms $\ell_i \colon \ca L \isom \ca L_i|_U$ of $\bb Q$-line bundles over $U$.
\end{itemize}
These data are subject to the following \emph{Cauchy condition}. By (iv) for each $i\ge 1$ we have an isomorphism of $\qq$-line bundles $\ell_i \ell_1^{-1} \colon \ca L_1|_U \isom \ca L_i |_U$ and hence a rational map $\ell_i \ell_1^{-1} \colon \ca L_1 \dashrightarrow \ca L_i $ of $\qq$-line bundles on a common modification $Y$ of $X_1$ and $X_i$. This determines a model adelic divisor $\divisor(\ell_i \ell_1^{-1} )$ on $(X, B)$; the Cauchy condition is that the sequence $(\divisor(\ell_i \ell_1^{-1} ))_{i \ge 1}$ is a Cauchy sequence in the boundary topology. 

We refer to \cite[Section~2.5]{yz} for the definition of morphism in the category  $\on{Ad-}\!\mathcal{P}\!\on{ic}(X,B)$; all morphisms are isomorphisms. The groupoid $\on{Ad-}\!\mathcal{P}\!\on{ic}(X,B)$ is equipped with a tensor product and has dual objects. We denote by $\on{Ad-Pic}(X,B)$ the resulting group of isomorphism classes. 
\end{definition}

We have a natural surjective map 
\begin{equation} \label{eqn:taking_O}
 \oo_X \colon \operatorname{Ad-Div}(X,B) \to \on{Ad-Pic}(X,B)  . 
\end{equation}

In fact, the groupoid $\on{Ad-}\!\mathcal{P}\!\on{ic}(X,B)$ and group $\on{Ad-Pic}(X,B)$ only depend on the open part $U= X \setminus |B|$, and one may as well write $\on{Ad-}\!\mathcal{P}\!\on{ic}(U)$ and $\on{Ad-Pic}(U)$, respectively.

We refer to the setting where $\Omega$ is empty as the \emph{geometric case}. We refer to the setting where $R$ is the ring of integers in a number field and $\Omega$ is the full set of embeddings $R \to \bb C$ as the \emph{global arithmetic case}. The setting where $R = \bb C$ and $\Omega = \{\on{id}\}$ is referred to as the \emph{local arithmetic case}. These will be the main cases of interest.

\subsection{Log adelic divisors and log adelic line bundles}

We consider now the case that $X$ is a \emph{log flat algebraic stack with log structure} (see Section~\ref{sec:log_spaces_and_stacks}). We take $U \subseteq X$ to be the largest open dense set where the log structure is trivial. Then $B = X \setminus U$ is a relative effective Cartier divisor, and we take it to be our boundary divisor.  For example, $X$ could be a regular scheme with a simple normal crossings divisor $B$, and then we have a natural log smooth log structure on $X$ with $U = X \setminus B$. 

\begin{definition}  A \emph{log model adelic divisor} on $X$  is an arithmetic $(\Q,\Z)$-divisor in $\on{Div}(X',U)$ for some subdivision  $X'$ of $X$.
\end{definition}
We refer to Section~\ref{sec:subdivisions} for the notion of subdivision of an algebraic stack with log structure. 
A log model adelic divisor on $X$  is really the same thing as a log Cartier b-divisor on $X$ (see Definition~\ref{def:log_b_Cartier}). 

\begin{definition} Let $E$ be an adelic divisor on $(X, B)$. We call $E$ a \emph{log adelic divisor} on $X$ if $E$ is the limit of a Cauchy sequence, in the boundary topology, of log model adelic divisors on $X$. We denote the subgroup of log adelic divisors on $X$ by $\operatorname{Ad-Div}_{\operatorname{log}}(X)$. 
\end{definition} 

\begin{example} \label{rem:special}
For $X$  a toric variety over a field, with $U \subset X$ the dense open torus, Song introduces in \cite{song} the notion of \emph{toric adelic divisor} on $U$. It follows from \cite[Theorem~4.5]{song} that a log adelic divisor on $X$ with the property that the restriction to $U$ vanishes is the same thing as a toric adelic divisor on $U$. When $X$ is smooth, Song moreover introduces in \cite{song_asymptotic} the notion of \emph{special adelic divisor} on $U$. 
By \cite[Theorem~4.8]{song} an adelic divisor on $(X, B)$ is toric (i.e., log adelic on $X$) if and only if it is special. 
\end{example}

Analogous to the notion of adelic line bundle in Section~\ref{sec:adelic-stacks} we have a notion of \emph{log adelic line bundle} on $X$. The only difference is that in item (ii) from Definition~\ref{def:adelic} we now require the $X_i$ to be all given by subdivisions of $X$. The category of log adelic line bundles is denoted by $\on{Ad-}\!\mathcal{P}\!\on{ic}_{\on{log}}(X,B)$, and the group of isomorphism classes of log adelic line bundles is denoted by $\on{Ad-Pic}_{\on{log}}(X)$. Analogous to \eqref{eqn:taking_O} we have a surjective natural map
\[ \oo_X \colon \operatorname{Ad-Div}_{\on{log}}(X) \to \on{Ad-Pic}_{\on{log}}(X) . \]

\subsection{Continuous log b-divisors and log adelic divisors}

We continue with the assumptions and notations from the previous section. In this setting we recall the space $\logwbDiv(X)$ of log Weil-b-divisors on $X$ from Definition~\ref{def:log_b_Cartier}.

Recall that any log model adelic divisor on $X$ can be viewed as a log Cartier b-divisor on $X$. Now let $(D_n)_n$ be a Cauchy sequence of model divisors on $X, B$ in the boundary topology. Then the sequence of corresponding Cartier b-divisors on $X$ converges in the weak topology, by an argument analogous to the one used to prove~\cite[Lemma~3.9]{BK}. Thus we have a natural continuous group homomorphism 
\begin{equation}\label{eq:map-b}
b \colon \operatorname{Ad-Div}_{\on{log}}(X) \to \operatorname{log-W-b-Div}(X)  . 
\end{equation}

Write $U=X \setminus |B|$. By Lemma \ref{lem:con_injective} there is a natural injective map 
\[ \on{Div}(U) \oplus \on{Con}(X) \to \logwbDiv(X)  , \]
where $\on{Con}(X)$ is the group of conical functions on $X$. 

\begin{definition} \label{def:continuous_log_Weil_b}
We call a log Weil b-divisor on $X$  \emph{continuous} if it is in the image of $\on{Div}_X(U) \oplus \on{CCon}(X)$ under this map. Here $\on{Div}_X(U)$ is the group of divisors on $U$ whose ``Zariski closure'' extension to some subdivision of $X$ is Cartier (this is just $\on{Div}(U)$ if $X$ is log smooth), and $\on{CCon}(X)$ is the group of continuous conical functions on $X$.  The group of continuous log Weil b-divisors on $X$ is denoted by $\logwbDiv^{\on{cont}}(X)$. 

\end{definition}

The group of continuous log Weil b-divisors injects into the group of log adelic divisors. In the log smooth case, the two groups are even isomorphic, as the next result states.

\begin{proposition} \label{prop:map_p} There is a natural map $ a \colon \logwbDiv^{\on{cont}}(X) \to \operatorname{Ad-Div}_{\on{log}}(X)$ with the property that $b \circ a = \id$. In particular the map $a$ is injective.  When $X$ is log smooth, the map $b$ has image in $ \logwbDiv^{\on{cont}}(X)$. We then have $a \circ b = \id$ as well, so that $\operatorname{Ad-Div}_{\on{log}}(X)$ and $\logwbDiv^{\on{cont}}(X)$ are naturally isomorphic.
\end{proposition}
\begin{proof}
We define the map $a$ on elements of $\on{Div}_X(U) \oplus \on{CCon}(X)$.  It is clear how to map an element of $\on{Div}_X(U)$ to a Cartier divisor on some model, and the results on different models will be compatible under pushforward. It remains to define the map on $\on{CCon}(X)$, and to prove $b \circ a = \id$. 

Let $f$ be a continuous conical function, and let $(f_i)_{i \in \bb N}$ be a sequence of PL functions converging uniformly to $f$; this exists by Lemma \ref{lem:PL_approx}. For each $i$ let $X_i \to X$ be a subdivision on which $f_i$ is strict. Then the sequence $\on{div}_{X_i}(f_i)$ is a sequence of log model adelic divisors. This is a Cauchy sequence in the boundary topology, and their limit log adelic divisor maps via the map $b$ to the unique log Weil b-divisor with support outside $U$ whose conical function is~$f$. 

To see the last statements, note that it suffices to check the property $a \circ b=\id$  on log adelic divisors $E$ with $E|_U =0$. 
 So let $E$ be a log adelic divisor on $X$ with $E|_U=0$, and set $\E = b(E)$. Let $E_i$ be a Cauchy sequence of log model adelic divisors with $E_i|_U = 0$ converging to $E$ with respect to the boundary norm, and let $ \phi_{E_i}$ be the corresponding PL functions on $X$ from Remark \ref{rem:phi_D} (here we use log smoothness). The Cartier condition for the $E_i$ implies that the functions $\phi_{E_i}$ converge pointwise to some conical function $\phi$, and this $\phi$ is sent by $a$ to $E$. On the other hand, after choosing an atomic cover of $X$ the convergence is uniform on the compact sets $H_i$ from Definition \ref{def:conical_top}, so that $\phi$ is a uniform limit of continuous functions and hence continuous. 
\end{proof}

\begin{definition} \label{def:assoc_Weil}
Let $D \in \on{Div}_X(U)$ with Zariski closure $\overline{D}$ on $X$. Let $c \in \on{CCon}(X)$. We denote the log adelic divisor coming from the pair $(D,c)$ via the map $a$ by $\overline{D} + \divisor c$.
\end{definition}

\begin{remark}\label{rem:log_ad_is_b}
The property that $a \circ b = \id$ might fail without the ``log smooth'' assumption, even in the atomic case. Indeed, a ray in the cone corresponds to some divisor supported on the boundary, but in general this need not be irreducible, and different irreducible components may have different multiplicities in a log adelic divisor. 
\end{remark}

\begin{remark}
Let $X$ be a smooth toric variety with dense open torus $U$. We may think of a log adelic divisor on $X$ whose restriction to $U$ vanishes as a special adelic divisor in the sense of Song \cite{song_asymptotic} (see Remark~\ref{rem:special}). Such a special adelic divisor has an associated \emph{skeletal function}. This skeletal function is, up to a sign,  the associated conical function.
\end{remark}

\subsection{Log b-line bundles and log adelic line bundles}

We continue with the assumptions and notations from the previous sections. We discuss here a line bundle analogue of the comparison between continuous log b-divisors and log adelic divisors that we saw in Proposition~\ref{prop:map_p}.

We recall from Section~\ref{sec:b-line-bundles} that a \emph{log b-line bundle} on $X$ is a pair $(\llb L, c)$ where $\llb L$ is a log line bundle on $X$ and $c$ is a conical section of the tropicalisation of $\llb L$. 
We denote by $\on{b-Pic}^{\on{cont}}(X)$ the group of isomorphism classes of log b-line bundles on $X$ where the conical section $c$ is \emph{continuous}. 
Applying Remark~\ref{rem:Pic_subdivisions} and the map \eqref{eqn:con_to_b-pic} we obtain a natural surjection 
\begin{equation} \label{eqn:map_from_Weil_b_to_b-Pic} 
\oo_X \colon \logwbDiv^{\on{cont}}(X)  \to  \on{b-Pic}^{\on{cont}}(X) . 
\end{equation}
\begin{proposition} \label{prop:commutative_p_bar} 
We have a commutative diagram with natural maps  \begin{equation}
\begin{tikzcd}
 \logwbDiv^{\on{cont}}(X)  \ar[d, "\oo_X"] \arrow[r, "a"] &  \operatorname{Ad-Div}_{\on{log}}(X) \ar[d, "\oo_X"]\\
 \on{b-Pic}^{\on{cont}}(X)     \arrow[r, "\overline{a}"] &  \on{Ad-Pic}_{\on{log}}(X),
\end{tikzcd}
\end{equation}
with $\overline{a}$ injective, and both vertical maps surjective. When $X$ is log smooth the map $\overline{a}$ is an isomorphism.
\end{proposition} 
\begin{proof} We construct a map $\overline{a} \colon  \on{b-Pic}^{\on{cont}}(X) \to  \on{Ad-Pic}_{\on{log}}(X)$ by taking our cue from the proof of Proposition~\ref{prop:map_p}. Let $(\llb L, c)$ be a b-line bundle on $X$ with $c$ continuous. By Lemma~\ref{lem:PL_approx_bundle} we can approximate $c$ uniformly by a sequence of PL sections $c_i$ of $\llb L$. For each $i$ we pick a model $X_i \to X$ on which $c_i$ is strict. 

On the model $X_i \coloneqq X_{\Sigma_i}$ the section $c_i$ of $\llb L|_{X_i}$ trivialises $\llb L^\trop|_{X_i}$ and hence induces a line bundle $\lb L_i$ on $X_i$ (see Definition~\ref{def:line_bundle_for_adm}). Moreover, these line bundles are all canonically identified on the dense open $U$ of $X_i$ where the log structure is trivial; we write $\lb L$ for the resulting line bundle on $U$ and $\ell_i$ for the corresponding rational section of $\ca L_i$. The sequence $(\lb L, (X_i, \lb L_i, \ell_i)_{i\ge 1})$ is the data for a log adelic line bundle on $X$. This gives the map $\overline{a}$. The commutativity of the diagram and the injectivity statement follow immediately from the definitions. The last statement follows from the fact that in the log smooth case the map $a$ is an isomorphism, as stated in Proposition~\ref{prop:map_p}.
\end{proof}

\begin{remark} \label{rem:b-inverse-to-p} When $X$ is log smooth the inverse of the isomorphism $ \overline{a}$ is the map
\[ \overline{b} \colon \on{Ad-Pic}_{\on{log}}(X) \to \on{b-Pic}^{\on{cont}}(X) \]
given by sending the isomorphism class of a log adelic line bundle $(\ca L, (X_i,\ca L_i,\ell_i))$ to the pair $(\ca L_1^\LOG, c)$ where $c$ is the conical function obtained by taking the pointwise limit of the PL functions corresponding to the quotients $\ell_i  \ell_1^{-1}$. 
\end{remark}

\begin{definition} \label{def:assoc_adelic}
Let $\ca L$ be a line bundle on $X$. The log b-line bundle $(\ca L^{\LOG},0)$ is mapped via $\overline{a}$ to the model log adelic line bundle determined by $\ca L$, and we simply denote it again by $\ca L$. Let $c \in \on{CCon}(X)$. We denote the image log adelic line bundle of the log b-line bundle $(\ca L^{\LOG},0) \otimes (\oo, c)$ under the map $\overline{a}$ by $\ca L \otimes \oo(\divisor c)$ or $\ca L \otimes \oo(c)$ or even $\ca L(c)$. 
\end{definition}

\begin{example} Let $D \in \on{Div}_X(U)$ with Zariski closure $\overline{D}$ on $X$. Let $c \in \on{CCon}(X)$. The continuous log b-line bundle determined by the pair $(D,c)$ is $(\oo_X(\overline{D})^{\LOG},0) \otimes (\ca O,c)$, so that the associated log adelic line bundle is denoted $\oo_X(\overline{D}) \otimes \oo(\divisor c)$. Recall from Definition~\ref{def:assoc_Weil} that the adelic divisor associated to the pair $(D,c)$ is denoted by $\overline{D} + \divisor c$. We obtain the convenient canonical identification of log adelic line bundles $\oo_X(\overline{D} + \divisor c) = \oo_X(\overline{D}) \otimes \oo(\divisor c)$.
\end{example}

\subsection{Pure log adelic line bundle for abelian schemes}\label{sec:yz-invariant-adelic}

Let $L$ be the universal rigidified line bundle on $N_{g,1}$. The technique of proof of Theorem~\ref{thm:YZ_invariant} can be applied in the setting of log flat log stacks, and readily gives the following.
\begin{theorem} \label{thm:YZ_invariant_N} There exists a unique adelic line bundle $\overline{L}$ on $N_{g,1}$ extending $L$ and satisfying $[2]^*\overline{L} \cong \overline{L}^{\otimes 4}$. For all $m \in \zz$ we have $[m]^*\overline{L} \cong \overline{L}^{\otimes m^2}$.
\end{theorem}
We call the adelic line bundle $\overline{L}$ from Theorem~\ref{thm:YZ_invariant_N} the \emph{pure adelic extension} of $L$ over $N_{g,1}$. 

Our work done so far allows a description of $\overline{L}$ in terms of log b-line bundles. Recall the smooth function $\theta^{\on{sm}}$ from Definition~\ref{def:eq_theta}. Let $\llb L$ be the universal  log line bundle on $N_{g,1}^\LOG$. Take any $N_{g,1}$-admissible cone decomposition $\widetilde{\Sigma}$. 

\begin{theorem} \label{thm:comparison_adelic_b_AV}  Under the map $\overline{a}$, the continuous log b-line bundle $(\llb L, -\theta^{\on{sm}}) $ on $N_{g,1}^{\widetilde{\Sigma}}$ is mapped to the log adelic line bundle $\overline{L}$ on $N_{g,1}^{\widetilde{\Sigma}}$. 
\end{theorem}
\begin{proof} We recall from Theorem~\ref{thm:main} that for each $m \in \zz$ the isomorphism $[m]^* L \isom L^{\otimes m^2}$ of rigidified line bundles over $N_{g,1}$ extends into an isomorphism of log b-line bundles $[m]^*(\llb L, -\theta^{\on{sm}})  \isom m^2 (\llb L, -\theta^{\on{sm}})$ over $N_{g,1}^{\widetilde{\Sigma}}$. The same invariance property with respect to multiplication-by-$m$ then holds for the image of $(\llb L,- \theta^{\on{sm}})$ under the homomorphism $\overline{a}$. The theorem then follows by the uniqueness statement in Theorem~\ref{thm:YZ_invariant_N}.
\end{proof}
Now take any admissible principal polarisation function $\phi$ on $\widetilde{\Sigma}$, and consider the continuous conical function $c=\phi-\theta^{\on{sm}}$ on ${N_{g,1}^{\widetilde \Sigma}}$. By Theorem~\ref{thm:theta-eq-cocycle-new} we have an isomorphism of log b-line bundles $(\llb L, -\theta^{\on{sm}})   \cong (\lb L_{-\phi}^{\LOG},0) \otimes (\oo,c)$ over $N_{g,1}^{\widetilde \Sigma}$. Using the notation from Definition~\ref{def:assoc_adelic} we then obtain the following corollary.

\begin{corollary} We have an identification of log adelic line bundles $\ca L_{-\phi}(c) \cong \overline{L}$ on $N_{g,1}^{\widetilde{\Sigma}}$.
\end{corollary}

As a particular case, we can take for $\widetilde \Sigma$ the mixed Delaunay--Voronoi cone decomposition together with the tropical theta function $\theta^{\mathrm{pl}}$ with characteristics, see Proposition~\ref{prop:decomposition-over}. Recall that we have the continuous conical function 
\[ \theta^{\on{inv}} = \theta^{\mathrm{pl}} - \theta^{\on{sm}} . \] 
This then gives the following.

\begin{corollary} \label{cor:specialize_Riemann}
We have an identification of log adelic line bundles $\ca L_{-\theta^{\mathrm{pl}}}(\theta^{\mathrm{inv}}) \cong \overline{L}$ on $N_{g,1}^{\mathrm{mDV}}$.
\end{corollary}

\section{Abelian schemes with theta divisor} \label{sec:AV_theta}

Let $g$ be a positive integer. 
By Remark \ref{rem:N_g_comments}, we can think of the stack $N_g$ as classifying pairs $(p \colon A \to S, \Theta)$ consisting of an abelian scheme $p \colon A \to S$ of relative dimension~$g$ and a symmetric relatively ample effective relative Cartier divisor $\Theta$ on $A/S$ representing a principal polarisation. Thus $N_g$ becomes the stack called $\widetilde{\ca A}_g$ in \cite{mb_fonct}. It follows that we have a universal symmetric relative effective Cartier divisor on $N_{g,1}$ over $N_g$, which we also denote by $\Theta$. It gives rise to a line bundle $\oo(\Theta)$ over~$N_{g,1}$.

The purpose of this section is to recall the ``key formula'' over $N_{g,1}$ due to Moret-Bailly \cites{mb, mb_fonct}, and to discuss an extension of this formula to the space of \emph{adelic} line bundles over $N_{g,1}$, see  Theorem~\ref{thm:adelic_key_restate}.

\subsection{The key formula} \label{subsec:key}

Let $\pi \colon N_{g,1} \to N_g$ denote the structure map and let $e \colon N_g \to N_{g,1}$ denote the zero section. We set
\[ \omega := e^*(\Omega^g_{N_{g,1}/N_g}) = \pi_* (\Omega^g_{N_{g,1}/N_g}) , \]
and we call $\omega$ the \emph{Hodge line bundle} on $N_g$.

Recall that we have the universal symmetric, relatively ample and rigidified line bundle $L$ on $N_{g,1}$. The pushforward $\pi_*L$ is a line bundle on $N_g$. The line bundles $L$ and $\oo(\Theta)$ are isomorphic on all fibers of the map $\pi \colon N_{g,1} \to N_g$ and since $\pi$ is proper and $L$ is rigidified we have a canonical identification
\[ L = \oo(\Theta) \otimes \pi^* e^* \oo(-\Theta) \] 
of line bundles on $N_{g,1}$. The key formula ``computes'' the difference $e^*\oo(-\Theta)$. 

As a first step, applying the projection formula, and using the fact that $\pi_* \oo(\Theta)$ is trivialised by the canonical section $1_\Theta$, we find a canonical identification
\[  e^* \oo(-\Theta) = \pi_*L  .  \]
The key formula is then the following.
\begin{theorem}[{\cite[Th\'eor\`eme~0.2]{mb_fonct}, \cite[Theorem I.5.1, p.25]{fc}}] \label{thm:key}  \label{thm:key_cor}  \label{thm:key_cor_direct_image}\leavevmode
\begin{itemize}
\item[(i)] The line bundle $\left( \pi_* L \right)^{\otimes 8} \otimes \omega^{\otimes 4}$
is trivial over $N_g$.  
\item[(ii)] The line bundle $\left( L \otimes  \oo(-\Theta) \right)^{\otimes 8} \otimes \pi^* \omega^{\otimes 4}$
is trivial over $N_{g,1}$.
\end{itemize}
\end{theorem}

From now on, we fix compatible trivialisations for (i) and (ii) above. As is observed in \cite[Section~0.3]{mb_fonct}, such trivialisations are unique up to a sign on both irreducible components of $N_g$.

\subsection{Moret-Bailly models and the key formula} \label{sec:MB_models}

We make a slight variation of the material in~\cite[Section~4.3]{bost_duke}. 
Assume we are given a discrete valuation ring $R$ with fraction field $F$, an abelian variety $A$ over $F$, and a rigidified, symmetric, ample line bundle $L$ on $A$ defining a principal polarisation. We denote by $K( L^{\otimes 2})$ the kernel of the polarisation $\lambda_{L^{\otimes 2}} \colon A \to A^\vee$ associated to $L^{\otimes 2}$. Write $S = \Spec R$.
\begin{definition} A \emph{Moret-Bailly model} of $(A, L)$ consists of a smooth group-scheme model $f \colon \aa \to S$ of $A$ whose fibrewise connected component of identity is semiabelian,  together with a cubical symmetric relatively ample line bundle $\ca L_{\ca A}$ on $\aa$ extending $ L$, such that the group scheme $K(L^{\otimes 2})$ extends as a finite flat subgroup scheme of $\aa$ over~$S$.
\end{definition}
A Moret-Bailly model of $(A,L)$ exists after a finite base change of $R$. See for instance \cite[Theorem~4.10(i)]{bost_duke} and its proof. 

When $f \colon \ca A \to S$ is a smooth group scheme with zero section $e \colon S \to \ca A$ we set
\[ \omega_{\aa/S} := e^*(\Omega^g_{\aa/S}) = f_* (\Omega^g_{\aa/S}) . \]

\begin{theorem} \label{thm:key_semi_abelian} 
Let $(f \colon \aa \to S, \ca L_{\ca A})$ be a Moret-Bailly model of $(A,L)$. Then $f_*{\ca L_{\ca A}}$ is a line bundle on $S$. Moreover, there exists a trivialisation of $(f_*{\ca L_{\ca A}})^{\otimes 8} \otimes \omega_{\aa/S}^{\otimes 4}$ compatible with the trivialisation of $H^0(A,L)^{\otimes 8} \otimes H^0(A,\Omega^g_{A/F})^{\otimes 4}$ obtained from the trivialisation fixed at the end of Section~\ref{subsec:key}.
\end{theorem}
\begin{proof} This follows from \cite[Th\'eor\`eme~VIII.1.3(i)]{mb}.
\end{proof}
Assume from now on that $R$ is \emph{complete}. Let $\mathbb{F}$ be the completion of an algebraic closure of the fraction field~$F$. Following Section~\ref{sec:Berk_anal} we have functorially associated to $A$ its Berkovich analytification $A^\an$ over $\mathbb{F}$. 
The rigidified line bundle~$L$ has a functorial analytification $L^\an$ over $A^\an$. 

As is explained in for example \cite[Example~3.7]{gubler_can_meas}, there exists a unique continuous metric $\|\cdot\|_L$ on $L^\an$ compatible with the isomorphisms $[m]^*L \isom L^{\otimes m^2}$ for all $m \in \zz$. We call the metric $\|\cdot\|_L$  the \emph{canonical metric} on $L^\an$. We note for future reference that the pair $(L^\an, \|\cdot\|_L)$ is the \emph{analytification} of Yuan--Zhang's invariant adelic line bundle $\overline{L}$ over $A$, in the sense of \cite[Section~3.4]{yz}.

Let $\Sk(A)$ be the canonical skeleton of $A^\an$ (see Section~\ref{sec:Berk_anal}).

\begin{proposition} \label{prop:order_vanishing_MB}
Let $(f \colon \aa \to S, \ca L_{\ca A})$ be a Moret-Bailly model of $(A,L)$. Let $s \in H^0(A,L)$ be a non-zero global section. As $H^0(A,L)$ is the generic fiber of the line bundle $f_* {\ca L_{\ca A}}$ over $S$, we may view $s$ as a rational section of $f_* {\ca L_{\ca A}}$. Let $v$ denote the closed point of $S$. The equality
\[   - \log \sup_{x \in \Sk(A)} \|s(x)\|_L  = \ord_{v,f_*\ca L_{\ca A}}(s)  \]
holds.
\end{proposition} 
\begin{proof} See \cite[Corollary~10.7]{djs_faltings}.
\end{proof}

\subsection{Alexeev--Nakamura models and the key formula} \label{sec:AN_cle}

Let $R$ be a complete discrete valuation ring and write $S=\Spec R$. Let $F$ be the fraction field of $R$, let $A$ be an abelian variety of dimension~$g$ with semiabelian reduction over~$F$, and let $L$ be a symmetric, rigidified, ample line bundle on $A$ defining a principal polarisation of $A$. 

In \cite{djs_canonical} one finds a construction of so-called \emph{tautological models} $(\ca P/S, \ca L_{\ca P})$ of $(A,L)$ based on the \emph{tropicalisation} of a \emph{non-archimedean theta function} for $L$. The model $\ca P$ is integral, and the structure map $g \colon \ca P \to S$ is projective and flat. 

Let $G/S$ be the connected component of the N\'eron model of $A$ and let $T$ be the torus of the special fiber of $G$ which is a semiabelian variety. We write $X=\Hom(T,\bb G_m)$ as usual; following the constructions in Section~\ref{sec:LAV} applied to our setting we moreover have a free abelian group $Y$ together with an isomorphism $Y \isom X$, as well as a quadratic function $a \colon X \to \zz$ with associated symmetric bilinear form $b_a \colon X \times X \to \bb Z$, resulting in a lattice embedding $Y \to \Hom(X,\bb R)$. 

As we have seen, the symmetric line bundle $L$ determines a class $\kappa = \overline{\rho}$ in $X/2X$ for some $\rho \in X$. Let $X \to \Hom(X,\rr)$ be the map given by $w \mapsto b_a(w,-)$, and let  $\Hom(X,\rr) \to \Hom(X,\rr)/Y$ be the projection. Following $\rho$ through the composition of maps $X \to \Hom(X,\rr) \to \Hom(X,\rr)/Y$, we obtain a $2$-torsion point $\kappa$ in the real torus $\Hom(X,\bb R)/Y$. In Section~\ref{sec:Berk_anal} we have canonically identified $\Hom(X,\bb R)/Y$ with the canonical skeleton $\Sk(A)$ of the Berkovich analytification $A^\an$ of $A$. The element $\kappa \in \Sk(A)$ coincides with what is called the \emph{tropical theta characteristic} of the pair $(A,L)$ in \cite{djs_canonical}. We denote by $\val \colon A^\an \to \Sk(A)$ the tropicalisation homomorphism from Section~\ref{sec:Berk_anal}.

The results in \cite[Section~4]{djs_canonical} give that the tropicalisation of any non-archimedean theta function for $(A,L)$ is equal to the map $\theta_\rho^{\mathrm{pl}}(a,-) \colon \Hom(X,\bb R) \to \bb R$, up to an additive constant. A canonical choice of tautological model $(\ca P, \ca L)$ of $(A,L)$ is therefore obtained by taking $\theta_\rho^{\mathrm{pl}}(a,-)$ itself. Let $\theta_\rho^{\mathrm{inv}}(a,-) \colon \Sk(A) \to \bb R$ be the invariant tropical theta function with characteristic $\rho$ from \eqref{def:inv-tropi-theta}, restricted to $\Sk(A) = \Hom(X,\bb R)/Y$.

We obtain a direct connection with the log moduli stacks constructed in Section~\ref{sec:tropical_RT}, as follows. Consider the $N_g$-admissible cone decomposition $\Sigma^{\mathrm{Del}}$, the compatible $N_{g,1}$-admissible cone decomposition ${\widetilde \Sigma}^{\mathrm{mDV}}$, and the principal polarisation function $\theta^{\mathrm{pl}}$. Recall the corresponding algebraic stacks with log structure $N_g^{\mathrm{Del}}$ and $N_{g,1}^{\mathrm{mDV}}$, with structure map $\pi \colon  N_{g,1}^{\mathrm{mDV}} \to N_g^{\mathrm{Del}}$ from Definition~\ref{def:comp_N}. 

The construction in Definition~\ref{def:line_bundle_for_adm} produces a line bundle $\ca L_{\theta^{-\mathrm{pl}}}$ on $N_{g,1}^{\mathrm{mDV}}$. It coincides with the universal rigidified line bundle $ L$ over $N_{g,1}$. 
By the valuative criterion of properness, up to replacing $S$ by a finite ramified cover we have a moduli map $j \colon S \to N_g^{\mathrm{Del}}$ associated to the pair $(A,L)$.  
Our tautological model $g \colon \ca P \to S$ is obtained as a fiber product
 \begin{equation} \label{eqn:taut_model}
\begin{tikzcd}
\ca P \ar[d,"g"] \arrow[r] &  N_{g,1}^{\mathrm{mDV}}  \ar[d, "\pi"] \\
 S  \arrow[r, "j"] &   N_g^{\mathrm{Del}}
\end{tikzcd}
\end{equation}
in the category of log stacks, with the line bundle $\ca L_{\ca P}$ on $\ca P$  being equal to the pullback of the line bundle $\ca L_{-\theta^{\mathrm{pl}}}$. 

Let $L^\an$ be the analytification of $L$ over $A^\an$. Following for example \cite[Section~4.1]{djs_faltings} or \cite[Section A.5.1]{yz}, the model $\ca L_{\ca P}$ of $L$ naturally defines a \emph{model metric} $\|\cdot\|_{\ca L_{\ca P}}$ on $L^\an$.  Let $\|\cdot\|_L$ be the canonical metric on $L^\an$ (see Section~\ref{sec:MB_models}).

Let $s \in H^0(A,L)$ be a non-zero global section, and write $D=\divisor_L(s)$. As before let $a \colon X \to \zz$ be the quadratic function associated to our pair $(A,L)$.
\begin{proposition} \label{prop:comparison_canonical_model_metric}
 For every $x \in A^\an \setminus |D|$ we have an equality
\begin{equation} 
 -\log \|s \|_L(x) = -\log\|s\|_{\ca L_{\ca P}}(x) + \theta_\rho^{\mathrm{inv}} (a,\val(x)) 
\end{equation}
in $\bb R$.
\end{proposition}
\begin{proof} Applying Theorem~\ref{cor:specialize_Riemann} we find  a canonical isomorphism $\ca L_{-\theta^{\mathrm{pl}}}(\theta_\rho^{\mathrm{inv}}(a,-)) \cong \overline{ L}$ of adelic line bundles over $A$. Upon applying the analytification functor to the adelic line bundle $\ca L_{-\theta^{\mathrm{pl}}}(\theta_\rho^{\mathrm{inv}}(a,-))$ we obtain the line bundle $L^\an$ on $A^\an$ with minus log of the norm of $s$ given by $x \mapsto -\log\|s\|_{\ca L_{\ca P}}(x) + \theta_\rho^{\mathrm{inv}} (a,\val(x))$. On the other hand, upon applying the analytification functor to the adelic line bundle $\overline{ L}$ we obtain the line bundle $L^\an$ on $A^\an$ with minus log of the norm of $s$ given by $x \mapsto -\log\|s\|_L(x)$. The required equality follows.
\end{proof}
\begin{remark} Proposition~\ref{prop:comparison_canonical_model_metric} can also be obtained from \cite[Corollary~5.8]{djs_canonical}. 
\end{remark}

Note that $g_*\ca L_{\ca P}$ is a line bundle over $S$. As $H^0(A,L)$ is the generic fiber of the line bundle $g_* \ca L_{\ca P}$ over $S$, we may view $s$ as a rational section of $g_*\ca L_{\ca P}$.
\begin{proposition} \label{prop:vanishing_AN} 
\begin{itemize}
\item[(i)] The function $\|s\|_{\ca L_{\ca P}}$ is constant over the canonical skeleton $\Sk(A)$. 
\end{itemize}
Denote the constant from (i) by $c$.
\begin{itemize}
\item[(ii)] The multiplicity of $s$ in the line bundle~$\ca L_{\ca P}$ along the irreducible components of the special fiber of $\pp$ over $S$ is constant, equal to $c$. 
\item[(iii)] Let $v$ denote the closed point of $S$. We have $ \ord_{v,g_*\ll_{\ca P}}(s) = c $, where on the left hand side we view $s$ as a rational section of the line bundle $g_*\ca L_{\ca P}$.  
\end{itemize}
Let $\ca D_{\ca P}$ be the flat extension in $\ca P$ over $S$ of the effective divisor $D$.
\begin{itemize}
\item[(iv)] The divisor $\ca D_{\ca P}$ is a Cartier divisor. 
\item[(v)] We have a canonical isomorphism of line bundles $\ca L_{\ca P} \cong g^*g_*\ca L_{\ca P} \otimes \ca O_{\ca P}(\ca D_{\ca P})$ over~$\ca P$.
\end{itemize}
\end{proposition}
\begin{proof} (i) follows from  \cite[Corollary~5.7]{djs_canonical}. (ii) As is recalled in \cite[Proposition~5.9]{djs_canonical} and its proof, the Shilov points in $A^\an$ of the irreducible components of the special fiber of $\ca P$ (see Section~\ref{sec:Berk_anal} for the definition of Shilov points) are contained in $\Sk(A)$. (iii) is clear from (ii). As to (iv), note that we have an equality of Weil divisors $\divisor_{\ca L_{\ca P}}(s) = c \cdot \ca P_0 + \ca D_{\ca P}$ on $\ca P$, where $\ca P_0$ is the special fiber. As $\ca P_0$ and $\divisor_{\ca L_{\ca P}}(s)$ are Cartier we find that $\ca D_{\ca P}$ is Cartier. As to (v), note that we obtain a canonical isomorphism $\ca L_{\ca P} \cong g^*\ca M \otimes \ca O_{\ca P}(\ca D_{\ca P})$ over $\ca P$ for some line bundle $\ca M$ over~$S$ with trivialisation over $F$. By the projection formula and using that $g_*\ca O_{\ca P}(\ca D_{\ca P})$ is canonically trivialised we compute that $\ca M \cong g_*\ca L_{\ca P}$. 
\end{proof}
\begin{remark} The pair $(\ca P,\ca D_{\ca P})$ constructed in Proposition~\ref{prop:vanishing_AN} coincides with the model of $(A,D)$ constructed by Alexeev--Nakamura \cite{an}. 
\end{remark}

\begin{theorem} \label{thm:mumford}  Let $G/S$ be the connected component of the N\'eron model of $A$. There exists a trivialisation of $(g_*\ll)^{\otimes 8} \otimes \omega_{G/S}^{\otimes 4}$ compatible with the trivialisation of $H^0(A,L)^{\otimes 8} \otimes H^0(A, \Omega^g_{A/F})^{\otimes 4}$ obtained from the trivialisation of $(\pi_*L)^{\otimes 8} \otimes \omega^{\otimes 4}$ over $N_g$ fixed in Section~\ref{subsec:key}. Also, there exists a trivialisation of
\[ \left( \ll_{\ca P} \otimes \oo_\pp(-\ca D_{\ca P}) \right)^{\otimes 8} \otimes g^* \omega_{G/S}^{\otimes 4} \]
compatible with the trivialisation of 
\[ ( L \otimes \oo_A(-D) )^{\otimes 8} \otimes_F  H^0(A,\Omega^g_{A/F})^{\otimes 4} \]
over $A$ obtained from the trivialisation of 
\[ \left( L \otimes  \oo(-\Theta) \right)^{\otimes 8} \otimes \pi^* \omega^{\otimes 4} \]
over $N_{g,1}$ fixed at the end of Section~\ref{subsec:key}.
\end{theorem}
\begin{proof} 
It is enough to show the result is true after passing to a finite extension of $R$, as the formation of both $g_*\ll_{\ca P}$  and of $\omega_{G/S}$ is compatible with base change. Hence without loss of generality we may assume that $(A,L)$ allows a Moret-Bailly model $(f \colon \aa \to S,\ca L_{\ca A})$ over $S$ as in Section~\ref{sec:MB_models}. Note that
\[ \omega_{G/S} = \omega_{\aa/S} . \]
By Theorem~\ref{thm:key_semi_abelian}, in order to arrive at the first statement it suffices to show the existence of an  isomorphism
\[ g_*\ll_{\ca P} \cong f_*\ca L_{\ca A}\]
of line bundles over $S$ extending the identity map on $H^0(A,L)$.

Let $s$ be a non-zero element of $H^0(A,L)$, and write $D=\divisor_L(s)$. Let $v$ denote the closed point of $S$. We are reduced to showing that the multiplicities 
$\ord_{v,g_*\ll_{\ca P}}(s)$ (resp. $ \ord_{v,f_*\ca L_{\ca A}}(s)$)
in  $g_*\ll_{\ca P}$ (resp. $f_* \ca L_{\ca A}$)
 are equal. 
  
Let $x \in A^\an \setminus |D|$. By Proposition~\ref{prop:comparison_canonical_model_metric} we have
  \[ -\log \|s(x) \|_L = - \log \|s (x) \|_{\ll_{\ca P}} + \theta_\rho^{\mathrm{inv}}(a,\val(x)) . \]
  In particular, for all $x \in \Sk(A)$ we have
 \[   -\log \|s(x)\|_L   = - \log \|s(x)\|_{\ll_{\ca P}} + \theta_\rho^{\mathrm{inv}}(a,x) . \]
By Proposition~\ref{prop:vanishing_AN}(ii),  the multiplicity of $s$ in $\ll_{\ca P}$ along the irreducible components of the special fiber of $\pp$ over $S$ is constant. Let $c$ denote this constant.
We find  for all $x \in \Sk(A)$ that
 \[  -\log \|s(x)\|_L = c + \theta_\rho^{\mathrm{inv}}(a,x) . \]

 By Proposition~\ref{prop:trop_theta_new}, the function $\theta_\rho^{\mathrm{inv}}(a,-)$ has minimum value equal to zero over $\Sk(A)$.  It follows that 
\begin{equation} \label{eqn:proof_cle_II}
 c = -\log \sup_{x \in \Sk(A)} \| s(x) \|_L  . 
\end{equation}
 By Proposition~\ref{prop:order_vanishing_MB} we have
\begin{equation} \label{eqn:proof_cle_III}
   - \log \sup_{x \in \Sk(A)} \|s(x)\|_L  = \ord_{v,f_*\ca L_{\ca A}}(s) . 
\end{equation}

We conclude that
\begin{equation} \label{eq:first_c}
  \ord_{v,f_*\ca L_{\ca A}}(s) = c . 
\end{equation}
By Proposition~\ref{prop:vanishing_AN}(iii) we have
 \begin{equation} \label{eqn:proof_cle_I}
  \ord_{v,g_*\ll_{\ca P}}(s) = c  . 
 \end{equation}
Combining \eqref{eq:first_c} and \eqref{eqn:proof_cle_I} we find
$ \ord_{v,g_*\ll_{\ca P}}(s) =  \ord_{v,f_*\ca L_{\ca A}}(s)$,
as required.

We now consider the second statement. By Proposition~\ref{prop:vanishing_AN}(v) we have a canonical isomorphism $\ca L_{\ca P} \cong g^* g_*\ca L_{\ca P} \otimes \ca O_{\ca P}(\ca D_{\ca P})$ of line bundles on $\ca P$. Hence we have a canonical isomorphism
\[   \left( \ll_{\ca P} \otimes \oo_\pp(-\ca D_{\ca P}) \right)^{\otimes 8} \otimes g^* \omega_{G/S}^{\otimes 4} \cong
g^*  \left( \left( g_*\ll_{\ca P} \right)^{\otimes 8} \otimes  \omega_{G/S}^{\otimes 4} \right)  \]
of line bundles on $\ca P$. 
Now use the trivialisation of $(g_*\ll_{\ca P})^{\otimes 8} \otimes \omega_{G/S}^{\otimes 4}$ that we have just obtained and pull it back along $g$.
\end{proof} 

\subsection{Pure weight 2 theta divisor and the key formula}

Let $G$ be the unique semiabelian scheme over $N_g^{\mathrm{Del}}$ extending the universal abelian scheme $N_{g,1}$ over $N_g$ and denote by $\omega_{G/N_g^{\mathrm{Del}}}$ the associated Hodge line bundle over $ N_g^{\mathrm{Del}}$. The key formula in Theorem~\ref{thm:key} gives rise to a global section of the line bundle $L^{\otimes 8} \otimes \pi^* \omega^{\otimes 4}$ over $N_{g,1}$, with divisor $8\,\Theta$. We obtain from this a rational section $f$ of the line bundle
\begin{equation}\label{eq:bundle}
 \ca L_{-\theta^{\mathrm{pl}}}^{\otimes 8} \otimes \pi^*\omega_{G/N_g^{\mathrm{Del}}}^{\otimes 4} 
\end{equation}
over $N_{g,1}^{\mathrm{mDV}}$. Write $\divisor f$ for the divisor of $f$. 
Let $\bar\Theta$ be the Zariski closure in $N_{g,1}^{\mathrm{mDV}}$ of the theta divisor $\Theta$ over $N_{g,1}$.  It is a Weil divisor on $N_{g,1}^{\mathrm{mDV}}$. 

\begin{theorem}\label{thm:bar_theta_def_equivalence}
 The equality of Weil divisors 
\begin{equation*}
\on{div} f = 8 \,\bar \Theta
\end{equation*}
holds on $N_{g,1}^{\mathrm{mDV}}$. 
\end{theorem}
That is to say, eight times the Zariski closure $\bar \Theta$ of the theta divisor over $N_{g,1}^{\mathrm{mDV}}$ has a canonical structure of Cartier  divisor. We prove Theorem~\ref{thm:bar_theta_def_equivalence} in Section~\ref{sec:proof_theorem}. 

\begin{definition} (The pure weight $2$ theta divisor)
Over $N_{g,1}^{\mathrm{mDV}}$ we define $\Theta^{\mathrm{pure}}$ to be the continuous log Weil b-divisor associated to the pair $(\frac{1}{8}\divisor(f) ,\theta^{\on{inv}})$. 

We also use the notation $\Theta^{\mathrm{pure}}$ to denote the associated log adelic divisor. 
\end{definition}

Note that $\Theta^{\mathrm{pure}}$ is effective in the sense that $\divisor f$ is an effective Weil divisor by Theorem~\ref{thm:bar_theta_def_equivalence}, and the continuous function $\theta^{\on{inv}}$ is non-negative by Proposition~\ref{prop:trop_theta_new}.

We are now able to state our extension of the key formula.
Let $\overline{\omega}_{N_{g,1}/N_g}$ be the \emph{log adelic Hodge line bundle}, i.e., the model log adelic line bundle on $N_g$ determined by the line bundle $\omega_{G/\bar N_g}$. See also \cite[Section~5.5]{yz} where this is discussed in the context of usual adelic line bundles.
\begin{theorem}  \label{thm:adelic_key_restate}
Over $N_{g,1}$ we have an isomorphism of log adelic line bundles
\[  \overline{L}^{\otimes 8} \cong \oo(8 \, \Theta^\mathrm{pure}) \otimes  \pi^* \overline{\omega}_{N_{g,1}/N_g}^{\otimes -4}   \]
extending the isomorphism of line bundles
\[   L^{\otimes 8} \cong  \oo(8 \, \Theta)  \otimes \pi^* \omega_{N_{g,1}/N_g}^{\otimes -4} \]
over $N_{g,1}$ obtained from the key formula. 

\end{theorem}
\begin{proof} The continuous log b-line bundle determined by the pair $(\divisor f,8 \, \theta^{\on{inv}})$ is equal to 
$(\oo (\divisor f)^{\LOG},0) \otimes (\oo,8 \, \theta^{\on{inv}}) $.
We can rewrite this as 
\[  (\ca L_{-\theta^{\mathrm{pl}}}^{\otimes 8,\LOG},0) \otimes (\pi^*\omega_{G/N_g^{\mathrm{Del}}}^{\otimes 4,\LOG} ,0) \otimes (\oo,8 \, \theta^{\on{inv}}) = 8\, (\ca L_{-\theta^{\mathrm{pl}}}^{\LOG},\theta^{\on{inv}}) \otimes 4 (\pi^*\omega_{G/N_g^{\mathrm{Del}}}^{\LOG} ,0) . \]
By Corollary~\ref{cor:specialize_Riemann} the continuous log b-line bundle $(\ca L_{-\theta^{\mathrm{pl}}}^{\LOG},\theta^{\on{inv}})$ has associated log adelic line bundle $\overline{L}$. The continuous log b-line bundle $(\pi^*\omega_{G/N_g^{\mathrm{Del}}}^{\LOG} ,0)$ has associated log adelic line bundle $  \pi^* \overline{\omega}_{N_{g,1}/N_g} $. The theorem follows by the commutative diagram in Proposition~\ref{prop:commutative_p_bar}.
\end{proof}
\begin{corollary}  \label{thm:adelic_key_restate_bis}
Over $N_{g}$ we have an isomorphism of log adelic line bundles
\[   e^*\oo(8 \, \Theta^\mathrm{pure}) \cong \overline{\omega}_{N_{g,1}/N_g}^{\otimes 4}   \]
extending the isomorphism of line bundles
\[    e^*\oo(8 \, \Theta) \cong  \omega_{N_{g,1}/N_g}^{\otimes 4} \]
over $N_{g}$ obtained from the key formula. 
\end{corollary}

\begin{remark}
Over $\bb Z[1/2]$ the space $N_{g,1}^{\mathrm{mDV}}$ is normal. Theorem~\ref{thm:bar_theta_def_equivalence} then implies that over $\bb Z[1/2]$, the Weil divisor $8 \,\bar \Theta$ comes from a \emph{unique} Cartier divisor. 

Thus, away from characteristic $2$ we might as well define $\Theta^{\mathrm{pure}}$ as the continuous log Weil b-divisor (or log adelic divisor) associated to the pair $(\bar \Theta, \theta^{\on{inv}})$, as we did in the Introduction. 

In characteristic $2$ the stack $N^{\mathrm{mDV}}_{g,1}$ is not normal, and so the Weil divisor $8 \,\bar \Theta$ may not \emph{uniquely} determine a Cartier divisor, and we need the rational section~$f$ to single out the one we care about. 
\end{remark}

\subsection{Zariski closure of the theta divisor} \label{sec:proof_theorem}

It remains to give our proof of Theorem~\ref{thm:bar_theta_def_equivalence}. 
To simplify the notation, for the remainder of this section we write $\bar N_g$ for $N_g^{\mathrm{Del}}$ and $\bar N_{g, 1}$ for $N_{g,1}^{\mathrm{mDV}}$.  To prove Theorem \ref{thm:bar_theta_def_equivalence}, we construct a suitable logarithmic base change of $\pi\colon \bar N_{g,1} \to \bar N_{g}$. Because there are finitely many cone orbits, there exists a positive integer $n_g$ such that if $\sigma \to \tau$ is a surjection from an mDV cone in $\tilde C_\rho(W)_\bb Q$ to a Delaunay cone in $C_\rho(W)_\bb Q$, the image of the lattice in $\sigma$ is contained in $n_g$ times the lattice in $\tau$. We fix such an $n=n_g$, and note that it can be taken to be $1$ if~$g \le 4$. 

We write $\bar N_{g}^{n}\to \bar N_g$ for the root stack induced by multiplication-by-$n$ on the lattice $C_\rho(W)$. This map is an isomorphism over $N_g$, and is a log monomorphism. We write $\bar N_{g,1}^n$ for the logarithmic fibre product $\bar N_{g,1} \times_{\bar N_g} \bar N_g^n$. 

\begin{lemma}\label{lem:n_flatness}
The natural map $\pi^n\colon \bar N_{g,1}^n \to \bar N_g^n$ is integral and saturated, hence is flat with reduced geometric fibres. 
\end{lemma}
\begin{proof}
The map $\pi^n$ is log smooth. By construction, the map on fans is equidimensional, and is surjective on lattice points for every cone. By, for instance, \cite[Theorem 2.1.4]{molcho2021universal} the map is integral and saturated. It is then flat and has reduced geometric fibres by \cite[Theorems IV.4.3.5 and IV.4.3.6]{oguslogbook}.
\end{proof}

We write $\lb L$ for the pullback of $\lb L_{-\theta^{\mathrm{pl}}}$ to $\bar N_{g,1}^n$. 

\begin{lemma}\label{lem:cohomology_polarisation}
For every positive integer $d$, the sheaf $\pi^n_*\lb L^d$ is a vector bundle of rank $d^g$ on $\bar N_g^n$, and all higher derived pushforwards vanish. Its formation commutes with arbitrary schematic base change. 
\end{lemma}
\begin{proof}
Lemma \ref{lem:n_flatness} allows us to apply Cohomology and Base Change. It is then enough to show that on each geometric point of $\bar N^n_{g}$, we have $h^0(\lb L^d) = d^g$, and $h^i(\lb L^d) = 0$ for $i>0$. By restricting to test curves in $\bar N^n_g$ whose generic points lie in the interior, we may reduce to considering a log abelian variety over a complete dvr with divisorial log structure, pulling back the mixed Delaunay--Voronoi cone decomposition. The necessary cohomology computations are then found in \cite[Theorem~4.4]{an}. 
\end{proof}

We now define a Cartier divisor $\widetilde \Theta$ on $\bar N_{g,1}^n$ by locally choosing a generating section of $\pi^n_*\lb L$; the divisor is independent of the choice. 

\begin{lemma}
The Cartier divisor $\widetilde \Theta$ is flat over $\bar N_g^n$. 
\end{lemma}
\begin{proof}
This can checked locally over $\bar N_g^n$. We consider an open subset $U$ of  $\bar N_g^n$ where the line bundle $\pi^n_*  \lb L$ has a trivialising section $s$. We take a trait $S$ together with a map $j \colon S \to U$ such that the generic point maps into $N_g$. As in Section~\ref{sec:AN_cle}, taking logarithmic pullback along $j$ we obtain a projective flat map $g \colon \ca P \to S$ with generic fiber an abelian variety, equipped with a line bundle $\lb L_{\ca P}$ obtained by base change from $\lb L$. Because $\pi^n$ is integral and saturated, this logarithmic pullback coincides with the corresponding schematic pullback. The generating section $s$ over $U$ pulls back to a generating section of $\ca L_{\ca P}$. 

As the section $s$ is generating in $\ca L_{\ca P}$, it is non-zero as a section of $\ca L_{\ca P}$ in the special fiber of $\ca P \to S$, and hence there exists a point $x$ in the special fiber where $s$ is not zero. By Proposition~\ref{prop:vanishing_AN} we have that the order of vanishing of $s$ is constant along the irreducible components of the special fiber of $g$. By the existence of a point $x$ where $s$ is not zero, we see that $s$ does not vanish at each of the irreducible components of the special fiber of $g$. 

By using again the fact that the formation of $\pi_* \lb L$ commutes with arbitrary base change, we conclude that $s$ is non-vanishing at all irreducible components of all fibers of the map $\bar N_{g,1}^n|_U \to U$. As the fibers of $\pi^n$ are reduced we conclude that $s$ is not a zero divisor in all fibers of the map $\bar N_{g,1}^n|_U \to U$. Applying \cite[Lemma~9.3.4]{fga_explained} we conclude that $\divisor s$ is flat over $U$.
\end{proof}
\begin{corollary}\label{lem:closure_on_root}
The Cartier divisor $\widetilde \Theta$ is equal to the Zariski closure of $\Theta$ on the root stack $\bar N_{g,1}^n$. 
\end{corollary}
We need one more lemma.
\begin{lemma}\label{lem:div_on_root} The Cartier divisors
$8 \, \widetilde \Theta$ and the pullback of $\on{div}f$ from $\bar N_{g,1}$ to $\bar N_{g,1}^n$ are equal as Weil divisors. 
\end{lemma}
\begin{proof}
It suffices to show that for all prime Weil divisors $D$ supported over the boundary of $N_{g,1}$ in $\bar N_{g,1}^n$,  the Cartier divisors $\on{div} f$ and $8\,\widetilde \Theta$ have equal multiplicities along $D$. Let $D$ be such a prime divisor. Let $S$ be a trait mapping to $\bar N^n_{g,1}$ with generic point landing in $N_{g,1}$ and closed point mapping to the generic point of $D$. After replacing $S$ by a ramified cover, the map $S \to \bar N^n_{g,1}$ extends uniquely to a logarithmic map where $S$ has the standard divisorial log structure. Composition yields a logarithmic map $S \to \bar N^n_g$. 

We now form the log fibre product $$A_S \coloneqq  S \times_{\bar N_g^n} \bar N_{g,1}^n$$
with projection map $\pi_S\colon A_S \to S$. It is enough to show that the divisors $8 \, \widetilde \Theta$ and $\on{div}f$ agree after pullback to $A_S$, since this implies that their multiplicities along $D$ agree. 

This is now done as follows. Since forming $\pi^n_*\lb L$ commutes with schematic base change, we know that $\widetilde \Theta|_{A_S}$ is the divisor of a generating section of the invertible sheaf $(\pi_S)_*\lb L|_{A_S}$. 
By Theorem \ref{thm:mumford}, the pullback of $f$ to $S$ is actually a \emph{generating} section of $(\pi_S)_*\lb L^{\otimes 8} \otimes \omega_{G/S}^{\otimes 4}$. Hence its associated divisor is equal to $8\, \widetilde \Theta|_{A_S}$, as required. 
\end{proof}

\begin{proof}[Proof of Theorem \ref{thm:bar_theta_def_equivalence}]
Let $\widetilde N_{g,1}$ be the open of $\bar N_{g,1}$ obtained by deleting all strata of codimension $\ge 2$; on the tropical level, this means that we keep only the rays of the cone complex. We similarly define $\widetilde N_{g,1}^n$. It suffices to check the theorem on $\widetilde N_{g,1}$. As the map $\widetilde N^n_{g,1} \to \widetilde N_{g,1}$ is flat,  forming Zariski closure commutes with pullback along this map. It therefore suffices to check the result on $\widetilde N^n_{g,1}$. Here it follows from Corollary~\ref{lem:closure_on_root} and Lemma~\ref{lem:div_on_root}. 
\end{proof}

\subsection{Example: rank one degenerations} In this section we illustrate Theorem~\ref{thm:adelic_key_restate} by restricting to the log substack of rank one degenerations of abelian varieties. Let $N_g^{\le 1}$ be the partial compactification of $N_g$ given by those moduli points of $N_g^{\on{Del}}$ where the torus ranks of the fibers of the underlying semiabelian scheme are $\le 1$. Let $N_{g,1}^{\le 1}$ denote the pullback of $N_{g,1}^{\on{mDV}}$ over $N_g^{\le 1}$, and denote by $\pi \colon N_{g,1}^{\le 1} \to N_g^{\le 1}$ the projection map. 

 The fibers of $\pi \colon N_{g,1}^{\le 1} \to N_g^{\le 1}$ are irreducible, and the zero section $e \colon N_g \to N_{g,1}$ extends to a ``zero'' section $\overline{e} \colon N_g^{\le 1} \to N_{g,1}^{\le 1}$. We can realise $\overline{\Theta}\big|_{N_{g,1}^{\le 1}}$ as a Cartier divisor on $N_{g,1}^{\le 1}$, and this gives canonical isomorphisms of line bundles
\[ \overline{L} \big|_{N_{g,1}^{\le 1}} = \oo (\overline{\Theta}\big|_{N_{g,1}^{\le 1}}) \otimes \pi^* \overline{e}^*\oo (-\overline{\Theta}\big|_{N_{g,1}^{\le 1}}) \]
and
\begin{equation} \label{eqn:rank_one_theta} \pi_*\left( \overline{L} \big|_{N_{g,1}^{\le 1}} \right) =  \overline{e}^*\oo (-\overline{\Theta}\big|_{N_{g,1}^{\le 1}}) , 
\end{equation}
compare with Section~\ref{subsec:key}.

The real tori associated to the boundary points of $N_g$ in $N_g^{\le 1}$ are circles. At a generic point of the boundary divisor $\delta$, the tropical theta characteristic of the tropicalisation of the associated polarising log line bundle is either the trivial, or the non-trivial $2$-torsion point of a circle. This gives rise to a natural decomposition $\delta = \delta^{\on{tr}} + \delta^{\on{ntr}}$ of the boundary divisor of $N_g$ in $N_g^{\le 1}$. 

The formulas in Proposition~\ref{prop:trop_theta_new} show that at the primitive vector $v$ of a ray corresponding to the fiber over a generic point of $\delta^{\on{tr}}$ we have $\theta^{\on{inv}}(v) = \theta_{\rho=0}^{\on{inv}}(0) = 0$. Similarly, at the primitive vector of a ray corresponding to the fiber over a generic point of $\delta^{\on{ntr}}$ we have $\theta^{\on{inv}}(v) = \theta_{\rho=1}^{\on{inv}}(0) = 1/8$. This gives the equality of Cartier divisors
\begin{equation} \label{eqn:boundary_rank_one}
 8 \, \divisor (\theta^{\on{inv}}\big|_{N_{g,1}^{\le 1}}) = \pi^* \delta^{\on{ntr}} 
\end{equation}
on $N_{g,1}^{\le 1}$.

Let $G^{\le 1}$ denote the universal semiabelian scheme over $N_g^{\le 1}$.
By restricting the canonical isomorphism of adelic line bundles from Theorem~\ref{thm:adelic_key_restate} to the log stack $N_{g,1}^{\le1}$ we obtain a canonical isomorphism of line bundles
\[ \overline{L} \big|_{N_{g,1}^{\le 1}}^{\otimes 8} \cong \oo (8\, \overline{\Theta}\big|_{N_{g,1}^{\le 1}}) \otimes \oo(8 \, \divisor (\theta^{\on{inv}}\big|_{N_{g,1}^{\le 1}}) ) \otimes \pi^* \omega_{G^{\le1}/N_g^{\le 1}}^{\otimes -4}  \]
over $N_{g,1}^{\le 1}$. Pulling back along $\overline{e}$ and using \eqref{eqn:rank_one_theta} and \eqref{eqn:boundary_rank_one} gives  canonical isomorphisms of line bundles
\[ 8 \, \pi_*\left( \overline{L} \big|_{N_{g,1}^{\le 1}} \right) \cong  8 \, \overline{e}^*\oo (-\overline{\Theta}\big|_{N_{g,1}^{\le 1}})
\cong \oo( \delta^{\on{ntr}} ) \otimes \omega_{G^{\le1}/N_g^{\le 1}}^{\otimes -4}  \]
over $N_g^{\le 1}$. This refines a formula in Chow groups with $\qq$-coefficients due to Van der Geer in~\cite[Lemma~4.8]{vdg}.

\section{The arithmetic cases} \label{sec:arith_cases}

In this final section we discuss some aspects of the arithmetic side of our theory. First of all, we establish a connection between the function $\theta^{\mathrm{inv}}$ and the degeneration of the natural hermitian metric on $\oo(\Theta)$ over the complex numbers (see Corollary~\ref{cor:degeneration_theta}).
 Second, we derive a ``universal'' formula for the N\'eron--Tate height of a point on an abelian variety over a number field (see Theorem~\ref{univ_NT_formula}). 
 
 Let $g$ be a positive integer.

\subsection{Hermitian structures} \label{sec:hermitian}

We start by working in the \emph{local arithmetic case}, that is, we consider the map of orbifolds $N_{g,1}(\cc) \to N_g(\cc)$, and aim to endow (the sets of $\cc$-valued points of) the line bundles $L$, $\oo(\Theta)$ over $N_{g,1}$ and $\omega$ over $N_g$ with natural hermitian metrics. To ease notation we will just write $N_g$, $N_{g,1}$, $L$, etc.\ when we mean their sets of $\cc$-valued points.

For the line bundles $\oo(\Theta)$ and $\omega$ such natural hermitian metrics are discussed in \cite[Section~3]{mb_fonct}. We denote them by $\|\cdot\|_\Theta$ and $\|\cdot\|_{\mathrm{Hdg}}$, respectively.

The metric $\|\cdot\|_\Theta$ is characterised as a smooth hermitian metric on $\oo(\Theta)$ by the following two properties: (i) the curvature form $c_1(\oo(\Theta),\|\cdot\|_\Theta)$ is translation invariant in all fibers of the structure map $\pi \colon N_{g,1} \to N_g$; (ii) in each fiber of the structure map $N_{g,1} \to N_g$, the function $\|1_\Theta\|_\Theta^2$ has integral $2^{-g/2}$ against the Haar measure with volume one. Below we consider a rescaled metric $\| \cdot \|'_\Theta$ on $\oo(\Theta)$ defined by $\| \cdot \|_\Theta' = (2\pi)^{g/2} \| \cdot \|_\Theta$.

The metric $\|\cdot\|_{\mathrm{Hdg}}$ on the Hodge line bundle $\omega$ is also extensively discussed in \cite[Section~5.5]{yz} (where it is called the \emph{Faltings metric}). The discussion in loc.\ cit.\ essentially gives rise to the construction of a canonical adelic Hodge line bundle $\overline{\omega}$ over the arithmetic stack $A_g$ over $\zz$ (see \cite[Theorem 5.5.2]{yz}).

Finally, consider the invariant adelic line bundle $\overline{L}$ over the arithmetic stack $N_{g,1}$ over $\zz$, constructed in \cite[Chapter~6]{yz} (see also Section \ref{sec:yz-invariant-adelic}). The line bundle $L$ has a natural underlying ``pure'' hermitian metric which we denote by $\|\cdot\|_L$. It can be characterised as the unique smooth hermitian metric on $L$ such that for all $m \in \zz$ the isomorphism $[m]^*L \cong L^{\otimes m^2}$ of rigidified line bundles is an isometry. The curvature form $c_1(L,\|\cdot\|_L)$ is translation invariant in all fibers of the structure map $\pi \colon N_{g,1} \to N_g$, as follows for example from \cite[Proposition~II.2.1]{mb}.

\begin{theorem} \label{thm:metric_isom} (Moret-Bailly)
With the above hermitian metrics $\|\cdot\|_L$, $\|\cdot\|_\Hdg$ and $\|\cdot\|'_\Theta$ on the line bundles $L$, $\omega$ and $\oo(\Theta)$, respectively, the trivialising section of
the line bundle
\[ \left( L \otimes  \oo(-\Theta) \right)^{\otimes 8} \otimes \pi^* \omega^{\otimes 4} \]
fixed in Section~\ref{subsec:key} has constant norm equal to one over $N_{g,1}$.
\end{theorem}
\begin{proof} By our discussion above, the curvature form of the line bundle from the theorem vanishes in all fibers of the map $\pi \colon N_{g,1} \to N_g$. We would be done if we find a trivial metrised line bundle upon pulling back along the zero section $e \colon N_g \to N_{g,1}$. But \cite[Th\'eor\`eme~3.3]{mb_fonct} yields that any trivialising section of the line bundle $e^* \oo(-8\Theta) \otimes \omega^{\otimes 4}$ over $N_g$ has norm $(2\pi)^{4g}$, if $\oo(\Theta)$ is equipped with the norm $\|\cdot\|_\Theta$. This gives what we want as we work with the rescaled metric $\| \cdot \|_\Theta' = (2\pi)^{g/2} \| \cdot \|_\Theta$ on $\oo(\Theta)$. 
\end{proof}
Together with Theorem~\ref{thm:adelic_key_restate} we now arrive at the following.
\begin{corollary} \label{cor:global_adelic_isom} View $N_{g,1}$ as an arithmetic stack over $\Z$. Then we have an isomorphism of hermitian adelic line bundles
\[ \overline{L}^{\otimes 8} \cong \oo(8 \,\Theta^{\mathrm{pure}}) \otimes \pi^* \overline{\omega}^{\otimes - 4}   \]
over $N_{g,1}$ over $\zz$. 
\end{corollary}

\subsection{Explicit formula for the theta metric} \label{sec:explicit}

The aim of this section is to make the metric $\|\cdot\|_\Theta$ more explicit.
Let $\hh_g$ denote Siegel's upper half space consisting of complex symmetric $g$-by-$g$ matrices with positive definite imaginary part. This is a moduli space of complex principally polarised abelian varieties $A$ together with a symplectic basis of $H_1(A,\zz)$.

We have the Riemann theta function
\begin{equation} \label{eqn:Riemann_theta} \vartheta(z,\tau) = \sum_{n \in \zz^g} \exp\left(2\pi i\, \left( \frac{1}{2}\, {}^t n \cdot \tau \cdot n +  {}^t n \cdot z \right) \right) , \quad z \in \cc^g , \tau \in \hh_g  . 
\end{equation}
For $\tau \in \hh_g$ write $A_\tau = \cc^g/(\zz^g + \tau \zz^g)$. We have a symmetric effective Cartier divisor $\divisor \vartheta$ on $A_\tau$ defining the canonical principal polarisation $\lambda_\tau$ of $A_\tau$, i.e., the polarisation given by the Hermitian form $(z,w) \mapsto {}^tz \cdot \mathrm{Im}(\tau)^{-1} \cdot \overline{w}$.

Next we have the normalised Riemann theta function
\begin{equation} \label{eqn:Riemann_theta_norm}
 \|\vartheta\|(z,\tau) = (\det \Im \tau)^{1/4} \exp(-\pi\, {}^t (\Im z) \cdot (\Im  \tau)^{-1} \cdot (\Im z))  |\vartheta(z,\tau)|   
 \end{equation} 
 for $ z \in \cc^g$, $ \tau \in \hh_g$.
For fixed $\tau \in \hh_g$, this is invariant under translations by the lattice $\zz^g + \tau \zz^g$, and hence defines a map $\|\vartheta\|(-,\tau) \colon A_\tau \to \rr$.

The discussion in \cite[Section 3.2]{mb_fonct} implies the following.
\begin{proposition} \label{prop:archimedean_two_torsion} Let $A$ be a complex abelian variety of dimension~$g$, and let $L$ be a symmetric rigidified ample line bundle on $A$ defining a principal polarisation~$\lambda$. Fix  an isomorphism of ppav $\alpha \colon (A,\lambda) \isom (A_\tau,\lambda_\tau)$ where $\tau \in \hh_g$. Let $D$ be the symmetric effective divisor on $A$ determined by any global section of~$L$. There exists a unique $2$-torsion point $\kappa \in A_\tau$ such that under the isomorphism $\alpha$, the divisor $D$ on $A$ corresponds to the translate of the divisor $\divisor \vartheta$ on $A_\tau$ by $\kappa$. Furthermore, let $x \in A$. Then the equality 
\[ \|1_\Theta\|_\Theta(A,L,x) = \|\vartheta\|(\alpha(x)+\kappa,\tau) \]
holds.
\end{proposition}
We call the $2$-torsion point $\kappa \in A_\tau$ as in the above proposition the \emph{theta characteristic} of the line bundle~$L$ along the isomorphism $\alpha$. 

\subsection{Degeneration of the Riemann theta function} \label{sec:degenerate_Riemann_theta}

Our results yield information about the degeneration behavior of the classical Riemann theta function \eqref{eqn:Riemann_theta} and its normalised counterpart \eqref{eqn:Riemann_theta_norm} over the orbifold~$N_{g,1}$. The aim of this section is to provide some details. 

As above let $N_g^{\mathrm{Del}}$ and $N_{g,1}^{\mathrm{mDV}}$ be the log compactifications corresponding to the Delaunay and mixed Delaunay--Voronoi decomposition, respectively. In the work done in \cite[Chapter~5]{chern-weil} and \cite{siegel-jacobi} we study the (often singular) extensions of hermitian metrics on line bundles along the boundary of a compactification. In particular, we show  that if the metric extends as a \emph{plurisubharmonic (psh)}  hermitian metric, then one can define an associated toroidal $b$-line bundle encoding the so-called Lelong numbers of the psh metric as one varies over smooth log modifications. 

The results in loc.\ cit.\ are stated and proved for smooth complex varieties. However, one can extend these results to the line bundles $L$, $\omega$ and $\mathcal O(\Theta)$ (endowed with their canonical metrics) on the orbifolds $N_g^{\mathrm{Del}}$ and $N_{g,1}^{\mathrm{mDV}}$ by defining psh metrics, multiplier ideal sheaves, Lelong numbers, etc.\ by pulling back to local uniformizing charts, see for instance \cite[Section 6]{DK-orbifolds}. Concretely, similar to the results in \cite[Chapter 6]{chern-weil} one obtains the following:
\begin{itemize}
\item the universal line bundle $L$ together with its smooth invariant metric $\|\cdot\|_L$ has a natural extension (the ``Lear extension'') as a $\qq$-line bundle with a psh (orbi)metric over the orbifold $N_{g,1}^{\mathrm{mDV}}$ (see \cite[Section 6.1]{chern-weil});
\item the Hodge line bundle $\omega$ together with its smooth Hodge metric $\|\cdot\|_{\mathrm{Hdg}}$ has a natural extension (the ``Mumford extension'') as a line bundle with good and psh (orbi)metric over the orbifold $N_g^{\mathrm{Del}}$ (see \cite[Lemmas 7.18 and 7.19]{hodge} for the smooth case and \cite{hodge-orbifold} for the extension to orbifolds).
\end{itemize}
Hence, invoking Theorem \ref{thm:metric_isom} we conclude that 
\begin{itemize}
\item the line bundle $\oo(\Theta)$ together with its smooth hermitian metric $\|\cdot\|_\Theta$   has a natural extension as a $\qq$-line bundle with psh (orbi)metric over the orbifold $N_{g,1}^{\mathrm{mDV}}$.  
\end{itemize}

Now, let $X$ be an snc log scheme over $\cc$ with boundary divisor $D$ and complement $U = X \setminus |D|$. Let $(M,\|\cdot\|)$ be a continuously metrised line bundle over $U$ and assume that $(M,\|\cdot\|)$ extends as a psh $\qq$-line bundle over the orbifold $X$. 
Following the constructions in \cite[Chapter~5]{chern-weil}, one defines a log Weil b-divisor
\[
\mathbb{D}(M, \| \cdot \|, s) \coloneqq \on{div}(s) - \sum_P \nu(\| \cdot \|, \eta_P),
\] 
where $s$ is a rational section of $M$, $P$ ranges over all prime divisors on all log modifications of $X$, and $\nu(\| \cdot \|, \eta_P)$ denotes the Lelong number of (any local potential of) the metric $\| \cdot \|$ at the generic point $\eta_P$ of $P$. Assuming that the log Weil b-divisor $\mathbb{D}(M, \| \cdot \|, s)$ is continuous in the sense of Definition~\ref{def:continuous_log_Weil_b}, we obtain a continuous log b-line bundle, and hence a log adelic line bundle over $X$. We denote it by $\overline{(M,\|\cdot\|)}$.

\begin{proposition} \label{prop:is_continuous} For any non-zero rational section $s$ of the universal line bundle $L$ over $N_{g,1}$, the log Weil b-divisor $\mathbb{D}(L, \| \cdot \|_L, s)$ is continuous. Also, for any non-zero rational section $t$ of the Hodge line bundle $\omega$ over $N_g$, the log Weil b-divisor $\mathbb{D}(\omega, \| \cdot \|_{\mathrm{Hdg}}, t)$ is continuous.  
\end{proposition}
\begin{proof}  In \cite[Theorem~1.1]{bghdj_sing} the asymptotic behavior of the metric $\|\cdot\|_L$ is given. It follows from part 3 of this theorem that the log Weil b-divisor $\mathbb{D}(L, \| \cdot \|_L, s)$ has continuous conical part over any closed cone in the mixed Delaunay--Voronoi decomposition. Since the latter has a closed covering by closed cones it follows that the log Weil b-divisor $\mathbb{D}(L, \| \cdot \|_L, s)$ has continuous conical part everywhere. As to the second statement, we claim that the incarnation of $\mathbb{D}(\omega, \| \cdot \|_{\mathrm{Hdg}}, t)$ on any smooth toroidal compactification of $N_g$ is given by the canonical extension of $\omega$ defined by Mumford \cite{mu}. Indeed, the latter is uniquely determined by the property that the norms of local generating sections have at most logarithmic growth in any local coordinate system. We see that $\mathbb{D}(\omega, \| \cdot \|_{\mathrm{Hdg}}, t)$ lies in the image of $\on{Div}(N_g)$ in $\on{log-W-b-Div}(N_g^{\mathrm{Del}})$.
\end{proof}

From the above proposition we obtain natural log adelic line bundles 
\[ \overline{(L, \|\cdot\|_L)}  \qquad \textrm{and} \qquad \overline{\pi^* (\omega, \|\cdot \|_\Hdg)}  \]
over the log compactification $N_{g,1}^{\mathrm{mDV}}$. 
By Theorem~\ref{thm:metric_isom} we obtain a log adelic line bundle
\[  \overline{(\oo(\Theta), \|\cdot\|_\Theta)} \cong \overline{(L, \|\cdot\|_L)} \otimes \overline{\pi^*(\omega, \|\cdot \|_\Hdg)}^{\otimes 1/2}   \]
over $N_{g,1}^{\mathrm{mDV}}$. 

Let $G$ be the unique semiabelian scheme over $N_g^{\mathrm{Del}}$ extending the abelian scheme $N_{g,1}$ over $N_g$. 
\begin{proposition}  The log adelic line bundle $\overline{(L, \|\cdot\|_L)}$ is identified with the pure adelic line bundle $\overline{L}$.

The log adelic line bundle $\overline{(\omega, \|\cdot \|_\Hdg)}$ is identified with the adelic Hodge line bundle, i.e., the model adelic line bundle determined by the line bundle $\omega_{G/N_g^{\mathrm{Del}}} $.
\end{proposition}

\begin{proof} Recall that $\|\cdot\|_L$ can be characterised as the unique smooth hermitian metric on $L$ compatible with the rigidification and with the property that for all $m \in \zz$ the isomorphism $[m]^*L \cong L^{\otimes m^2}$ of rigidified line bundles is an isometry. This implies by functoriality that for the associated continuous log b-line bundle $\overline{(L, \|\cdot\|)}$ dictated by the metric $\|\cdot\|_L$ on $L$ we have for all $m \in \zz$ an isomorphism $[m]^* \overline{(L, \|\cdot\|)} \cong \overline{(L, \|\cdot\|)}^{\otimes m}$. 

By Proposition~\ref{prop:commutative_p_bar} we have a natural injective homomorphism 
\[ \bar{a} \colon \on{b-Pic}^{\on{cont}}(N_{g,1}^{\mathrm{mDV}}) \to  \on{Ad-Pic}_{\on{log}}(N_{g,1}^{\mathrm{mDV}}). \]
If we write $\overline{M}$ for the image of the continuous log b-line bundle $\overline{(L, \|\cdot\|)}$ we have for all $m \in \zz$ an isomorphism $[m]^* \overline{M} \cong \overline{M}^{\otimes m}$. The adelic line bundle~$\overline{M}$ extends the line bundle~$L$. Theorem~\ref{thm:YZ_invariant_N} gives that $\overline{M}$ equals Yuan--Zhang's pure adelic line bundle $\overline{L}$. 

Finally, we recall from the proof of Proposition~\ref{prop:is_continuous} that the incarnation of the log Weil b-divisor $\mathbb{D}(\omega, \| \cdot \|_{\mathrm{Hdg}}, t)$ on any smooth toroidal compactification of $N_g$ is given by the canonical extension of $\omega$ defined by Mumford \cite{mu}. As follows from \cite[p.~225]{fc}, Mumford's canonical extension over $N_g^{\mathrm{Del}}$ equals $\omega_{G/N_g^{\mathrm{Del}}}$. This gives the second statement.
\end{proof}
Applying Theorem~\ref{thm:adelic_key_restate} now immediately gives the following.
\begin{corollary} \label{cor:degeneration_theta} The log adelic line bundle $ \overline{(\oo(\Theta), \|\cdot\|_\Theta)}$ is equal to $\oo(\Theta^{\mathrm{pure}})$. 
\end{corollary}
Corollary \ref{cor:degeneration_theta} is really a statement about the degeneration behavior of the function $\|\vartheta\|$ from \eqref{eqn:Riemann_theta_norm} on all log smooth compactifications of $N_{g,1}$ over $\cc$ corresponding to representable subdivisions of $\widetilde{\Sigma}^{\textrm{mDV}}$.

\subsection{N\'eron--Tate height of a point}

Our results imply (and were in fact partially inspired by finding) a ``universal'' formula for the N\'eron--Tate height of a point on a principally polarised abelian variety defined over the field of algebraic numbers. The aim of this section is to provide some details. The formula is given in Theorem~\ref{univ_NT_formula} below.

Let $K$ be a number field, and write $S=\Spec O_K$. Let $A$ be an abelian variety over $K$, and assume $A$ has split semiabelian reduction over $S$. Let $L$ be a symmetric ample rigidified line bundle over $A$ defining a principal polarisation~$\lambda$. Let $D$ be the effective symmetric theta divisor on $A$ determined by $L$. Write $g=\dim A$. We assume as usual that $g \ge 1$.

For every $x \in A(K)$ the \emph{N\'eron--Tate height} of $x$ with respect to~$L$ is given by the Arakelov degree
\[ \h_L(x) =  \frac{1}{[K:\qq]} \, \widehat{\deg} \, x^* \overline{L} . \]
Here $\overline{L}$ is the adelic line bundle over $A/K$ given by $L$ together with its canonical pure adelic structure. We recall that $\h_L(x)$ only depends on the polarisation~$\lambda$.  

Let $|S|$ denote the set of closed points of $S$. For every $v \in |S|$ we let $\Sk(A_v)$ denote the canonical skeleton of the Berkovich analytification $A_v^\an$ of $A$ at $v$ (see Section~\ref{sec:Berk_anal}). This is a real torus. We denote by $\val_v \colon A_v^\an \to \Sk(A_v)$ the associated tropicalisation homomorphism. We have canonical free finitely generated abelian groups $X_v$ and $Y_v$ of rank~$\le g$ at~$v$ and an identification of real tori $\Sk(A_v) \cong X_{v,\rr}^\lor/Y_v$. 

Let $\theta_{0,v}^{\mathrm{inv}} \colon \Sk(A_v) \to \bb R$ be the invariant tropical Riemann theta function at~$v$ with vanishing characteristic as  in Example~\ref{ex:inv_rho_vanish}, and let $\kappa_v \in \Sk(A_v)$ be the tropical theta characteristic of the pair $(A,L)$ at $v$ (see Section~\ref{sec:AN_cle}).

Let $\hh_g$ denote Siegel's upper half space in degree~$g$.
For every complex embedding $\sigma \colon K \hookrightarrow  \cc$ we fix an element $\tau_\sigma \in \hh_g$ and an isomorphism of complex ppav $\alpha_\sigma \colon (A_\sigma,\lambda_\sigma) \isom (A_{\tau_\sigma},\mu_{\tau_\sigma})$. We let $\kappa_\sigma \in A_{\tau_\sigma}$ denote the theta characteristic of $L_\sigma$ determined by $\alpha_\sigma$, as discussed in Section~\ref{sec:explicit}.

Let $\ca N$ denote the N\'eron model of $A$ over $S$. Given $x \in A(K)$ we write $\overline{x}_{\ca N}$ for the extension of $x$ over $\ca N$. We write $\ca D_{\ca N}$ for the flat extension (or ``thickening'') of the effective divisor $D$ over $\ca N$. As $\ca N$ is regular we have that $\ca D_{\ca N}$ is an effective relative Cartier divisor on $\ca N$.  For every $v \in |S|$ and every $x \in A(K) \setminus  |D|$ we finally write $i_v(x,\ca N,D)$ for the intersection multiplicity $\ord_v \overline{x}_{\ca N}^*\ca D_{\ca N}$. 

Similarly, for every $v \in |S|$ we let $(\ca P_v, \ca D_{v})$ denote the Alexeev--Nakamura model of $(A,L)$ over the completion of~$O_K$ at~$v$ (see Section~\ref{sec:AN_cle}). Recall that $\ca D_{v}$ is a relative effective Cartier divisor on $\ca P_v$. Let $x \in A(K)$. As the model $\pp_v $ is proper, we have a unique section $\overline{x}_{v}$ of $\pp_v$ extending~$x_v$. For $x \in A(K) \setminus |D|$ we denote by $i(x,\ca P_v,D)$ the intersection multiplicity $\ord_v \overline{x}_{v}^*\ca D_{v}$. 

\begin{proposition} \label{prop:inters_mult} Assume that $x \in A(K) \setminus |D|$. We have an equality of intersection multiplicities $i(x,\ca P_v,D) = i_v(x,\ca N,D)$.
\end{proposition}
\begin{proof} This follows from the proof of \cite[Theorem~6.1]{djs_canonical}.
\end{proof} 

The \emph{Faltings height} $\h_F(A)$ of $A$ is given as the Arakelov degree
\[ \h_F(A) = \frac{1}{[K:\qq]} \, \widehat{\deg} \, \overline{\omega}_{A/S}, \]
with  $\overline{\omega}_{A/S}$ the adelic Hodge line bundle over $S$. For $v \in |S|$ we denote by $Nv$ the cardinality of the residue field at $v$. The following is our ``universal'' formula for the N\'eron--Tate height of a point on a principally polarised abelian variety.

\begin{theorem} \label{univ_NT_formula} Assume that $x \in A(K) \setminus |D|$. The identity
\[ \begin{split}   \h_L(x)  & =   -\frac{1}{2} \h_F(A) + \frac{1}{[K:\qq]} \left( \sum_{v \in |S|} \left( i_v(x,\ca N,D) + \theta^{\mathrm{inv}}_{0,v}(\val_v(x_v)+\kappa_v) \right)\log Nv + \right. \\
& \hspace{1cm}  \left. + \sum_{\sigma \colon K \hookrightarrow \cc} - \log \left( (2\pi)^{g/2}  \|\vartheta\|(\alpha_\sigma(x_\sigma)+\kappa_\sigma, \tau_\sigma) \right) \right) \end{split} \]
holds. 
\end{theorem}
Note that the non-archimedean and archimedean theta characteristics play similar roles in the above formula.
\begin{proof}[Proof of Theorem~\ref{univ_NT_formula}]  We take the isomorphism of hermitian adelic line bundles from Corollary~\ref{cor:global_adelic_isom}, pull it back along the moduli point $(A,L,x)$ of $N_{g,1}$, take Arakelov degrees on both sides, and divide by eight. The result is that
\[ \begin{split}   \h_L(x)  & =   -\frac{1}{2} \h_F(A) + \frac{1}{[K:\qq]} \left( \sum_{v \in |S|} \left( i(x,\ca P_v,D) + \theta^{\mathrm{inv}}_{\rho_v}(\val_v(x_v)) \right)\log Nv + \right. \\
& \hspace{1cm}  \left. \sum_{\sigma \colon K \hookrightarrow \cc} - \log \left(   \|1_\Theta\|'(A_\sigma,L_\sigma,x_\sigma) \right) \right) . \end{split} \]
Here we denote by $\rho_v$ any element of the lattice $X_v$ such that $\rho_v/2$ represents the tropical theta characteristic of~$L_v$.

We use Proposition~\ref{prop:inters_mult} to replace $i(x,\ca P_v,D)$ by  $i_v(x,\ca N,D)$. 
We use Proposition~\ref{prop:trop_theta_new} to replace $\theta^{\mathrm{inv}}_{\rho_v}(\val_v(x_v))$ by $\theta^{\mathrm{inv}}_{0,v}(\val_v(x_v)+\kappa_v) $. We use Proposition~\ref{prop:archimedean_two_torsion} to replace $\|1_\Theta\|'(A_\sigma,L_\sigma,x_\sigma) $ by $(2\pi)^{g/2}  \|\vartheta\|(\alpha_\sigma(x_\sigma)+\kappa_\sigma, \tau_\sigma)$. The formula from the theorem follows.
\end{proof}

\begin{example} We illustrate Theorem~\ref{univ_NT_formula} by considering the case of elliptic curves. Let $E$ be an elliptic curve over the number field $K$. Assume that $E$ has split semistable reduction over $K$. We take the origin $O$ as theta divisor on $E$. Write $L$ for the associated rigidified line bundle. Write $S=\Spec O_K$ and let $|S|$ be the set of closed points of $S$. Let $x \in E(K) \setminus \{O\}$.

Let $v \in |S|$. We put
\[ \widehat{\lambda}_v(x) = i_v(x,\ca N,O) \]
if $E$ has good reduction at $v$. Assume that $E$ has bad reduction at~$v$. We let $\Delta_v$ denote the minimal discriminant at~$v$, and we set $\ell_v = \ord_v \Delta_v$. Let $\Phi_v$ denote the group of components of the special fiber $\ca N_v$ of the N\'eron model $\ca N$ of $E$ at $v$. We fix a group isomorphism $\Phi_v \cong \zz/\ell_v\zz$, and we label the components of $\ca N_v$ accordingly as $\{0,1,\ldots,\ell_v-1\}$. Let $\mathrm{sp}_v(x) \in \{ 0, 1, 2, \ldots, \ell_v-1 \}$ be the label of the component of $\ca N_v$ where our given rational point $x$ specialises. We have $\Sk(E_v) = \rr/\ell_v \zz$ and $\mathrm{sp}_v(x) \bmod \ell_v\zz = \val_v(x_v)$ in $\Sk(E_v)$. Let $B_2(T)= T^2-T+1/6$ be the second Bernoulli polynomial.  Still assuming that $E$ has bad reduction at $v$, we put
\begin{equation} \label{eqn:pre_lambda_finite}
 \widehat{\lambda}_v(x) = i_v(x,\ca N,O) +  \frac{\ell_v}{2} B_2\left( \frac{\mathrm{sp}_v(x) }{\ell_v}\right) .
 \end{equation}
Now let $\sigma \colon K \hookrightarrow \cc$ be a complex embedding of $K$, and fix an isomorphism $\alpha_\sigma \colon E_\sigma \isom A_{\tau_\sigma}$ of complex elliptic curves with $\tau_\sigma \in \hh_1$. Let $\Delta=\Delta(\tau)$ be the discriminant modular form
\[ \Delta = q \prod_{n=1}^\infty (1-q^n)^{24} , \quad q = \exp(2\pi i \tau) \]
on the Siegel upper half plane $\hh_1$.  We put
\begin{equation} \label{eqn:pre_lambda_complex} \widehat{\lambda}_\sigma(x) =  -\log  \|\vartheta\|(\alpha_\sigma(x_\sigma) + (1+\tau_\sigma)/2,\tau_\sigma) +   \frac{1}{24} \log \left( (\Im \tau_\sigma)^6 |\Delta(\tau_\sigma)| \right) , 
\end{equation}

where $\|\vartheta\|$ is the normalised Riemann theta function \eqref{eqn:Riemann_theta_norm}.

A result due to Tate (see~\cite[Theorem~VI.4.2]{sil_adv} or~\cite[Section~6.5]{serre_MW}) gives us that
\begin{equation} \label{eqn:Tate} \h_L(x) = \frac{1}{[K:\qq]}  \left( \sum_{v \in |S|} \widehat{\lambda}_v(x) \log Nv + 
\sum_{\sigma \colon K \hookrightarrow \cc} \widehat{\lambda}_\sigma(x) \right) . 
\end{equation}
We will now see how this can be rewritten to give the formula in Theorem~\ref{univ_NT_formula}. 

First of all let $v \in |S|$ and assume $E$ has bad reduction at $v$. We have $ X_v \cong \zz $ with bilinear form $(x,y) \mapsto \ell_v xy$ for $x, y \in \zz $, and $\theta_0^{\on{inv}}(T) = \frac{1}{2\ell_v}T^2$ for $T \in [-\ell_v/2,\ell_v/2]$. Further, the domains of linearity of any tropical theta function associated to $L_v$ are given by the intervals $\ell_v[m,m+1]$ with $m \in \zz$. This implies that the tropical theta characteristic $\kappa_v$ associated to $L_v$ is the unique non-trivial $2$-torsion point $\ell_v/2 \bmod \ell_v\zz$ of the circle $\Sk(E_v) = \rr/\ell_v \zz$. 

Still assuming that $E$ has bad reduction at $v$, we find from \eqref{eqn:pre_lambda_finite} the equalities
\begin{equation} \label{eqn:finite_lambda} \begin{split} \widehat{\lambda}_v(x) & = i_v(x,\ca N,O) + \frac{\ell_v}{2} \left( \left( \frac{\mathrm{sp}_v(x)}{\ell_v} \right)^2 - \frac{\mathrm{sp}_v(x)}{\ell_v} \right) + \frac{1}{12}\ell_v \\
& = i_v(x,\ca N,O) + \frac{1}{2\ell_v} \left( \mathrm{sp}_v(x) - \frac{\ell_v}{2} \right)^2 - \frac{1}{24}\ell_v  \\
& =  i_v(x,\ca N,O) + \theta^{\mathrm{inv}}_{0,v}(\val_v(x_v)+\kappa_v) - \frac{1}{24}\ell_v. \end{split}
 \end{equation}
Next let $\sigma \colon K \hookrightarrow \cc$ be a complex embedding of $K$. The function $\|\vartheta\|$ vanishes on $(1+\tau_\sigma)/2$, hence we have $\kappa_\sigma = (1+\tau_\sigma)/2$ for the theta characteristic associated to $(L_\sigma,\alpha_\sigma)$. We obtain from \eqref{eqn:pre_lambda_complex} the equality
\begin{equation} \label{eqn:complex_lambda} \widehat{\lambda}_\sigma(x) =  -\log  \|\vartheta\|(\alpha_\sigma(x) + \kappa_\sigma,\tau_\sigma) +   \frac{1}{24} \log \left(  (\Im \tau_\sigma)^6 |\Delta(\tau_\sigma)| \right)  . 
\end{equation}
Now by \cite[Theorem~7]{fa} and  \cite[Proposition~1.1]{sil} we have the formula
\begin{equation} \label{eqn:Faltings_height} \h_F(E) = \frac{1}{[K:\qq]} \frac{1}{12}  \left( \sum_{v \in |S|} \ell_v \log Nv - \sum_{\sigma \colon K \hookrightarrow \cc} \log \left( (2\pi)^{12}  (\Im \tau_\sigma)^6 |\Delta(\tau_\sigma)| \right) \right) . 
\end{equation}
Putting together \eqref{eqn:Tate}, \eqref{eqn:finite_lambda}, \eqref{eqn:complex_lambda} and \eqref{eqn:Faltings_height} we obtain the formula in Theorem~\ref{univ_NT_formula} in the special case of elliptic curves.
\end{example}

\begin{remark} The assumption in Theorem~\ref{univ_NT_formula} that $x \notin |D|$  is easy to avoid in practice. Indeed, the linear system $\frak g$ on $A$ corresponding to twice the principal polarisation $\lambda$ is base-point free. The divisors $2D$ where $D$ runs through all symmetric effective theta divisors for $\lambda$ are all in $\frak g$ and span it. It follows that, given $x \in A(\overline{K})$, there is at least one symmetric theta divisor $D$  associated to $\lambda$, and defined over $\overline{K}$, with the property that $x \notin |D|$.
\end{remark}

\printbibliography

\end{document}